\documentclass[aos]{imsart}    

\RequirePackage{amsthm,amsmath,amsfonts,amssymb}
\RequirePackage[numbers]{natbib}
\RequirePackage[OT1]{fontenc}
\RequirePackage[numbers]{natbib}
\RequirePackage[colorlinks,citecolor=blue,urlcolor=blue]{hyperref}
\usepackage{graphicx}
\usepackage{xcolor}
\usepackage{comment}
\usepackage[shortlabels]{enumitem}
\usepackage{mathrsfs}
\usepackage{algorithm}
\usepackage{algpseudocode}
\usepackage{bibunits}
\usepackage{definitions}

\startlocaldefs
\numberwithin{equation}{section}
\theoremstyle{plain}
\newtheorem{theorem}{Theorem}[section]
\newtheorem{lemma}{Lemma}[section]
\newtheorem{corollary}{Corollary}[section]
\newtheorem{definition}{Definition}[section]
\theoremstyle{definition}
\newtheorem{remark}{Remark}[section]
\newtheorem{example}{Example}[section]
\endlocaldefs

\begin{document}

\begin{bibunit}[imsart-number]

\begin{frontmatter}
\title{The Cost of Adaptation under Differential Privacy: Optimal Adaptive Federated Density Estimation}
\runtitle{The Cost of Adaptation under Differential Privacy}

\begin{aug}
\author[A]{\fnms{T. Tony} \snm{Cai}\ead[label=e1]{tcai@wharton.upenn.edu}},
\author[B]{\fnms{Abhinav} \snm{Chakraborty}\ead[label=e2]{ac4662@columbia.edu}}
\and
\author[C]{\fnms{Lasse} \snm{Vuursteen}\ead[label=e3]{lv121@duke.edu}}

\address[A]{Department of Statistics and Data Science, \\
\printead[presep={The Wharton School, University of Pennsylvania,\ }]{e1}}

\address[B]{Department of Statistics, \\
\printead[presep={Columbia University,\ }]{e2}}

\address[C]{Department of Statistical Science,\\
\printead[presep={Duke University,\ }]{e3}}


\runauthor{T. T. Cai, A. Chakraborty and L. Vuursteen}
\end{aug}
  
\begin{abstract}
    Privacy-preserving data analysis has become a central challenge in modern statistics. At the same time, a long-standing goal in statistics is the development of adaptive procedures\textemdash methods that achieve near-optimal performance across diverse function classes without prior knowledge of underlying smoothness or complexity. While adaptation is often achievable at no extra cost in the classical non-private setting, this naturally raises a fundamental question: to what extent is adaptation still possible under privacy constraints?

    We address this question in the context of density estimation under federated differential privacy (FDP), a framework that encompasses both central and local DP models. We establish sharp results that characterize the cost of adaptation under FDP for both global and pointwise estimation, revealing fundamental differences from the non-private case. We then propose an adaptive FDP estimator that achieves explicit performance guarantees by introducing a new noise mechanism, enabling one-shot adaptation via post-processing. This approach strictly improves upon existing adaptive DP methods. Finally, we develop new lower bound techniques that capture the limits of adaptive inference under privacy and may be of independent interest beyond this problem.
     
   Our findings reveal a sharp contrast between private and non-private settings. For global estimation, where adaptation can be achieved for free in the classical non-private setting, we prove that under FDP an intrinsic adaptation cost is unavoidable. For pointwise estimation, where a logarithmic penalty is already known to arise in the non-private setting, we show that FDP introduces an additional logarithmic factor, thereby compounding the cost of adaptation. Taken together, these results provide the first rigorous characterization of the adaptive privacy-accuracy trade-off.
\end{abstract}

\begin{keyword}[class=MSC]
    \kwd[Primary ]{62G07}
    \kwd{62G20}
    \kwd[; secondary ]{62C20}
    \end{keyword}
    
    \begin{keyword}
    \kwd{Adaptation}
    \kwd{Differential Privacy}
    \kwd{Density Estimation}
    \kwd{Minimax Rates}
    \end{keyword}

\end{frontmatter}


\section{Introduction}


Privacy protection has become a critical concern in modern data analysis. Differential privacy (DP), introduced by \cite{dwork2006calibrating}, provides a rigorous mathematical framework that guarantees statistical analyses can be published without compromising the privacy of individual data subjects. Many differentially private statistical methods have since been developed, see for example \cite{dwork2010differential}, \cite{abadi2016deep}, and \cite{dwork2017exposed}. Among the various privacy protection frameworks available today, DP has emerged as particularly popular both for its theoretical foundations and its many applications across academic disciplines and industrial applications, see for example \cite{ficek2021differential, pan2024differential, jiang2021differential}, and is finding applications within major technology companies such as Google, Amazon, Microsoft, and Apple, and governmental institutions such as the U.S. Census Bureau. Estimation and inference under differential privacy become a major focus in statistics and machine learning. Rigorous studies of privacy-utility trade-offs have established fundamental performance limits in diverse statistical settings. These include work on estimation \cite{wasserman2010privacy,hall2013differential,duchi2013local, duchi2018minimax,steinberger2020geometrizing, cai2021cost, li2024federated, cai2024optimal, xue2024optimalestimationprivatedistributed}, testing \cite{pmlr-v89-acharya19b, berrett2020locally, narayanan2022tight, pmlr-v178-narayanan22a}, classification \cite{pmlr-v19-chaudhuri11a} and uncertainty quantification \cite{karwa2017finite,steinberger2024efficiencylocaldifferentialprivacy}.

For many complicated statistical problems, the performance of methods depends on unknown hyperparameters, which need to be tuned to unknown underlying regularity parameters. This is known as the problem of \emph{adaptation}. For example, high-dimensional regression procedures typically require tuning to the sparsity level, while in nonparametric regression or density estimation, the smoothness of the underlying function is often unknown and methods must adapt to this unknown regularity. Such adaptation problems have been extensively studied and are well understood in the classical non-private setting, see for example \cite{lepski1991problem, donoho1995adaptingvia, spokoiny_adaptive_1996, lepski1997optimalpointwise, tsybakov1998pointwise,cai2005adaptiveDifferentPerformanceMeasures, bickel2009simultaneous}.

In the context of differential privacy, however, adaptation poses a significant challenge and remains poorly understood. While adaptive DP procedures have been developed for various statistical problems -— including density estimation \cite{butucea_LDP_adaptation,butucea2023interactive,kroll_density_at_a_point_LDP,RandrianarisoaSteinbergerSzabo2025}, goodness-of-fit testing \cite{lam2022minimax, cai2024privateTesting}, classification \cite{auddy2024minimax}, hyperparameter tuning for stochastic gradient descent \cite{papernot2022hyperparametertuningrenyidifferential} and adaptation to sparsity in linear regression \cite{ZhangNakadaZhang2024} -— all of these methods exhibit performance gaps between the optimal non-adaptive DP rate and the rate achieved by the adaptive DP procedure. To date, no specific adaptive lower bounds have been established to determine whether these performance gaps are an intrinsic consequence of adaptation under differential privacy constraints, or merely artifacts of current state-of-the-art adaptive procedures.



In this work, we address the fundamental question of adaptation under differential privacy constraints in the context of federated nonparametric density estimation. Density estimation serves as a canonical nonparametric problem with well-understood optimality theory in the non-private setting, that allows us to isolate the intrinsic cost of adaptation under privacy constraints. We consider both global risk (mean squared error) and pointwise risk, which capture fundamentally different aspects of adaptation and exhibit distinct phenomena. Moreover, density estimation has particular relevance for privacy applications, enabling synthetic data generation that preserves statistical properties while protecting individual privacy.

While most existing work on differential privacy focuses on either \emph{central} differential privacy (CDP), where a trusted curator holds all data, or \emph{local} differential privacy (LDP), where privacy is enforced at the individual level before data collection, our analysis uses the more general \emph{federated differential privacy} (FDP) framework. FDP distributes data across multiple servers with privacy guarantees enforced at each server, generalizing both CDP (a single server with multiple observations) and LDP (many servers with one observation each). This unified framework makes our results applicable across the entire spectrum of privacy frameworks.

\subsection{Our contributions and related work}

Our work addresses a critical gap in the literature by providing the first complete characterization of the fundamental cost of adaptation in differentially private estimation. While adaptation has been extensively studied in classical nonparametric statistics, the additional constraints imposed by differential privacy introduce new complexities that have remained theoretically unresolved.

The main contributions of our paper are as follows:

\begin{itemize}
    \item We establish the first sharp minimax rates for adaptive density estimation under differential privacy constraints for both global and pointwise risk, fully characterizing the fundamental cost of adaptation under privacy for both LDP and CDP settings.
    
	\item We develop a new privacy method that enables optimal adaptive estimation. This mechanism provides DP guarantees while avoiding the variance inflation that typically plagues adaptive DP methods, allowing us to achieve the theoretical limits.
    
	\item We derive novel lower bound techniques that precisely quantify the fundamental cost of adaptation under differential privacy. Our results demonstrate that this cost manifests as an unavoidable logarithmic penalty relative to non-private adaptive estimation, resolving an open question in the literature. 
\end{itemize}

Our findings reveal fundamental differences between private and non-private adaptation. In the classical setting, global risk adaptation can be achieved without cost, while our results show that differential privacy introduces an inherent penalty. For pointwise risk, we show that adaptation can incur additional costs beyond those already present in the non-private case. In particular, this closes the open question raised by \cite{butucea2023interactive} and \cite{RandrianarisoaSteinbergerSzabo2025} on whether such log-factors can be circumvented.

Existing work on adaptive density estimation under LDP includes methods based on Lepski-type procedures \cite{kroll_density_at_a_point_LDP, RandrianarisoaSteinbergerSzabo2025, schluttenhofer2022adaptivepointwisedensityestimation} and a wavelet thresholding approach \cite{butucea_LDP_adaptation}. However, these methods suffer from suboptimal rates because they either compute estimators for the worst-case regularity level or maintain multiple estimators for each regularity level considered. In both cases, the required privacy noise variance scales proportionally with the number of regularity levels, leading to rate-suboptimal performance. Our noise mechanism overcomes these limitations by carefully accounting for dependencies across different regularity levels, allowing us to achieve optimal adaptive rates with a single pass through the data while maintaining the same privacy guarantees.

Various minimax lower bound techniques in private settings have been developed using modifications of classical methods, such as Fano's inequality, Assouad's lemma, and Le Cam's method \cite{lalanne2023statistical, acharya2021differentially}, as well as van Trees-based bounds \cite{barnes2020fisher, Cai2024FL-NP-Regression}. Alternative approaches, including fingerprinting methods, tracing attack, and score attack, have also been explored in private learning theory \cite{dwork2015robust, cai2023score, Kamath2024BiasAccuracyPrivacy}. Despite these advances, none of these methods have successfully characterized the exact cost of adaptation under differential privacy constraints. To the best of our knowledge, our work presents the first lower bound techniques that precisely quantifies the adaptation cost in the differentially private setting, establishing a fundamental limit on adaptive private estimation. While our optimal adaptive procedure is one-shot, our lower bound techniques are more broadly applicable to sequential protocols, showing that the cost of adaptation is unavoidable even in such interactive settings.

\subsection{Problem formulation}\label{sec:problem_formulation}

We consider a federated setting with $m$ servers, where server $j \in \{1,\ldots,m\}$ holds $n$ i.i.d. observations $X^{(j)} = (X^{(j)}_1, \ldots, X^{(j)}_n)$ drawn from an unknown density $f$ on $[0,1]$, yielding $N = mn$ total samples. Each server computes a local transcript $T^{(j)}$ based on its data, and these transcripts are aggregated centrally to form estimators $\hat{f}$ or $\hat{T}$. We consider two estimation objectives: global risk $\mathbb{E}_f \|\hat{f} - f\|_2^2$ for estimating the entire density function, and pointwise risk $\mathbb{E}_f[(\hat{T} - f(t_0))^2]$ for estimating the density at a specific point $t_0 \in (0,1)$.

We shall assume that the true density $f$ belongs to a Besov space $\mathcal{B}^{\alpha}_{p,q}$, which provides a flexible framework capturing diverse smoothness classes including Sobolev and H\"older spaces. Loosely speaking, $\mathcal{B}^{\alpha}_{p,q}$ contains functions with $\alpha$ derivatives, bounded in $L_p$-space, where the parameters $(\alpha,p,q)$ together control different aspects of regularity. We give a formal definition of Besov spaces in Section \ref{sec:besov_definitions}. 

Crucially, optimal estimation rates depend on these smoothness parameters, but in practice they are unknown. The adaptive estimation challenge is to develop data driven procedures achieve optimal performance across a range of smoothness parameters, without requiring prior knowledge of $(\alpha,p,q)$. 

Privacy constraints are formalized through {federated differential privacy}: Each server's transcript must satisfy differential privacy with respect to changes in its local dataset.

\begin{definition}\label{def:federated_differential_privacy}
A distributed estimation protocol is $(\varepsilon,\delta)$-\emph{FDP} if, for each server $j=1,\ldots,m$, its local transcript $T^{(j)}$ satisfies
\begin{equation*}
	\mathbb{P}\left(T^{(j)} \in A \mid X^{(j)} = x \right) \leq e^\varepsilon \mathbb{P}\left(T^{(j)} \in A \mid X^{(j)} = x' \right) + \delta,
\end{equation*}
for any pair of local datasets $x,x' \in [0,1]^n$ differing in at most one observation, and for any measurable subset $A$ in the range of $T^{(j)}$, where the probability is over the randomness of the local mechanism.
\end{definition}

The parameters $\varepsilon > 0$ and $\delta \in [0,1)$ control privacy strength, with smaller values providing stronger guarantees. The FDP framework encompasses CDP (which is recovered when $m=1$) and LDP (when $n=1$) as special cases. Let $\mathcal{F}^{\varepsilon,\delta}$ and $\mathcal{T}^{\varepsilon,\delta}$ denote the classes of all $(\varepsilon,\delta)$-FDP protocols producing estimators in $L_2[0,1]$ and $\mathbb{R}$, respectively.  The FDP framework for density estimation is illustrated in Figure \ref{fig: FDP-Illustration}.

\begin{figure}[htb]
	\centering
	\includegraphics[width=0.75\textwidth]{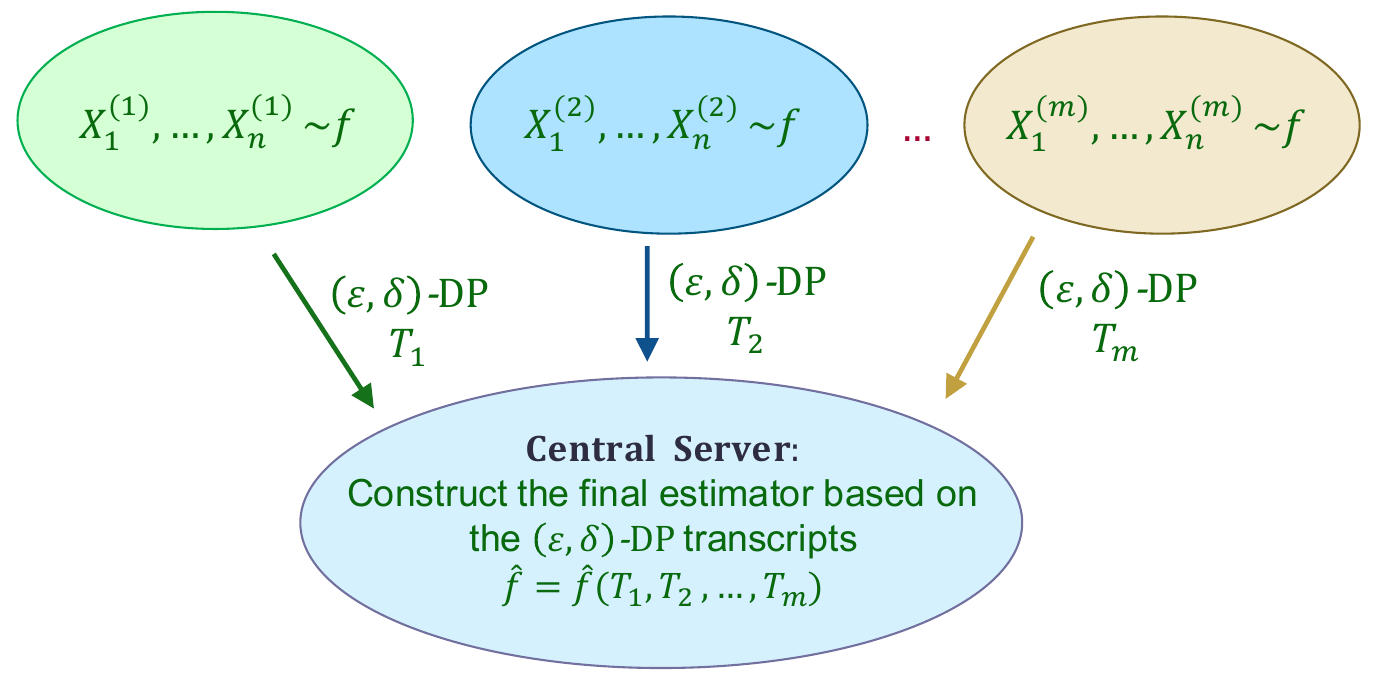} 
	\caption{An illustration of federated density estimation under DP.}
	\label{fig: FDP-Illustration}
\end{figure}

To understand the challenge of adaptation under privacy constraints, it is instructive to first consider the non-adaptive case where smoothness parameters are assumed to be known. When $(\alpha,p,q)$ are given, \cite{Cai2024FL-NP-Regression} established the minimax estimation rates for both risk types:
\begin{align}
	\inf_{\hat{f} \in \mathcal{F}^{\varepsilon,\delta}} \sup_{f \in \bar{\mathcal{B}}^{\alpha, R}_{p,q}} \mathbb{E}_f \|\hat{f} - f\|_2^2 &\asymp N^{-\frac{2\alpha}{2\alpha + 1}} + \left(m n^2\varepsilon^2\right)^{-\frac{2\alpha}{2\alpha + 2}} =: \rho^2_{\textsc{MSE,non-adaptive}}(\alpha), \label{eq:nonadaptive-rates} \\
	\inf_{\hat{T} \in \mathcal{T}^{\varepsilon,\delta}} \sup_{f \in \bar{\mathcal{B}}^{\alpha,R}_{p,q}} \mathbb{E}_f[(\hat{T} - f(t_0))^2] &\asymp N^{-\frac{2\nu}{2\nu + 1}} + \left(m n^2\varepsilon^2\right)^{-\frac{2\nu}{2\nu + 2}} =: \rho^2_{\textsc{point,non-adaptive}}(\alpha,p), \nonumber
\end{align}
where $\nu = \alpha - 1/p$ and $\bar{\mathcal{B}}^{\alpha,R}_{p,q}$ denotes the Besov ball of radius $R > 0$ intersected with the space of probability densities on $[0,1]$.


For both risk types, the rate consists of a non-private term and a privacy cost that dominates when $\varepsilon$ is small. The term representing the non-private rate -- which is of the form ${N}^{-\frac{2r}{2r + 1}}$ -- is the rate achievable when privacy is not a concern (e.g., $\varepsilon = \infty$). The second term --  which is of the form $({m n^2 \varepsilon^2})^{-\frac{2r}{2r + 2}}$ -- represents the additional cost of adaptation due to privacy constraints, which becomes substantial whenever $\varepsilon \ll n^{-\frac{r}{2r + 1}} m^{\frac{1}{2(2r + 1)}}$; where the privacy cost begins to dominate the classical statistical estimation error. Notably, when a fixed number of data points $N = mn$ is distributed across sufficiently many servers $m$, the privacy cost always dominates.

The rates depend critically on Besov parameters $\alpha$ and $p$: smaller $\alpha$ (less smooth densities) makes estimation harder. The parameter $p$ affects the pointwise problem through $\nu = \alpha - 1/p$. The pointwise problem is potentially more difficult than the global problem for small values of $p$. For example, in the the Sobolev space $\cB^{\alpha}_{2,2}$, where $\nu = \alpha - 1/2$. For H\"older smooth densities -- corresponding to $\cB^{\alpha}_{\infty,\infty}$ -- $\nu = \alpha$ and the problems of global and pointwise estimation are equally difficult in terms of $\alpha$.

Since these rates and optimal procedures of \cite{Cai2024FL-NP-Regression} depend on the (in practice) unknown parameters $(\alpha,p,q)$, adaptation becomes essential. This leads to a fundamental question: \emph{When $(\alpha,p,q)$ are unknown, to what extent can the minimax rates in \eqref{eq:nonadaptive-rates} still be attained?}

Our objective is to characterize the optimal rates under federated privacy constraints in the adaptive setting, where the regularity of the underlying function is unknown. Formally, given a collection $\mathcal{I}$ of smoothness parameter values, we seek to understand for which functions $(\alpha,p) \mapsto \rho^2_{\alpha,p;m,n,\varepsilon}$ the following adaptive minimax risks can be attained:
\begin{equation*}
	\inf_{\hat{f} \in \mathcal{F}^{\varepsilon,\delta}} \sup_{(\alpha,p,q) \in \mathcal{I}} \sup_{f \in \bar{\mathcal{B}}^{\alpha,R}_{p,q}}  \frac{\mathbb{E}_f \|\hat{f} - f\|_2^2}{\rho_{\textsc{MSE}}^2(\alpha;m,n,\varepsilon)} \quad \text{and} \quad \inf_{\hat{T} \in \mathcal{T}^{\varepsilon,\delta}} \sup_{(\alpha,p,q) \in \mathcal{I}} \sup_{f \in \bar{\mathcal{B}}^{\alpha,R}_{p,q}} \frac{\mathbb{E}_f[(\hat{T} - f(t_0))^2]}{\rho^2_{\textsc{point}}( \alpha,p ;m,n,\varepsilon)}.
\end{equation*}
That is, we seek estimators that automatically adapt to unknown smoothness, achieving (near-)optimal performance across all parameter values in $\mathcal{I}$. In the non-private setting, adaptation to unknown smoothness incurs no additional cost for global risk, but pointwise estimation pays a logarithmic penalty: the optimal adaptive rate becomes $(\tfrac{\log N}{N})^{\frac{2\nu}{2\nu + 1}}$, see e.g. \cite{cai1999adaptive,cai2003pointwiseBesov}. Our central question is whether differential privacy fundamentally alters these adaptation costs.


\subsection{Organization of the paper}

The remainder of this paper is organized as follows.  Section \ref{sec:main_results} presents our main theoretical results, including sharp minimax rates for both global and pointwise risks under adaptive differential privacy constraints. Section \ref{sec:adaptive_density_estimation_at_a_point} develops our optimally adaptive estimation procedures, introducing the novel noise mechanism and wavelet thresholding estimator that achieve the theoretical limits. Section \ref{sec:lower-bound} derives the fundamental lower bounds that characterize the unavoidable cost of adaptation under differential privacy. Section \ref{sec:discussion} concludes with a discussion of our findings and directions for future research. Proofs and auxiliary results are collected in the Supplementary Material \cite{cai2025supplementary}.

\subsection{Notation, definitions and assumptions}

For positive sequences $a_k$, $b_k$, we write $a_k \lesssim b_k$ if $a_k \leq Cb_k$ for some universal constant $C$, and $a_k \asymp b_k$ if $a_k \lesssim b_k$ and $b_k \lesssim a_k$. We denote $a_k \ll b_k$ when $a_k/b_k = o(1)$. We use $a \vee b$ and $a \wedge b$ for the maximum and minimum of $a$ and $b$, respectively. The $\ell_p$-norm of $v \in \mathbb{R}^d$ is $\|v\|_p$. Throughout, $c$ and $C$ are universal constants that may change from line to line.

\section{Main results}\label{sec:main_results}

Our central finding is that differential privacy fundamentally alters the cost of adaptation in nonparametric estimation. While classical statistics shows that adaptation to unknown smoothness can be achieved without penalty for global risk, we prove that differential privacy introduces an unavoidable logarithmic deterioration in the privacy-dominated regime.

This is formalized in our first main result, which characterizes the global risk when smoothness parameters are unknown. Recall that $N=mn$ is the total number of data points.

\begin{theorem}\label{thm:global-FDP-rate}
	Let $\delta \ll \varepsilon^2 / m$ and define
	\begin{equation}\label{eq:adaptive-global-rate}
		\rho_{\textsc{MSE}}^2(\alpha) \equiv \rho_{\textsc{MSE}}^2(\alpha,m,n,\varepsilon) = N^{-\frac{2\alpha}{2\alpha + 1}} + \left( \frac{m n^2 \varepsilon^2}{\log (N)} \right)^{-\frac{2\alpha}{2\alpha + 2}}.
	\end{equation}
For any $p \geq 2$, $q \geq 1$, $\alpha_{\max} > \alpha_{\min} > 1/2$, it holds that
\begin{equation}\label{eq:global-lb-to-show-thm}
	\inf_{\hat{f} \in \cF^{\varepsilon,\delta}} \, \sup_{ \alpha \in (\alpha_{\min}, \alpha_{\max})} \, \sup_{f \in \bar{\cB}_{pq}^{\alpha,R}} \, \E_f \| \hat{f} - f \|_2^2 \; \rho_{\textsc{MSE}}^{-2}(\alpha) \asymp 1.
\end{equation}
\end{theorem}

The theorem establishes that $\rho_{\textsc{MSE}}^2(\alpha)$ in \eqref{eq:adaptive-global-rate} is the sharp adaptive rate. Comparing to the non-adaptive rate $N^{-2\alpha/(2\alpha+1)} + (mn^2\varepsilon^2)^{-2\alpha/(2\alpha+2)}$, adaptation introduces the logarithmic factor $\log N$ in the denominator of the privacy term, making the privacy cost strictly worse. This penalty is unavoidable: no estimator can achieve better than constant normalized risk across smoothness levels in any interval $(\alpha_{\min}, \alpha_{\max})$.

The result follows from matching upper and lower bounds. The lower bound is established in Section \ref{sec:lower-bound}, while our constructive upper bound (Theorem \ref{thm:upper-bound-global-risk} in Section \ref{sec:adaptive_density_estimation_at_a_point}) provides an explicit $(\varepsilon,0)$-FDP estimator that achieves the rate in \eqref{eq:adaptive-global-rate}. This demonstrates that the larger class of $(\varepsilon,\delta)$-FDP protocols offers no improvement over $\delta=0$ protocols for the adaptive problem. 

Our second main result characterizes the pointwise estimation problem, for which the situation is more complex than the global case. Even in the non-private setting, adaptation to unknown smoothness incurs a logarithmic penalty, yielding the rate $(\log N/N)^{2\nu/(2\nu+1)}$ compared to the non-adaptive rate $N^{-2\nu/(2\nu+1)}$. Under differential privacy, this adaptation cost is either `equal' or compounded by additional logarithmic deterioration in the privacy-dominated regime, depending on the number of servers $m$ in relation to the total number of samples $N$.

Our characterization, formally given in the theorem below, shows that the sharp adaptive rate for pointwise estimation is given by:

\begin{equation}\label{eq:adaptive-pointwise-rate}
	\rho_{\textsc{point}}^2(\nu):= \rho_{\textsc{point}}(\nu; m,n,\varepsilon) = \left(\frac{N }{\log N}\right)^{-\frac{2\nu}{2\nu + 1}} + \left(\frac{ m n^2\varepsilon^2}{L_{m,N}}\right)^{-\frac{2\nu}{2\nu + 2}},
\end{equation}
where $\nu = \alpha - 1/p$ and
\begin{equation}\label{eq:elbow-effect-factor}
	L_{m,N} = \begin{cases}
		\frac{\log^2 N}{m}  \quad \text{ if } m \leq \log N, \\ 
		{\log N} \quad \text{ if } m > \log N.
	\end{cases}
\end{equation}

Compared to the non-adaptive pointwise rate $N^{-2\nu/(2\nu+1)} + (mn^2\varepsilon^2)^{-2\nu/(2\nu+2)}$, adaptation introduces logarithmic penalties in both terms: $\log N$ appears in the denominator of the first term (extending the non-private adaptation cost), while $L_{m,N}$ creates additional deterioration in the privacy term that depends on the server configuration. 

\begin{theorem}\label{thm:pointwise_cost_of_adaptation}
	Let $\delta \ll \varepsilon / (N \log N)$. Consider any collection $\mathcal{I} \subset (0,\infty)^2$ with $\alpha - 1/p > 1/2$ for all $(\alpha,p) \in \mathcal{I}$ and $(\alpha,p),(\alpha',p') \in \mathcal{I}$ such that $ \alpha - 1/p > \alpha' - 1/p'$.
	
	It holds that 
	\begin{equation*}
		\inf_{\hat{T} \in  \mathcal{T}^{\varepsilon, \delta}} \, \sup_{(\alpha,p) \in \mathcal{I}} \, \sup_{f \in \bar{\mathcal{B}}_{p,q}^{\alpha,R}} \, \mathbb{E}_f ( \hat{T} - f(t_0) )^2 \; \rho_{\textsc{point}}^{-2}(\alpha - 1/p) \asymp 1.
	\end{equation*}
\end{theorem}

Contrasting the above theorem with Theorem \ref{thm:global-FDP-rate} reveals a fundamental difference between global and pointwise adaptation under privacy. While global risk suffers only from the logarithmic deterioration in the privacy term, pointwise estimation incurs logarithmic penalties in both the statistical and privacy components. The factor $L_{m,N}$ captures an interesting elbow-effect: when $m \leq \log N$ (few servers), the privacy cost degrades by an additional logarithmic factor, but when $m > \log N$ (many servers), this extra penalty disappears. In particular, this implies that the adaptation cost for the pointwise problem under LDP is (in relative terms) less severe than for the CDP setting, as highlighted in Corollaries \ref{cor:central-dp} and \ref{cor:local-dp} below.


As with the global risk case, this result follows from matching upper and lower bounds established in Sections \ref{sec:adaptive_density_estimation_at_a_point} and \ref{sec:lower-bound}, showing that the rate in \eqref{eq:adaptive-pointwise-rate} is both necessary and achievable. Furthermore, also the pointwise adaptive rate remains valid for the larger class of chained sequentially interactive DP protocols, eventhough the rate is attained by a one-shot $(\varepsilon,0)$-FDP protocol.  

\subsection{Special cases: Central and local differential privacy}

Our general federated framework encompasses the classical central and local DP settings as special cases. To better understand the implications of our results for each of these respective extremes of the federated spectrum, we now examine them separately, which reveal different behaviors for the cost of adaptation.

The first corollary considers the cost of adaptation in the CDP setting, where a single trusted server holds all $N$ samples.

\begin{corollary}\label{cor:central-dp}
	Consider the CDP setting, where a single server holds all $n=N$ samples ($m=1$). Let $\delta \ll \varepsilon/(N \log N)$. 

	It holds that
\begin{align}
	\inf_{\hat{f} \in \mathcal{F}^{\varepsilon,\delta}} \, \sup_{ \alpha \in (\alpha_{\min}, \alpha_{\max})} \, \sup_{f \in \bar{\mathcal{B}}_{p,q}^{\alpha,R}} \, \mathbb{E}_f \| \hat{f} - f \|_2^2 & \left[ N^{-\frac{2\alpha}{2\alpha + 1}} + \left( \frac{N^2 \varepsilon^2}{\log N} \right)^{-\frac{2\alpha}{2\alpha + 2}} \right]^{-1} \asymp 1, \\
	\inf_{\hat{T} \in  \mathcal{T}^{\varepsilon, \delta}} \, \sup_{(\alpha,p) \in \mathcal{I}} \, \sup_{f \in \bar{\mathcal{B}}_{p,q}^{\alpha,R}} \, \mathbb{E}_f ( \hat{T} - f(t_0) )^2 &\left[ \left(\frac{N}{\log N}\right)^{-\frac{2\nu}{2\nu + 1}} + \left( \frac{N^2 \varepsilon^2}{\log^2 N} \right)^{-\frac{2\nu}{2\nu + 2}} \right]^{-1} \asymp 1,
\end{align}
for any $\alpha_{\max} > \alpha_{\min} > 0$ and collection $\mathcal{I} \subset (0,\infty)^2$ for which there exists $(\alpha,p),(\alpha',p') \in \cI$ such that $\alpha - 1/p \neq \alpha' - 1/p'$ and $\alpha - 1/p > 1/2$ for all $(\alpha,p) \in \mathcal{I}$.
\end{corollary}

In the central case, adaptation costs manifest as an additional logarithmic factor in the denominator of privacy term, while the non-private adaptive statistical rates remain unchanged. Specifically, in the regime where privacy dominates, the minimax global risk incurs a logarithmic penalty relative to the non-adaptive private rate, while the pointwise risk suffers logarithmic penalties in both the statistical and privacy components, but the private penalty is a squared logarithm.

Turning to the other extreme of the federated framework: LDP, where each server holds a single observation. We have the following corollary characterizing the cost of adaptation in this setting.

\begin{corollary}\label{cor:local-dp}
	Consider the LDP setting ($n=1$ sample per server and $m=N$ total servers). Let $\delta \ll \varepsilon/(N \log N)$. 

	The adaptation risks satisfy
\begin{align}
	\inf_{\hat{f} \in \mathcal{F}^{\varepsilon,\delta}} \, \sup_{ \alpha \in (\alpha_{\min}, \alpha_{\max})} \, \sup_{f \in \bar{\mathcal{B}}_{p,q}^{\alpha,R}} \, \mathbb{E}_f \| \hat{f} - f \|_2^2 & \left( \frac{N \varepsilon^2}{\log N} \right)^{-\frac{2\alpha}{2\alpha + 2}} \asymp 1, \\
	\inf_{\hat{T} \in  \mathcal{T}^{\varepsilon, \delta}} \, \sup_{(\alpha,p) \in \mathcal{I}} \, \sup_{f \in \bar{\mathcal{B}}_{p,q}^{\alpha,R}} \, \mathbb{E}_f ( \hat{T} - f(t_0) )^2 & \left( \frac{N \varepsilon^2}{\log N} \right)^{\frac{2\nu}{2\nu + 2}} \asymp 1,
\end{align}
for any $\alpha_{\max} > \alpha_{\min} > 0$ and collection $\mathcal{I} \subset (0,\infty)^2$ for which there exists $(\alpha,p),(\alpha',p') \in \cI$ such that $\alpha - 1/p \neq \alpha' - 1/p'$ and $\alpha - 1/p > 1/2$ for all $(\alpha,p) \in \mathcal{I}$.
\end{corollary}

In the case of LDP, the privacy constraints are so stringent that they completely dominate the estimation error, eliminating the non-private rate terms entirely, as was found in e.g. \cite{Cai2024FL-NP-Regression}. However, regarding adaptation costs, we observe another striking asymmetry when comparing to the CDP case: for pointwise risk, the adaptation cost matches that of the non-private case with no additional privacy-induced penalty. This behavior contrasts sharply with the central case, where privacy consistently worsens adaptation costs for both risk types.
 
The relative adaptation costs -- measured as the factor by which rates deteriorate compared to their non-adaptive private counterparts -- are thus less severe under LDP than CDP, despite LDP yielding uniformly worse absolute rates. This counterintuitive phenomenon arises from the distributed nature of local privacy mechanisms. To sketch an intuition for why this happens: DP procedures add some form of noise to (some statistic of) the data. To ensure DP, the noise needs to be, in an appropriate sense, `sufficiently heavy-tailed'. When adapting to unknown smoothness, each of the $m$ servers can add independent noise to its transcript. Aggregating these $m$ noisy transcripts effectively produces a distribution for the `aggregated noise' with lighter tails than the original noise distributions. This improved tail behavior can be exploited to obtain a performance gain when methods have to be adaptive. We note that this gain is relative to the otherwise more severe privacy costs inherent to the LDP model. In contrast, the CDP setting offers no such aggregation benefits. All $N$ samples are held by a single server, and the privacy noise added to the transcript cannot be averaged across multiple sources. This forces CDP mechanisms to either accept heavier-tailed noise or maintain lighter tails at the cost of increased variance, both of which exacerbate adaptation costs relative to the local setting.

\section{Optimally adaptive and differentially private procedures}\label{sec:adaptive_density_estimation_at_a_point}


In this section, we develop a differentially private estimator that achieves the optimal adaptive rate established in Theorems \ref{thm:global-FDP-rate} and \ref{thm:pointwise_cost_of_adaptation}. We begin by introducing the Besov space framework and the wavelet thresholding estimator, which is a standard technique for adaptive estimation in non-private settings. We then describe the limitations of the Laplace and Gaussian mechanisms for thresholding in the private setting, motivating the need for a novel noise distribution that balances tail decay and sensitivity in an optimal way. This novel noise distribution and the corresponding optimal thresholding estimator is then introduced in Section \ref{sec:optimal-noise-geometry}. In order to contextualize the method properly, we start with a recap introducing wavelet estimation in Besov space in Section \ref{sec:besov_definitions} and cover private estimation when smoothness is known in Section \ref{sec:known-smoothnes-DP-method}.

\subsection{Wavelet thresholding for estimation in Besov space}\label{sec:besov_definitions}

We start by giving a formal definition of a Besov space. For a function $f \in L_p[0,1]$, the $K$-th order difference operator $\Delta_{h}^{(K)}$ with step size $h > 0$ and integer $K > \alpha$ maps $f$ to a function that equals $\sum_{k=0}^K (-1)^{K-k} {K \choose k}^{} f(t+kh)$ when $t \in [0,1-Kh]$ and zero elsewhere. The Besov space $\cB^{\alpha}_{p,q}$ is defined as the set of all $f \in L_p[0,1]$ such that the \emph{Besov norm}, defined as
\begin{equation}\label{eq:besov-norm}
	\|f\|_{\alpha,p,q}:=\begin{cases}
		\|f\|_p + \left(\int \left[h^{-\alpha} \|\Delta_{h}^{(K)} f\|_p\right]^q \frac{dh}{h} \right)^{1/q} \quad & \text{ if } q < \infty,\\
		\|f\|_p + \sup_{h>0} h^{-\alpha} \|\Delta_{h}^{(K)} f\|_p \quad & \text{ if } q = \infty,
	\end{cases}
\end{equation}
is finite. Loosely speaking, $\alpha$ measures the degree of smoothness and $2 \leq p \leq \infty$ and $1 \leq q \leq \infty$ specify the norm used to measure the size of its derivatives. Besov spaces contain many important function classes, such as the Sobolev spaces ($p=q=2$), the H\"older spaces ($p=q=\infty$) and the space of functions with bounded variation.

Wavelets are known to have many favourable properties when using them for function estimation in classical settings, see for example \cite{donoho1998minimax,hall1999minimax,cai1999adaptive}. Under DP constraints, wavelet constructions have other desirable properties: they allow for exact control of the estimator's \emph{sensitivity} to changes in the data. Loosely speaking, this allows us to control the ``influence'' each individual observation has on the outcome of the estimator, whilst retaining the information the full sample has to a large extent.

Consider compactly supported, $A$-regular wavelets. Wavelet bases allow characterization of Besov spaces, with $\alpha$, $p$, and $q$ capturing the decay of wavelet coefficients. 

Following the Daubechies' construction, for any $A \in\mathbb{N}$ one can obtain an $A$-regular `father' wavelet $\phi(\cdot)$ with support on $[0,2A-1]$, and a `mother' $\psi(\cdot)$ wavelet with $A$ vanishing moments and support on $[-A+1,A]$. We refer to \cite{daubechies1992ten} for details. The basis functions are then obtained as dilations and translations of these functions:
\begin{align*}
	\big\{ \phi_{r},\psi_{lk}:\, r\in\{1,...,2^{l_0}\},\quad l \geq l_0,\quad k\in\{1,...,2^{l}\} \big\},
\end{align*}
with $\psi_{lk}(x)=2^{l/2}\psi(2^lx-k)$, and $\phi_{k}(x)=2^{l_0/2}\phi(2^{l_0}x-k)$. For a large enough primary resolution level $l_0$ and appropriate treatment of the boundary, the wavelet basis forms an orthonormal basis of $L_2[0,1]$. Hence, any function $f\in L_2[0,1]$ can be represented as
\begin{align}\label{eq:wavelet_transform}
	f = \sum_{k=1}^{2^{l_0}} f_{0k} \phi_{lk} +  \sum_{l=l_0}^{\infty}\sum_{k=1}^{2^{l}} f_{lk} \psi_{lk},
\end{align}
where $f_{0k} := \int f(x) \phi_{lk}(x) dx$ and $f_{lk} := \int f(x) \psi_{lk}(x) dx$ are the wavelet coefficients. Roughly speaking, the part of the wavelet decomposition corresponding to $\phi_k$ approximates the `low-frequency' or `global' part of the function, while the part corresponding to $\psi_{lk}$ captures the `high-frequency' or `local' part.

We describe non-private wavelet based estimation first. Define $\hat{f}_{0k} = N^{-1} \sum_{i=1}^N \phi_{lk}(X_i)$ and $\hat{f}_{lk} = N^{-1} \sum_{i=1}^N \psi_{lk}(X_i)$. A \emph{truncated wavelet estimator} $\hat{f}$ is then defined as follows:
\begin{equation*}
	\hat{f}(t) = \sum_{k=1}^{2^{l_0}} \hat{f}_{0k} \, \phi_{k}(t) + \sum_{l=l_0}^{L} \sum_{k =1}^{2^l} \hat{f}_{lk} \, \psi_{lk}(t).
\end{equation*}
It is a common knowledge that the above estimator achieves the non-private minimax rate for the global risk if $L$ is chosen as $L = \lceil 1/(2\alpha+1) \log (N) \rceil$, which requires knowledge of $\alpha$. Similarly, for the pointwise risk, the estimator $\hat{f}(t_0)$ achieves the minimax rate if $L$ is chosen as $L = \lceil 1/(2\nu+1) \log (N) \rceil$, requireing knowledge of both $\alpha$ and $p$.

In the non-private setting with $\alpha$ and $p$ unknown, a successful approach to constructing rate optimal adaptive estimators is through \emph{wavelet thresholding}. Here, the wavelet coefficients are grouped together in blocks and thresholded at an appropriate level. 

More specifically, consider blocks $B_{lj} = {k: j b_l \leq k \leq (j+1)b_l}$ with $b_l \in \N$ such that the $B_{lj}$'s partition ${1,\dots,2^l}$ for $j \in \cJ_l := {1,\dots, 1 \wedge (2^l/b_l)}$ and $b_l$ an integer such that $2^l/b_l \in \N$. Consider the \emph{$b_l$-block soft-thresholding operator} $\eta_{\tau}^{b_l}: \R^{b_l} \to \R^{b_l}$ defined as
\begin{equation}\label{eq:soft-thresholding-operator}
	\eta_{\tau}^{b_l}(y) = \left( 1 - \frac{\tau}{\|y\|_2} \right)_+y,
\end{equation} 
and write $v_{lB_{lj}} = (v_{lk})_{k \in B_{lj}}$ for the elements of a vector $v = \{v_{lk} : l = 1,\dots,L^*, k =1,\dots,2^l\}$ in the block $B_{lj}$. Wavelet (block) soft-thresholding estimators are of the form
\begin{equation}\label{eq:classical_thresholding_estimator}
	\hat{f}^{\text{WT}}(t) = \sum_{k =1}^{2^{l_0}}  \hat{f}_{0k} \, \phi_{k}(t) + \sum_{l=l_0}^{L^*} \sum_{j \in \cJ_l} \left(\psi_{lB_{lj}}(t) \right)^\top \eta_{\tau}^{b_l}\left(\hat{f}_{lB_j}\right).
\end{equation}
For $b_l = 1$ for all $l$, this corresponds to term-by-term thresholding estimator of \cite{donoho1998minimax}, and are known to be minimax adaptive for the pointwise risk \cite{cai2003pointwiseBesov} for the choice of threshold $\tau = \sqrt{2 \log (N) / N }$. For the global risk, block thresholding with $b_l \asymp \log (N)$ is known to be minimax adaptive \cite{cai1999adaptive}.

The idea behind thresholding, is that the wavelet coefficients decay at a rate that depends on the smoothness of the function. By thresholding the estimated wavelet coefficients, roughly speaking, one can remove the coefficient estimates that are dominated by noise, and retain the large coefficients, for which the coefficient estimates surpass the threshold. In addition, by grouping the coefficients in blocks, a borrowing strength effect is achieved, which allows for optimal performance in the global risk. We note that for pointwise risk, such grouping in blocks is not needed \cite{cai1999adaptive,cai2003pointwiseBesov}.

\subsection{Optimal FDP wavelet estimators when smoothness is known}\label{sec:known-smoothnes-DP-method}

When the smoothness parameters $\alpha$ and $p$ are known, optimal private density estimation under federated differential privacy can be achieved using the Laplace mechanism applied to wavelet coefficients, by using truncated wavelet estimator approach (without a thresholding step). In the FDP setting, each server $j = 1,\ldots,m$ computes empirical wavelet coefficients from its local data $X^{(j)} = (X_1^{(j)}, \ldots, X_n^{(j)})$:
\begin{equation}\label{eq:local-wavelet-coeffs}
	\hat{f}_{0r}^{(j)} = \frac{1}{n}\sum_{i=1}^n \phi_r(X_i^{(j)}), \quad \hat{f}_{lk}^{(j)} = \frac{1}{n}\sum_{i=1}^n \psi_{lk}(X_i^{(j)}),
\end{equation}
for $r = 1,\ldots,2^{l_0}$ and $(l,k)$ with $l_0 < l \leq L$, $k = 1,\ldots,2^l$.

Each server then adds calibrated Laplace noise to its local coefficients before transmitting them to the aggregator. The noise scale is determined by the $L_1$-\emph{sensitivity} of the wavelet coefficients: the maximum change in $L_1$-norm when altering one observation. For compactly supported, $A$-regular wavelets, this sensitivity is bounded by $c_{\phi,\psi} 2^{L/2} / n$ for some wavelet-dependent constant $c_{\phi,\psi} > 0$. Server $j$ transmits the noisy transcript:
\begin{equation}\label{eq:local-noisy-transcript}
	T^{(j)} = \left\{ \hat{f}_{0r}^{(j)} + W_{0r}^{(j)}, \hat{f}_{lk}^{(j)} + W_{lk}^{(j)} : r = 1,\ldots,2^{l_0}, \; l_0 < l \leq L, \; k = 1,\ldots,2^l \right\},
\end{equation}
where $W_{0r}^{(j)}, W_{lk}^{(j)} \sim \text{Laplace}(c_{\phi,\psi} 2^{l/2} / (n \varepsilon))$ are independent. The final estimator aggregates these noisy transcripts:
\begin{equation}\label{eq:fdp-laplace-estimator}
	\hat{f}^{\text{Lap}}(t) = \sum_{r=1}^{2^{l_0}} \left( \hat{f}_{0r} + \frac{1}{m}\sum_{j=1}^m  W_{0r}^{(j)}\right) \phi_r(t) + \sum_{l=l_0}^{L} \sum_{k=1}^{2^l} \left(\hat{f}_{lk} + \frac{1}{m}\sum_{j=1}^m  W_{lk}^{(j)}\right) \psi_{lk}(t).
\end{equation}

With appropriate choice of truncation level $L$ (again depending on the smoothness), this $(\varepsilon,0)$-FDP estimator achieves the optimal non-adaptive rates established in \cite{Cai2024FL-NP-Regression}. 

However, similarly to the non-private setting, this approach fails when smoothness is unknown. Adaptive estimation requires data-driven selection of the truncation level $L$, typically through thresholding of the aggregated noisy coefficients. Naively applying soft-thresholding to the Laplace noise-peturbed aggregated coefficients leads to suboptimal performance. Even sophisticated approaches such as the $L_1$-norm rescaling methods of \cite{butucea_LDP_adaptation} achieve suboptimal rates.

The fundamental issue, loosely speaking, lies in the fact that current privacy mechanisms do not account for dependencies present across the  resolution levels. Privacy mechanisms such as the Laplace mechanism add effectively independent noise to each wavelet coefficient. However, this independent treatment ignores the joint structure present in the wavelet coefficients. The resulting noise, while individually calibrated, creates an unfavorable trade-off between variance and tail behavior, hampering the concentration properties essential for adaptive procedures such as thresholding. 

To achieve optimal adaptive rates under FDP constraints, we introduce a novel noise mechanism that accounts for dependence accross resolution levels, whilst aiming at optimal tail decay to facilitate thresholding.

\subsection{Private thresholding: finding the right noise geometry}\label{sec:optimal-noise-geometry}

The oracle inequality for soft-thresholding reveals why standard privacy mechanisms struggle with adaptive estimation. 

We motivate the choice of noise distribution by considering the following oracle inequality for the (block) soft-thresholding error. For a proof, see \cite{CaiZhou2010NEF}, Lemma 6. 

\begin{lemma}\label{lem:private-block-thresholding-oracle}
	Consider $f \in \R^d$, a random vector $Z$ taking values in $\R^d$, and let $Y = f + Z$. Let $\eta_{\tau}^{d}: \R^d \to \R^d$ be the soft-thresholding operator for threshold $\tau > 0$. Then,
	\begin{equation}\label{eq:oracle-inequality}
		\E \| \eta_{\tau}^{d}(Y) - f \|_2^2 \leq \min \{ \|f\|_2^2, \, 4\tau^2 \} + 4 \E \| Z \|_2^2 \mathbbm{1}\{ \|Z\|_2 > \tau \}.
	\end{equation}
\end{lemma}

The oracle inequality decomposes the thresholding error into bias and variance components. The term $\min \{ \|f\|_2^2, \, 4\tau^2 \}$ represents a bias-variance trade-off: smaller thresholds $\tau$ reduce bias for small signals but increase sensitivity to noise. The tail term $4 \E \| Z \|_2^2 \mathbbm{1}\{ \|Z\|_2 > \tau \}$ heavily depends on the noise distribution of $Z$, in particular its tail decay around the threshold $\tau$. Ignoring the presence of statistical noise, this term captures the contribution of privacy noise to the thresholding error.

For optimal thresholding, we want the smallest possible threshold $\tau$ such that the tail term remains controlled. This requires noise $Z$ with two key properties: (1) variance scaling appropriately with the sensitivity of the underlying statistic, and (2) rapid tail decay to ensure $\P(\|Z\|_2 > \tau)$ decreases quickly as $\tau$ increases. Standard privacy mechanisms fail to achieve both properties simultaneously. The Laplace mechanism does not achieve this balance (even after applying sophisticated rescaling techniques as in \cite{butucea_LDP_adaptation}), and while the Gaussian mechanism has good tail decay properties, it requires larger variance to achieve the same privacy guarantees for small values of $\delta$.

Our solution is based on using a carefully designed norm that exploits the multiscale structure of wavelets, and employing the exponential mechanism with respect to that norm, akin to the $K$-norm mechanism of \cite{hardt2010knorm}. This approach allows us to achieve the desired balance between sensitivity and tail decay.

Define the \emph{oscillation} of a function $g \in L_2[0,1]$ as $\operatorname{osc}(g) := \sup_{t \in [0,1]} g(t) - \inf_{t \in [0,1]} g(t)$. The \emph{multiscale-oscillation norm} of the wavelet coefficient vectors is defined as
\begin{equation}\label{eq:osc-norm}
	\|u\|_{V_L} = \sup_{\substack{g \in \operatorname{span}\{\psi_{lk} : (l,k) \in V_L\} \\ \operatorname{osc}(g) \leq 1}} \left| \left\langle \sum_{(l,k) \in V_L} u_{lk} \psi_{lk}, g \right\rangle_{L^2[0,1]} \right|,
\end{equation}
where for $l_0, L \in \N$, $V_L = \{ (l,k) \, : \, l=l_0,\dots,L, \, k = 1,\dots,2^l \}$. Let $s_L = \sum_{l=l_0}^L 2^l$. This norm has the crucial property that its sensitivity with respect to changing one data point is exactly $1/n$, independent of the resolution levels considered. 

\begin{lemma}\label{lem:sensitivity-osc-norm}
	For any datasets $ x, x' \in [0,1]^n $ that differ in at most one coordinate, say $ x_i \neq x'_i $, and let $\Delta_{x,x'}$ denote the coordinate-wise difference of the corresponding empirical wavelet coefficients in \eqref{eq:local-wavelet-coeffs}. In multi-scale-oscillation norm it satisfies the following bound:
	\begin{equation*}
		\|\Delta_{x,x'}\|_{V_L} \leq \frac{1}{n}.
	\end{equation*}
\end{lemma}
\begin{proof}
	Let $f_{\Delta}$ denote the inverse wavelet transform of the vector $\Delta_{x,x'}$. We compute
\begin{equation*}
\|\Delta_{x,x'}\|_{V_L} = \sup_{\substack{g \in \operatorname{span}\{\psi_{lk}\} \\ \operatorname{osc}(g) \leq 1}} \left| \left\langle f_{\Delta}, g \right\rangle_{L^2} \right|.
\end{equation*}
Consider the reproducing kernel $ K_L(s,t) = \sum_{(l,k) \in V_L} \psi_{lk}(s) \psi_{lk}(t) $, such that for $g \in \operatorname{span}\{\psi_{lk} : (l,k) \in V_L\}$, it holds that 
\begin{align*}
	g(s) &= \sum_{(l,k) \in V_L} \psi_{lk}(s) \left\langle \psi_{lk}, g \right\rangle_{L^2} = \left\langle K_L(s, \cdot), g \right\rangle_{L^2}.
\end{align*}
By linearity of the integral,
\begin{align*}
	\left\langle f_{\Delta}, g \right\rangle_{L^2} &= \frac{1}{n} \int_0^1 \left( \sum_{(l,k) \in V_L} (\psi_{lk}(x_i) - \psi_{lk}(x'_i)) \psi_{lk}(t) \right) g(t) \, dt \\ 
	&= \frac{1}{n} \left[ \left\langle K_L(x_i, \cdot), g \right\rangle_{L^2} + \left\langle K_L(x_i', \cdot), g \right\rangle_{L^2} \right] = \frac{1}{n} \left[ g(x_i) - g(x'_i) \right].
\end{align*}
As $ |g(x_i) - g(x'_i)| \leq \operatorname{osc}(g) \leq 1 $, it follows that $\|\Delta_{x,x'}\|_{V_L} \leq \frac{1}{n}$.
\end{proof}

The multiscale-oscillation norm can be used to generate noise through the exponential mechanism: $V^{(j)} = (V_{lk}^{(j)})_{l,k}$ with density proportional to the map $v\mapsto \exp(-\varepsilon n \|v\|_{V_L})$. This yields noisy coefficients
\begin{equation}\label{eq:osc-mechanism-coefficients}
	T_{lk}^{(j)} = \hat{f}_{lk}^{(j)} + V_{lk}^{(j)}, \quad (l,k) \in V_L,
\end{equation}
that are $(\varepsilon,0)$-differentially private. Averaging these noisy coefficients across servers yields a private estimate of the wavelet coefficients:
\begin{equation*}
	\hat{f}_{lk}^{\textsc{PW}} = \frac{1}{m} \sum_{j=1}^m T_{lk}^{(j)} = \hat{f}_{lk} + \bar{V}_{lk}, 
\end{equation*}
where $\bar{V}_{lk} = m^{-1} \sum_{j=1}^m V_{lk}^{(j)}$. As with thresholding in the non-private case, we choose a a sufficiently large truncation level $L^* = \lceil \log(N) \rceil$, and then apply (block) soft-thresholding to the wavelet coefficients as a post-processing step to obtain the final estimator. This final post-processing differs depending on whether we consider pointwise or global risk.

\subsubsection{Adaptive global risk post-processing step}

For global risk estimation, we employ block thresholding to leverage borrowing-of-strength effects across wavelet coefficients. We partition the coefficients at each resolution level $l$ into blocks $B_{lj} = \{k: jb_l \leq k \leq (j+1)b_l\}$ for $j \in \mathcal{J}_l := \{1,\ldots, \lfloor 2^l/b_l \rfloor\}$, where $b_l$ is the integer closest to $\lceil \log N \rceil$ such that $2^l/b_l$ is an integer.

The adaptive estimator applies soft-thresholding to each block of noisy wavelet coefficients:
\begin{equation}\label{eq:private-block-soft-thresholding-estimator}
	\hat{f}^{\textsc{BTPW}}(t) = \sum_{r=1}^{2^{l_0}} \hat{f}_{0r}^{\textsc{PW}} \phi_r(t) + \sum_{l=l_0}^{L^*} \sum_{j \in \mathcal{J}_l} \psi_{lB_j}^T(t) \, \eta_{\tau_l}^{b_l}\left(\hat{f}_{lB_{lj}}^{\textsc{PW}}\right),
\end{equation}
where $\psi_{lB_j}(t) = (\psi_{lk}(t))_{k \in B_{lj}}$ and $\hat{f}_{lB_{lj}}^{\textsc{PW}} = (\hat{f}_{lk}^{\textsc{PW}})_{k \in B_{lj}}$ denotes the vector of aggregated noisy coefficients in block $B_{lj}$.

What remains is to choose thresholds that balance bias and variance for both statistical and privacy noise. We set:
\begin{equation}\label{eq:block-threshold-choice}
	\tau_l = \sqrt{\kappa_1 \frac{\log N}{N} + \kappa_2 \frac{2^l \log^2 N}{mn^2\varepsilon^2}},
\end{equation}
where $\kappa_1, \kappa_2 > 0$ are wavelet-dependent constants. The first term corresponds to the statistical noise threshold, while the second term accounts for the privacy noise, with the additional $\log N$ factor reflecting the adaptation cost.

The threshold choice is motivated by the following tight tailbound for our multiscale-oscillation noise mechanism.

\begin{lemma}\label{lem:block-thresholding-tailbound-privacy-noise-only}
	Let $\bar{V}_{lk} = \frac{1}{m} \sum_{j=1}^m V_{lk}^{(j)}$ be the aggregated privacy noise and $\bar{V}_{lB_j} = (\bar{V}_{lk})_{k \in B_{lj}}$. For sufficiently large $\kappa > 0$ and all $l \in \{l_0,\ldots,L^*\}$, $j \in \mathcal{J}_l$,
	\begin{equation*}
		\mathbb{E} \|\bar{V}_{lB_{lj}}\|_2^2 \mathbbm{1}\left\{ \|\bar{V}_{lB_{lj}}\|_2 \geq \kappa \frac{2^{l/2} b_l}{\sqrt{m} n \varepsilon} \right\} \leq C \frac{2^l b_l^2}{mn^2\varepsilon^2} e^{-b_l}
	\end{equation*}
	for a constant $C > 0$.
\end{lemma}

The exponential decay in the block size $b_l$ shown in Lemma~\ref{lem:block-thresholding-tailbound-privacy-noise-only} is crucial for controlling the tail term in the oracle inequality and enables the use of small thresholds necessary for optimal adaptation. 

Analysis of the multiscale-oscillation noise is considerably more complex than for standard Laplace or Gaussian noise. An important step in the analysis is that the multiscale-oscillation noise mechanism admits a tractable decomposition as a Gamma-distributed `radius' and a uniformly distributed `direction' on the norm's unit ball:
\begin{equation*}
	(V_{lk}^{(j)})_{(l,k) \in V_{L^*}} \overset{d}{=} D^{(j)} U^{(j)}, \quad D^{(j)} \sim \Gamma\left(s_{L^*} + 1, \frac{1}{n\varepsilon}\right), \quad U^{(j)} \sim \mathcal{U}\left( \{ u: \|u\|_{V_{L^*}} \leq 1 \} \right),
\end{equation*}
where $s_{L^*} = \sum_{l=l_0}^{L^*} 2^l$ is the total number of detail coefficients, see e.g. \cite{hardt2010knorm}. Exploiting this decomposition, the analysis of the tail behavior boils down to of the averaged ${U}_{lB_j}^{(j)}$. This is possible by combining stochastic domination arguments where we compare with more tractable uniform distributions on convex compacts. Complete technical details of this analysis are provided in Section~\ref{sec:multiscale-oscillation-norm-properties} in the Supplementary Material.

Combining the oracle inequality with the exponential tail bounds yields the following performance guarantee for our estimator.

\begin{theorem}\label{thm:upper-bound-global-risk}
	For any $\alpha, p, q$ with $0 < \alpha < A$, $p \geq 2$, and $q \geq 1$,
	\begin{equation*}
		\sup_{f \in \mathcal{B}_{p,q}^\alpha(R)} \mathbb{E}_f \| \hat{f}^{\textsc{BTPW}} - f \|_2^2 \lesssim N^{-\frac{2\alpha}{2\alpha + 1}} + \left( \frac{\log N}{mn^2\varepsilon^2} \right)^{\frac{2\alpha}{2\alpha + 2}}.
	\end{equation*}
\end{theorem}

Since the estimator is a post-processing of an $(\varepsilon,0)$-FDP protocol, it inherits the same privacy guarantee. The theorem shows that the estimator achieves the optimal adaptive rate for the global risk simultaneously over all Besov spaces $\cB_{pq}^\alpha(R)$ with $0 < \alpha < A$, $p \geq 2$ and $q \geq 1$. Taking $A$ large enough, we see that the estimator attains the optimal adaptive rate of Theorem \ref{thm:global-FDP-rate}. We defer the proof to Section \ref{sec:proof-upper-bound-thms} of the Supplementary Material.

\begin{remark}
	As an alternative to soft-thresholding, one can consider the hard-thresholding operator. The estimator is then given by
	\begin{equation}\label{eq:private-block-thresholding-estimator}
		\breve{f}^{\textsc{BTPW}}(t) = \sum_{r=1}^{2^{l_0}} \hat{f}_{0r}^{\textsc{PW}} \phi_r(t) + \sum_{l=l_0}^{L^*} \sum_{j \in \mathcal{J}_l} \mathbbm{1} \left\{ \|\hat{f}_{lB_{lj}}^{\textsc{PW}}\|_2 \geq \tau_l \right\} \psi_{lB_j}^T(t) \hat{f}_{lB_{lj}}^{\textsc{PW}}.
	\end{equation}
	A similar proof yields the same upper bound as in Theorem \ref{thm:upper-bound-global-risk}. 
\end{remark}

\subsubsection{Adaptive pointwise risk estimation}

For pointwise risk with unknown $\alpha$ and $p$, we use term-by-term thresholding (setting $b_l = 1$ for all $l$). While block thresholding can  achieve optimal adaptive pointwise rates in the non-private setting \cite{cai1999adaptive}, term-by-term thresholding offers superior performance under the multiscale-oscillation norm privacy mechanism. We expand on the reason for this below, in Remark \ref{rmk:term-by-term-versus-pt-eval}. 

The pointwise estimator is given by
\begin{equation}\label{eq:private-pointwise-soft-thresholding-estimator}
	\hat{f}^{\textsc{TTPW}}(t_0) = \sum_{r=1}^{2^{l_0}} \hat{f}_{0r}^{\textsc{PW}} \phi_r(t_0) + \sum_{l=l_0}^{L^*} \sum_{k=1}^{2^l} \eta_{\tau_l}\left(\hat{f}_{lk}^{\textsc{PW}}\right) \psi_{lk}(t_0),
\end{equation}
where $\eta_{\tau_l}(\cdot)$ denotes the soft-thresholding operator applied coordinate-wise.

The pointwise setting requires different threshold calibration. We set:
\begin{equation}\label{eq:pointwise-threshold-choice}
	\tau_l = \sqrt{\kappa_{1\psi} \frac{\log N}{N} + \kappa_{2\psi} \frac{2^l L_{m,N}}{mn^2\varepsilon^2}},
\end{equation}
where $\kappa_{1\psi}, \kappa_{2\psi} > 0$ are constants depending on the choise of wavelets and
\begin{equation*}
	L_{m,N} = \begin{cases}
		\log (N), & m \geq \log (N), \\
		\frac{\log^2 (N)}{m}, & m < \log (N).
	\end{cases}
\end{equation*}

The factor $L_{m,N}$ captures how server distribution affects noise tail behavior around the threshold. When $m \geq \log N$, averaging across many servers transforms the tail decay from sub-exponential (for individual servers) to sub-Gaussian rates, reducing the adaptation penalty from up to $\log^2 N$ to $\log N$. When $m < \log N$, averaging effects are insufficient and the sub-exponential tails necessitate larger thresholds. This is captured by the following tail bound for the pointwise setting, for which we defer its proof to Section \ref{sec:multiscale-oscillation-norm-properties} of the Supplementary Material.

\begin{lemma}\label{lem:term-thresholding-tailbound-privacy-noise-only}
	Consider $\bar{V}_{lk} = \frac{1}{m} \sum_{j=1}^m V_{lk}^{(j)}$ for $(l,k) \in V_{L^*}$. For all $l \in \{l_0, \dots, L^*\}$, $k = 1, \dots, 2^l$, and $t \ge c_0 \frac{2^{l/2}}{\sqrt{m} n \varepsilon}$,
	\begin{equation*}
		\mathbb{E} \bar{V}_{lk}^2 \mathbbm{1}\left\{ |\bar{V}_{lk}| \geq t \right\} \lesssim \left( t^2 + \frac{2^l}{m n^2 \varepsilon^2} \right) \exp \left( - c_1 m \min \left\{ \frac{t^2 n^2 \varepsilon^2}{2^l}, \frac{t n \varepsilon}{2^{l/2}} \right\} \right)
	\end{equation*}
	for constants $c_0, c_1 > 0$.
\end{lemma}

The next theorem confirms that the pointwise estimator achieves optimal adaptive rates. Its proof follows by the oracle inequality in Lemma \ref{lem:private-block-thresholding-oracle} (applied with $d=1$) combined with the tail bound in Lemma \ref{lem:term-thresholding-tailbound-privacy-noise-only}, and is provided in Section \ref{sec:proof-upper-bound-thms}.

\begin{theorem}\label{thm:upper-bound-pointwise}
	Let $t_0 \in (0,1)$ be given and consider the estimator $\hat{T} := \hat{f}^{\textsc{TTPW}}(t_0)$. For any $p \in [2,\infty]$, $q \in [1,\infty]$, and $\alpha$ such that $\nu := \alpha - 1/p > 1/2$ and $\alpha < A$, it holds that
	\begin{equation*}
		\sup_{f \in \mathcal{B}_{p,q}^{\alpha}(R)} \mathbb{E}_f |\hat{T} - f(t_0)|^2 \lesssim \left(\frac{\log N}{N}\right)^{\frac{2\nu}{2\nu + 1}} + \left(\frac{L_{m,N}}{mn^2\varepsilon^2}\right)^{\frac{2\nu}{2\nu + 2}}.
	\end{equation*}
\end{theorem}

The theorem confirms that pointwise adaptive estimation incurs logarithmic penalties in both the statistical term (changing $N^{-2\nu/(2\nu+1)}$ to $(N/\log N)^{-2\nu/(2\nu+1)}$, as in the non-private setting) and the privacy term (through $L_{m,N}$). 

\begin{remark}\label{rmk:term-by-term-versus-pt-eval}
	When $m=1$ (the CDP setting), the plug-in block thresholding estimator $\hat{f}^{\textsc{BTPW}}(t_0)$ performs equally well as the term-by-term approach, since averaging across machines provides no reduction in the heaviness of the tails of the privacy noise.
\end{remark}

\section{Deriving the adaptation lower bounds}\label{sec:lower-bound}

This section establishes fundamental lower bounds that characterize the unavoidable cost of adaptation under differential privacy constraints. We prove that any differentially private estimator must pay additional logarithmic penalties beyond the classical non-private adaptation cost, and this penalty is inherent to the privacy requirement rather than an artifact of any particular estimation procedure.

Our analysis separates the pointwise and global estimation settings, which each exhibit different behavior and require different proof strategies. The pointwise case, treated in Section \ref{sec:pointwise-lower-bound}, requires novel modifications to classical constrained risk inequality techniques. The global case, covered in Section \ref{sec:global-lower-bound}, is proven through a dimension-free Fisher information bound that captures the interaction between adaptation and privacy across multiple regularity levels.

Our impossibility results extend beyond the one-shot FDP protocols used in our upper bounds. The lower bounds apply to sequential protocols, which are known to show improvement in certain LDP settings (see e.g. \cite{acharya2022interactive,butucea2023interactive,butucea2025nonparametric} and references therein). In sequential protocols, servers communicate in a single round through a chain (server $j$ sends to server $j+1$, who sends to server $j+2$, etc.), allowing each server to condition on message from a previous server. Hence, our lower bounds demonstrate that the logarithmic adaptation penalties are fundamental limitations that cannot be circumvented even when servers can leverage information from earlier steps in the communication chain.

For completeness, we formalize the sequential FDP setting below: servers communicate in a chain where server $j$ receives message $T^{(j-1)}$ from the previous server and computes $T^{(j)}$ based on its local data $X^{(j)}$ and the received message. Each step must satisfy local differential privacy with respect to the server's own data.

\begin{definition}[Sequential FDP]\label{def:federated_differential_privacy_chained}
A sequential distributed protocol is $(\varepsilon,\delta)$-FDP if each server $j$'s transcript $T^{(j)}$, conditioned on any fixed input $T^{(j-1)}=t$, satisfies: for datasets $x,x' \in [0,1]^n$ differing in one observation,
\begin{equation*}
	\mathbb{P}\left( T^{(j)} \in A \mid X^{(j)} = x, T^{(j-1)} =t \right) \leq e^\varepsilon \mathbb{P}\left( T^{(j)} \in A \mid X^{(j)} = x', T^{(j-1)} = t \right) + \delta.
\end{equation*}
\end{definition}

We note that, unlike previous work, we do not require the conditional distributions of the transcripts to be dominated. This condition, which mandates that the distribution of the current transcript---conditional on the data and all previous transcripts---be dominated by a reference measure for every possible realization of the conditioning variables, is typically assumed in the literature studying global risk metrics (e.g. $L_2$-error) \cite{barnes2020fisher,Cai2024FL-NP-Regression, xue2024optimalestimationprivatedistributed}, either explicitly or implicitly by restricting the transcript space to take values in a countable space. In Section \ref{sec:properties-of-induced-fisher-information-matrix}, we demonstrate that our lower bound techniques remain valid without these restrictive assumptions. 

In the remainder of this section, for pointwise risk, let $\mathcal{T}(\varepsilon,\delta)$ include sequential $(\varepsilon,\delta)$-FDP protocols, and let $\cF(\varepsilon,\delta)$ denote the corresponding class for global risk. We note that this forms a strictly larger class than the one-shot protocols considered in Definition \ref{def:federated_differential_privacy}.

\subsection{Adaptation lower bound for pointwise estimation}\label{sec:pointwise-lower-bound}

Our approach builds upon the classical constrained risk inequality line of thought developed in \cite{brown1996constrained}, which establishes adaptation costs by showing that improved performance at one function necessarily creates worse performance elsewhere in the parameter space. 

Differential privacy fundamentally alters such an analysis in two important ways. First, privacy constraints limit the distinguishability between probability measures, making total variation distance a natural distance to capture distinguishability, as observed by for example \cite{steinberger2020geometrizing}. Second, the standard change-of-measure techniques used in classical proofs must be replaced with privacy-specific methods that capture how the randomness introduced by the privacy mechanism affects distinguishability.

Our key technical contribution is the a constrained risk inequality for federated differential privacy, which we derive in a general setting in Section \ref{sec:private-constrained-risk-inequality}, in the form of Lemma \ref{lem:private-constrained-risk}. This lemma has general consequences for super-efficient private estimation which might be of independent interest (see for an illustration Example \ref{example}). The lower bound on the pointwise risk is established in Section \ref{sec:pointwise-risk-lower-bound-subsection} by combining Lemma \ref{lem:private-constrained-risk} with arguments specific to the adaptive setting.

\subsubsection{A differentially private constrained risk inequality}\label{sec:private-constrained-risk-inequality}

We formalize the super-efficiency trade-off between differential privacy and risk in the following lemma, which we state here in the simpler $(\varepsilon,0)$-FDP case. The general $(\varepsilon,\delta)$-FDP version appears in Section \ref{sec:pointwise-risk-lower-bound-proofs} of the Supplementary Material, together with its proof.

\begin{lemma}\label{lem:private-constrained-risk}
Consider a model $\{P_f \, : \, f \in \Theta\}$ on $(\mathcal{X},\mathscr{X})$ indexed by a semi-metric space $(\Theta,\mathrm{d})$, and $f,g \in \Theta$ such that $\mathrm{d}(f, g) \geq \Delta$ for some $\Delta > 0$. Consider servers $j=1,\ldots,m$ each with i.i.d. samples $X_1^{(j)},\ldots,X_n^{(j)}$ with distribution $P_h$ for $h \in \Theta$. 

If an $(\varepsilon,0)$-FDP estimation protocol $\hat{T}$ on the basis of $(X^{(j)}_i)_{i=1,\dots,n}^{j=1,\dots,m}$ satisfies 
\begin{align*}
\mathbb{E}_f \, \mathrm{d}(\hat{T}, f)^2 &\leq \gamma^2 \Delta^2 \quad \text{for some } \gamma > 0,
\end{align*}
then
\begin{align*}
\mathbb{E}_g \, \mathrm{d}(\hat{T}, g)^2 \geq \frac{\Delta^2}{4} \left[1 - 2 \exp \left( m (\bar{\varepsilon} \wedge \bar{\varepsilon}^2) + \log \gamma \right) \right],
\end{align*}
where $\bar{\varepsilon} = 6 n \varepsilon \|P_f - P_g \|_{\text{TV}}$.
\end{lemma}

The lemma captures the fundamental trade-off imposed by differential privacy: when an estimator achieves risk $\gamma^2 \Delta^2$ under $P_f$ (where $\gamma$ quantifies the improvement factor at $f$), its performance under the alternative measure $P_g$ degrades necessarily; unless the privacy constraint is weak (large $\varepsilon$) or the measures are far apart in total variation. This degradation is proportional to the logarithm of the improvement factor $\gamma$.


The lemma hence establishes a private version of a constraint risk inequality: it states that super-efficiency at a particular point in the parameter space comes at a cost in terms of increased risk at other points in the parameter space. We exemplify this in the following simple example.

\begin{example}\label{example}
	Consider $(\varepsilon,0)$-differentially private estimation of a population proportion $p$ over a neighborhood of $1/2$, on the basis of observations $Y^{(j)}_1,\dots,Y^{(j)}_n \iid \operatorname{Ber}(p)$ for $j = 1,\dots,m$, where $\varepsilon,m,n$ are allowed to have the asymptotics as introduced in Section \ref{sec:problem_formulation}. The minimax rate for the problem is easy to derive: after adding Laplace noise, the server averages can be privately communicated, where averaging over the servers results in an $(\varepsilon,0)$-FDP estimator attaining the rate $(mn)^{-1}+(mn^2\varepsilon^2)^{-1}$. Known techniques (for instance, via Lemma \ref{lem:private-constrained-risk}) yield a matching lower bound, confirming the optimality of this rate. 

However, Lemma \ref{lem:private-constrained-risk} provides a more nuanced insight: it demonstrates that super-efficiency at a single point incurs a penalty elsewhere. To illustrate this, consider $\hat{T}$ to be an $(\varepsilon,0)$-FDP estimator that is \emph{super-efficient} at $p = 1/2$; meaning that $\E_{1/2} | \hat{T} - 1/2 |^2 \lesssim (m n^2 \varepsilon^2)^{- C}+ (mn)^{-C}$ for some $C > 1$ (we construct such an estimator in Section \ref{sec:appendix:DP-hodge} of the Supplementary Material). Lemma \ref{lem:private-constrained-risk} then implies that for any fixed neighborhood $B$ of $1/2$,
\begin{equation*}
	\sup_{p \in B} \, \E_p | \hat{T} - p |^2 \gtrsim \frac{L_{m,N}}{m n^2 \varepsilon^2}\left(1 - 2 e^{ - c \log (N)} \right)
\end{equation*}
for some constant $c > 0$, where $L_{m,N}$ is the elbow-effect factor from \eqref{eq:elbow-effect-factor}. To see this, one can apply the lemma with $p_{m,n,\varepsilon} = 1/2 + c' \sqrt{{L_{m,N}}/(m n^2 \varepsilon^2)}$ for sufficiently small $c' > 0$, noting that the total variation distance between $\operatorname{Ber}(1/2)$ and $\operatorname{Ber}(p_{m,n,\varepsilon})$ scales as $|p_{m,n,\varepsilon} - 1/2|$.

This example shows that super-efficient estimators performs worse than the minimax rate by a logarithmic factor $L_{m,N}$ over the neighborhood $B$. Crucially, this it mirrors the elbow effect inherent to adaptation under federated privacy constraints: the elbow effect in super-efficiency penalties manifest differently across the privacy spectrum from central to local models. For CDP ($m=1$), the penalty is $\log^2 N$, whereas for LDP ($n=1$), it is only $\log N$.


\end{example}

Armed with Lemma \ref{lem:private-constrained-risk}, we proceed with the lower bound for the pointwise risk.

\subsubsection{The pointwise risk lower bound}\label{sec:pointwise-risk-lower-bound-subsection}

 Theorem \ref{thm:adaptation-lower-bound-pointwise} below provides a fundamental super-efficiency lower bound for pointwise density estimation under differential privacy that. Corollary \ref{cor:pointwise-adaptation-lb} distills the result to the adaptive setting, which combined with the upper bound in Theorem~\ref{thm:upper-bound-pointwise}, precisely characterizes the unavoidable cost of adaptation. It generalizes the lower bound result stated earlier in Theorem~\ref{thm:pointwise_cost_of_adaptation} and provides the technical foundation for understanding why the additional logarithmic adaptation penalty is inherent to the privacy constraint.

 \begin{theorem}\label{thm:adaptation-lower-bound-pointwise}
	Let $m,n\in\N$ and $N:=mn$. Fix a sequence $A_N\ge e$ and an $(\varepsilon,\delta)$-FDP estimator $\hat T$ with
	\begin{equation}\label{eq:delta-condition}
	  \delta \;\ll\; \frac{\varepsilon}{mn\,A_N}.
	\end{equation}
	
	Suppose there exists $f_0\in\cB^\alpha_{p,q}(R')$ with $R'<R$ and $f_0(t_0)>0$ such that
	\begin{equation}\label{eq:thm:adaptation-lower-bound-pointwise-assumption-clean}
	  \E_{f_0}\!\left(\hat T - f_0(t_0)\right)^2
	  \;\lesssim\; \frac{1}{A_N}\left\{\, N^{-\frac{2\nu}{2\nu+1}}
	  \,\vee\, (mn^2\varepsilon^2)^{-\frac{2\nu}{2\nu+2}} \right\}
	  =o(1).
	\end{equation}
	Then,
	\begin{equation}\label{eq:thm:adaptation-lower-bound-pointwise-conclusion}
	  \sup_{f\in\cB^\alpha_{p,q}(R)} \E_f\!\left(\hat T - f(t_0)\right)^2
	  \;\gtrsim \;  \left(\frac{N}{\log A_N}\right)^{-\frac{2\nu}{2\nu+1}}
	  \,\vee\, \left(\frac{mn^2\varepsilon^2}{L_{m,N}}\right)^{-\frac{2\nu}{2\nu+2}},
	\end{equation}
where $L_{m,N} := \log A_N \left( 1 \vee \frac{\log A_N}{m} \right) $.
	\end{theorem}

We highlight that the theorem is a consequence of Lemma \ref{lem:private-constrained-risk}, combined with arguments specific to the nonparametric setting which we postpone to Section \ref{sec:pointwise-risk-lower-bound-proofs} of the Supplementary Material. The theorem establishes that any differentially private estimator that performs well at a specific function $f_0$ (better than the minimax rate by a factor of $A_N$) must necessarily perform worse on some other functions in the Besov ball. The better the performance at $f_0$, the worse the performance on other functions, where the loss is a logarithmic factor in $A_N$.

The fundamental trade-off posed by Theorem \ref{thm:adaptation-lower-bound-pointwise} has direct implications for adaptation across smoothness classes. Consider an estimator that achieves near-optimal rates for functions in a smoother Besov class $\mathcal{B}_{p,q}^{\alpha}$. Since such an estimator necessarily performs much better than the minimax rate on functions that lie in the interior of a less smooth class $\mathcal{B}_{p',q'}^{\alpha'}$ (corresponding to a large improvement factor $A_N$), Theorem~\ref{thm:adaptation-lower-bound-pointwise} implies it must pay a logarithmic penalty when estimating other functions in the less smooth class. This adaptivity trade-off is formalized in the following corollary, which shows that achieving optimal rates up to a polylogarithmic factor for one smoothness class forces a logarithmic adaptation penalty when estimating functions from any less smooth class.

	\begin{corollary}\label{cor:pointwise-adaptation-lb}
		Consider an estimator $\hat{T} \in \cT^{\varepsilon, \delta}$ for $\delta \ll \varepsilon / \log(N)$ in the federated setting with $N = m n$ total samples across $m$ servers. Suppose that, for some $(\alpha,p,q)$ such that $\nu := \alpha - 1/p > 1/2$ and $g \in \cB_{p,q}^{\alpha,R}$ it holds that
		\begin{equation}\label{eq:optimal-upperbound}
		 \E_{g} ( \hat{T} - g(t_0) )^2 \lesssim (\log N)^{O(1)} \left( \left(\frac{1}{N}\right)^{\frac{2\nu}{2\nu + 1}} + \left(\frac{1}{m n^2 \varepsilon^2}\right)^{\frac{2\nu}{2\nu + 2}} \right).
		\end{equation}
		Then, for any $(\alpha',p',q')$ such that $\nu' := \alpha' - 1/p' < \nu$, we have that
		\begin{equation}\label{eq:Besov_space_adaptation_lower_bound}
		\sup_{f \in \cB_{p',q'}^{\alpha',R}} \E_{f} ( \hat{T} - f(t_0) )^2 \gtrsim \left(\frac{N}{\log N}\right)^{-\frac{2\nu'}{2\nu' + 1}} + \left(\frac{m n^2 \varepsilon^2}{L_{m,N}}\right)^{-\frac{2\nu'}{2\nu' + 2}},
		\end{equation}
		where $L_{m,N}$ is as defined in \eqref{eq:elbow-effect-factor}.
		\end{corollary}
	
\begin{remark}
	Corollary \ref{cor:pointwise-adaptation-lb} shows that the cost of adaptation is at least $\log N$ in the case of the Besov classes $\cB^{\alpha_1}_{p_1,q_1}$ and $\cB^{\alpha_2}_{p_2,q_2}$ with $\alpha_1 - 1/p > \alpha_2 - 1/p$. Under DP, the phenomenon of pointwise adaptive estimation where one can `trade-off' regularity versus integrability of derivatives remains like in the non-private case (see Theorem 1 of \cite{cai2003pointwiseBesov}): Whenever $\alpha_1 - 1/p = \alpha_2 - 1/p$, no adaptive cost is paid. For example, adapting between a $2$-smooth Sobolev ball and a $3/2$-smooth H\"older ball is possible without adaptive cost for the pointwise risk. 
\end{remark}



\subsection{Adaptation lower bound for the global risk}\label{sec:global-lower-bound}

We now turn to the lower bound for global risk adaptation. The theorem below establishes that any $(\varepsilon,\delta)$-FDP estimator attempting to adapt across any range of smoothness levels $(\alpha_{\min}, \alpha_{\max})$ must incur the logarithmic penalty $\log N$ in the privacy term, uniformly across all smoothness levels. Unlike pointwise estimation, global risk exhibits a uniform logarithmic adaptation penalty: there is no benefit from distributing data across many servers.

\begin{theorem}\label{thm:global-lower-bound}
	Assume $\delta \ll n \varepsilon^2 / N$. Consider any $\alpha_{\min} > 1/2$ and $\alpha_{\max} > \alpha_{\min}$ and let 
	\begin{equation}\label{eq:adaptive-global-rate}
		\rho^2(\alpha) \equiv \rho^2_{\alpha,m,n,\varepsilon} = \left( \frac{1}{N} \right)^{\frac{2\alpha}{2\alpha + 1}} + \left( \frac{\log N}{m n^2 \varepsilon^2} \right)^{\frac{2\alpha}{2\alpha + 2}}.
	\end{equation}
	 
	Then,
	\begin{equation}\label{eq:global-lb-to-show-thm}
	\inf_{\hat{f} \in \cF(\varepsilon,\delta)} \sup_{ \alpha \in (\alpha_{\min}, \alpha_{\max})} \, \sup_{f \in \mathcal{B}_{p,q}^\alpha(R)} \, \mathbb{E}_f \| \hat{f} - f \|_2^2 \, \rho(\alpha)^{-2} \gtrsim 1.
	\end{equation}
\end{theorem}

Although the estimators for both risk types are similar, the global risk setting exhibits fundamentally different behavior and requires fundamentally different techniques from the pointwise case. In the non-private setting, adaptation for global risk can be achieved without cost. Capturing the cost of adaptation under privacy constraints through a lower bound thus requires a novel argument. 


Before giving a formal proof below, we sketch the argument of our technique. The central idea is to exploit the fact that under FDP constraints, the `Fisher information induced by an FDP protocol' within a finite-dimensional sub-model is, for well-behaved sub-models, dimension-free. This phenomenon was first observed in a non-adaptive LDP setting by \cite{barnes2020fisher}, and subsequently in non-adaptive FDP settings in \cite{Cai2024FL-NP-Regression,xue2024optimalestimationprivatedistributed}. While in these earlier works this property was used to characterize the cost of privacy for fixed regularity, the adaptive setting requires a more delicate construction: it is precisely this dimension-free structure, hitherto unexploited in this context, that enables our rate-optimal lower bound.

The first part of the construction is familiar (see e.g. \cite{spokoiny_adaptive_1996}): we formulate a sub-model that partitions into further sub-models, each indexed by a resolution level $L$ and consisting of random wavelet perturbations at that level. Due to the multiscale nature of wavelets, we can consider a grid of regularity values with cardinality of order $\log N$, where each regularity value corresponds to a different component of the partition. We then define a prior over the perturbations to ensure that the resulting densities belong to the appropriate Besov balls almost surely.


This construction allows us to invoke the van Trees inequality, which relates the `adaptive Bayes risk' to the \emph{minimum} transcript-induced Fisher information across the partition (Lemma \ref{lem:global-lower-bound-key-lemma}). Intuitively, successful adaptation requires the protocol to retain sufficient information for every sub-model. While this relationship holds generally (that is; we have hitherto not invoked privacy specific arguments), the privacy constraint introduces a bottleneck: the \emph{total} induced Fisher information is bounded by a dimension-free quantity (Lemma \ref{lem:data-processing}). With approximately $\log N$ sub-models sharing this limited information budget, the signal available for any single smoothness level is diluted. This dilution manifests as the $\log N$ penalty—the unavoidable cost of adaptation under privacy constraints.



\begin{proof}[Proof of Theorem \ref{thm:global-lower-bound}]

	Let $\cA \subset (\alpha_{\min}, \alpha_{\max})$ such that for all $\alpha \in \cA$,
	\begin{equation}\label{eq:privacy-rate-dominates-global-lb}
		\rho(\alpha) \leq 2\left( \frac{\log (N)}{m n^2 \varepsilon^2} \right)^{-\frac{\alpha}{2\alpha + 2}} \iff \rho(\alpha)^{-2} \geq \frac{1}{4} \left( \frac{\log (N)}{m n^2 \varepsilon^2} \right)^{\frac{2\alpha}{2\alpha + 2}} =: \tilde{\rho}_\alpha^{-2}.
	\end{equation}
	If the complement of $\cA$ in $(\alpha_{\min}, \alpha_{\max})$ is non-empty,
	\begin{equation*}
		\inf_{\hat{f} \in \cF(\varepsilon,\delta)} \sup_{ \alpha \in \cA^c} \, \sup_{f \in B_{pq}^\alpha(R)} \, \E_f \,  {N} ^{\frac{2\alpha}{2\alpha + 1}} \, \| \hat{f} - f \|_2^2 
	\end{equation*}
	lower bounds \eqref{eq:global-lb-to-show-thm}, and the latter quantity can be further lower bounded by a constant following standard arguments (e.g., using Assouad's lemma or Fano's inequality (ignoring privacy constraints); see \cite{tsybakov2009introduction}, Chapter 2). In case $\cA$ is empty as $N \to \infty$, the statement of the theorem follows.

	Assume next that $\cA$ is not empty. Since $\alpha \mapsto \rho(\alpha)$ is continuous, we can take $\cA$ such that there exists an open neighborhood $\cA_* \subset \cA$ for which \eqref{eq:privacy-rate-dominates-global-lb} holds for all $\alpha \in \cA_*$. Without loss of generality, write ${\cA_*} = (\alpha_{\min}, \alpha_{\max})$ and consider an (approximately) equispaced grid $\tilde{\cA}$ of $(\alpha_{\min}, \alpha_{\max})$ of size $\lceil \log N \rceil$; such that for $\alpha_1, \alpha_2 \in \tilde{\cA}$ with $\alpha_1 > \alpha_2$, we have $\alpha_1 - \alpha_2 \geq 1/ (2 \log N)$. 

	For each $\alpha \in \tilde{\cA}$, let $L_\alpha$ be the integer closest to the solution of $ 2^{L(\alpha+1)} = \tilde{\rho}^{-1}_\alpha$ and consider the set $\cL := \{ L_\alpha : \alpha \in \tilde{\cA} \}$. Recall that $N^{-\omega} \lesssim \varepsilon \lesssim 1$ for some $\omega \in [0,1)$, which implies that $|\cL| \asymp \log N$ and there exists a constant $c > 0$ such that $L \geq c \log N$ for all $L \in \cL$. 

	Consider now the random $L_2[0,1]$-valued elements 
	\begin{equation}\label{eq:global-lb-submodel-def}
		 F_\cL^U(x) = 1 + \sum_{L \in \cL} \sum_{k = 1}^{2^L} U_{Lk} \psi_{Lk}(x) \quad x \in [0,1],
	\end{equation}
	where the collection $U = \left\{U_{L} : L \in \cL \right\}$ of coefficient vectors $U_L = (U_{Lk})_{k=1}^{2^L}$ is drawn from a `prior' distribution supported on the hypercube 
	\begin{equation*}
		\cU = \cup_{\alpha \in \tilde{\cA}} \, [-C_R 2^{-L_\alpha(\alpha+1/2)}, C_R 2^{-L_\alpha(\alpha+1/2)}]^{2^{L_\alpha}}.
	\end{equation*}
	Here, $\psi_{Lk}$ are the elements of a wavelet basis that is $A> \alpha_{\max}$-smooth and satisfies $\int \psi_{Lk} = 0$ for all $k$ and $L \in \cL$. Since $\alpha > 0$ and $\min_{L \in \cL} L \to \infty$, $F_\cL$ is a probability density over the full support of the prior whenever $m n^2 \varepsilon^2$ is large enough. The following lemma establishes a lower bound on adaptive minimax risk in terms of the trace of the Fisher information of the submodels induced by $F_\cL^U$, $U \in \cU$.

	\begin{lemma}\label{lem:global-lower-bound-key-lemma}
		Consider $\hat{f} \in \cF(\varepsilon,\delta)$, let $T = (T^{(j)})_{j=1,\dots,m}$ denote the corresponding FDP transcripts. 
		
		There exists a distribution $\P^U$ of $U$ over $\cU$ such that \eqref{eq:global-lb-to-show-thm} is lower bounded by 
		\begin{equation*}
			 \left({  \, \frac{\log (N)}{m n^2\varepsilon^2} \, \min_{\alpha \in \tilde{\cA}} \, \E^{U} \text{Tr}(\cI_{L_\alpha}(U)) + 1}\right)^{-1},
		\end{equation*}
		where $\cI_{L_\alpha}(U) = \E_{F_\cL^U} \E_{F_\cL^U} \left[ S_{L_\alpha} | T \right] \E_{F_\cL^U} \left[ S_{L_\alpha} | T \right]^\top$, with
		\begin{equation*}
			S_L \equiv S_L(X_1,\dots,X_N) := \nabla_{(u_{Lk})_{k=1}^{2^L}}  \sum_{j=1}^m \sum_{i=1}^n \log F_\cL^u(X_i^{(j)}) \mid_{u=U} \,,
			\end{equation*}
		and where $\E^U$ denotes the push-forward expectation with respect to the distribution $\P^U$.
	\end{lemma}

	A proof of the above lemma can be found in Section \ref{sec:global-risk-lower-bound-proofs} of the Supplementary Material. To see that the object $\cI_{L_\alpha}(U)$ deserves the name \emph{transcript-induced Fisher information matrix} and is in fact well-defined, we refer the reader to Section \ref{sec:properties-of-induced-fisher-information-matrix}. 
		
	 We have that  
				\begin{align}
				\min_{\alpha \in \tilde{\cA}} \, \E^{U}  \text{Tr}(\cI_{L_\alpha}(U)) &\leq 	\frac{1}{|\cL|} \sum_{L \in \cL} \E^{U} \text{Tr}(\cI_{L_\alpha}(U)) = \frac{1}{|\cL|} \E^{U} \text{Tr} \left( \cI_{\cL}(U) \right), \label{eq:stacking-fisher-infos}
				\end{align}
				where $\cI_{\cL}(U)$ denotes the `full' induced Fisher information matrix corresponding to the model in \eqref{eq:global-lb-submodel-def}: 
				\begin{equation*}
					\cI_{\cL}(U) = \E_{F_\cL^U} \E_{F_\cL^U} \left[ S_{\cL} | T \right] \E_{F_\cL^U} \left[ S_{\cL} | T \right]^\top,
				\end{equation*}
				where $S_{\cL} = \text{vec}(S_L : L \in \cL)$ is the vector obtained by stacking the vectors $S_L$ for $L \in \cL$. 
		
				The proof now follows by combining the observation of \eqref{eq:stacking-fisher-infos} with the following data-processing inequality, which captures the rather striking phenomenon that the trace of the transcript-induced Fisher information matrix is free of dimension in the differential privacy setting.
		
				\begin{lemma}\label{lem:data-processing}
					Whenever $\delta  \ll n^{} \varepsilon^2 / N $, it holds that $\text{Tr}(\cI_{\cL}(u)) \leq C m n^2 \varepsilon^2$ for some absolute constant $C > 0$ independent of $n$, $m$, $\delta$ and $\varepsilon$.
					\end{lemma}
		
		We provide a proof of this lemma in Section \ref{sec:global-risk-lower-bound-proofs} of the Supplementary Material \cite{cai2025supplementary}.
		\end{proof}

\section{Discussion}\label{sec:discussion}
	
	The results presented in this paper provide a unified and comprehensive framework for understanding the cost of adaptation in federated differentially private (FDP) density estimation. Our analysis shows that privacy constraints introduce an intrinsic penalty to adaptive estimation\textemdash a subtle yet fundamental cost that distinguishes the private setting from its classical, non-private counterpart. Although logarithmic in magnitude (and thus modest compared with polynomial penalties), this adaptation cost is conceptually significant: it reflects a fundamental limitation imposed by privacy protection and has practical implications for valid statistical inference beyond point estimation. These findings open several avenues for further investigation in related statistical models and problem settings. For instance, understanding the exact cost of adaptation is essential for constructing adaptive confidence intervals or confidence bands with guaranteed coverage, as the adaptation penalty directly determines the required bandwidth \cite{Cai2004adaptation, gine2010confidence, Cai2014adaptive}. Extending our framework to such inferential tasks would clarify how privacy noise interacts with the uncertainty inherent in adaptation.

Beyond density estimation, the principles and tools developed here\textemdash particularly our lower bound techniques, noise mechanism design, and adaptive thresholding strategy\textemdash have the potential to inform a wide range of nonparametric and high-dimensional problems. In settings such as sparse regression or function estimation under shape constraints, adaptation to unknown sparsity or structural parameters poses analogous challenges. Applying or extending our methods in these contexts could lead to new privacy-preserving procedures that attain minimax-optimal or near-optimal adaptive performance.

Another promising direction concerns adaptation to more complex functional characteristics, such as spatial inhomogeneity, anisotropy, or locally varying smoothness. These features arise frequently in real-world applications and introduce dependencies that may amplify the privacy-adaptivity trade-off. Understanding whether such structure-dependent adaptation leads to more substantial costs under differential privacy could deepen our insight into statistical efficiency under privacy constraints.

Our findings also have implications for statistical tasks beyond estimation, including hypothesis testing and model assessment. Although recent work has shown that adaptive nonparametric goodness-of-fit testing under differential privacy is feasible \cite{lam2022minimax, cai2024privateTesting}, the fundamental cost of adaptation in these testing problems remains largely open. Determining whether similar penalties arise\textemdash or whether privacy induces different trade-offs\textemdash would be an important direction for future research.

Taken together, our framework highlights the broader potential for studying adaptation under privacy constraints across diverse statistical models. By systematically characterizing the interplay between privacy, adaptability, and statistical efficiency, this work deepens our theoretical understanding of privacy-utility trade-offs and lays the groundwork for developing principled, adaptive, and privacy-preserving methodologies for modern data analysis.
 


%
%

\begin{funding}
The research of Tony Cai was supported in part by NSF grant NSF DMS-2413106 and NIH grants
R01-GM123056 and R01-GM129781.
\end{funding}


\begin{supplement}

    \stitle{Supplementary Material to ``{The} Cost of Adaptation under Differential Privacy: Optimal Adaptive Federated Density Estimation''}
    
    \sdescription{The supplementary material \cite{cai2025supplementary} contains proofs of the main results, technical lemmas, and additional details on the construction of the estimators and lower bounds.}

    \end{supplement}

\putbib[references]
\end{bibunit}

\newpage



\begin{bibunit}[imsart-number]

\begin{frontmatter}
\title{Supplementary Material to: ``The Cost of Adaptation under Differential Privacy: Optimal Adaptive Federated Density Estimation''}
\runtitle{The Cost of Adaptation under Differential Privacy}

\begin{aug}

\author[A]{\fnms{T. Tony} \snm{Cai}\ead[label=e1]{tcai@wharton.upenn.edu}},
\author[B]{\fnms{Abhinav} \snm{Chakraborty}\ead[label=e2]{ac4662@columbia.edu}}
\and
\author[C]{\fnms{Lasse} \snm{Vuursteen}\ead[label=e3]{lv121@duke.edu}}

\address[A]{Department of Statistics and Data Science, \\
\printead[presep={The Wharton School, University of Pennsylvania,\ }]{e1}}

\address[B]{Department of Statistics, \\
\printead[presep={Columbia University,\ }]{e2}}

\address[C]{Department of Statistical Science,\\
\printead[presep={Duke University,\ }]{e3}}


\runauthor{T. T. Cai, A. Chakraborty and L. Vuursteen}
\end{aug}
\begin{abstract}
    The supplementary material contains proofs of the main results, technical lemmas, and additional details on the construction of the estimators and lower bounds \cite{cai2025adaptation}.
\end{abstract}
\end{frontmatter}

\appendix

\section{Proofs relating to estimator guarantees: Theorems \ref{thm:upper-bound-global-risk} and \ref{thm:upper-bound-pointwise}}

\subsection{Multiscale-oscillation norm properties}\label{sec:multiscale-oscillation-norm-properties}

We first recall some notation. Throughout consider for $l_0, L \in \N$, $V_L = \{ (l,k) \, : \, l=l_0,\dots,L, \, k = 1,\dots,2^l \}$. Let $s_L = \sum_{l=l_0}^L 2^l$. Define for $g \in L_2[0,1]$ the \emph{oscillation} as $\operatorname{osc}(g) := \sup_{t \in [0,1]} g(t) - \inf_{t \in [0,1]} g(t)$. Define the \emph{multiscale-oscillation norm} on the space of functions spanned by the wavelets $\psi_{lk}$ as follows:
\begin{equation}\label{eq:osc-norm-supplement}
    \|u\|_{V_L} = \sup_{\substack{g \in \operatorname{span}\{\psi_{lk} : (l,k) \in V_L\} \\ \operatorname{osc}(g) \leq 1}} \left| \left\langle \sum_{(l,k) \in V_L} u_{lk} \psi_{lk}, g \right\rangle_{L^2[0,1]} \right|,
\end{equation}
where $f = \sum_{(l,k) \in V_L} u_{lk} \psi_{lk}$.

Given a vector $x\in\R^{s_L}$ indexed as $x=(x_{ik} : (i,k) \in V)$, let $x_{B} = (x_{ik} : (i,k) \in B)$ for any $B \subset V$. The remainder of this section is devoted to the proof of the following concentration result for random vector generated from the exponential mechanism corresponding to the multiscale-oscillation norm defined in \eqref{eq:osc-norm}; in the form of the lemma below.

\begin{lemma}\label{lem:osc-exponential-mechanism-tails}
    Consider for $j=1,\dots,m$, i.i.d. draws $W^{(j)} \in\R^{s_L}$ with density proportional to $w\mapsto \exp(-\theta\|w\|_{V_L})$ with $\theta>0$ and set $\bar{W} = \frac{1}{m} \sum_{j=1}^m W^{(j)}$. Fix $l\in\{l_0,\dots,L\}$ and $S\subset\{1,\dots,2^l\}$ with $|S|=:b_L$, and write $B_l=\{(l,k):k\in S\}$.
    Assume $b_L/2^L<1$. 
    
    Then, there exist constants $\kappa,c_0,c_1,c_2>0$ depending only on the mother wavelet ($A,\|\psi\|_\infty$) such that, for all
    \begin{equation*}
    t\;\ge\; c_0\,\frac{2^{l/2}}{\sqrt{m} \theta} ,
    \end{equation*}
    \begin{equation*}
    \mathbb{E}\!\left[\|\bar{W}_{B_l}\|_2^2\,\mathbbm{1}\{\|\bar{W}_{B_l}\|_2\ge t\}\right]
    \lesssim  \left(t^2 + \frac{1}{m \theta^2} \right) e^{ b_L \log(5) - c_1 m \min\left( \frac{t^2 \theta^2}{2^l }, \frac{t \theta}{2^{l/2} } \right) }+ \frac{ 2^{4L}}{m \theta^2} e^{-c_2 2^L}.
    \end{equation*}
    In particular, it holds that for some constant $C>0$,
    \begin{equation*}
    \mathbb{E}\! \|\bar{W}_{B_l}\|_2^2 
    \leq  C \frac{2^{l}b_L}{m \theta^2}.
    \end{equation*}
\end{lemma}

The proof of this lemma, which we defer to the end of this section, relies on multiple auxiliary results. The distribution of the exponential mechanism induced by the norm \eqref{eq:osc-norm} is not straightforward to analyze. To this end, we first introduce two auxiliary lemmas that allow us to analyze.

The first of which decomposes as a Gamma distributed `radius' and a `direction' uniformly distributed on the set $K$. See \cite{hardt2010knorm} for details. This decomposition into a `radial' and `angular' part proves to be crucial for the analysis of the tail behavior of the noise of the multiscale oscillation-norm defined in \eqref{eq:osc-norm}. Given a set $K \subset \R^{s_L}$ of finite volume, let $\cU(K)$ denote the uniform distribution on $K$.

\begin{lemma}\label{lem:K-norm-mechanism}
    Let $\|\cdot\|_K$ be a norm on $\R^d$ and let $K$ denote its unit ball. Let $W^K$ be a random vector with density $\varphi_K$ is proportional to $w \mapsto \exp(-\frac{\varepsilon}{\Delta_K} \|w\|_K)$ with respect to the Lebesgue measure. 
    
    Then, 
    \begin{equation*}
        W^K = D U, \quad \text{ where } \quad D \sim \Gamma\left(d + 1, \frac{\Delta_K}{\varepsilon}\right) \text{ and } U \sim \cU(K).
    \end{equation*} 
\end{lemma}

With Lemma \ref{lem:K-norm-mechanism} in hand, it is easy to derive a bound on the second moment of the exponential mechanism corresponding to the multiscale-oscillation norm. However, in light of Lemma \ref{lem:private-block-thresholding-oracle}, the quantity we are looking to bound is $\mathbb{E}\!\left[\|W_{B_l}\|_2^2\,\mathbbm{1}\{\|W_{B_l}\|_2\ge t\} \right]$, not just the second moment. Lemma \ref{lem:K-norm-mechanism}, combined with a union bound and a straightforward tail bound for the Gamma distribution (Lemma \ref{lem:sub-exponentiality-parameter-gamma}) allow us to focus our efforts onto the concentration of a uniform draw from the unit ball in the multiscale-oscillation norm. Whilst the geometry of the multiscale-oscillation norm ball is not straightforward in and of itself, the following well known result (see e.g. \cite{kanter1977unimodality}) allows us to compare the marginal distributions of a random variable uniformly distributed on the unit ball in the multiscale-oscillation norm with the marginal distributions of a random variable uniformly distributed on a centrally symmetric convex superset whose marginal distributions are easier to analyze.

A set $K \subset \R^d$ is (origin) \emph{centrally symmetric} if $x \in K \iff -x \in K$. Recall also that a random variable $X$ is said to \emph{stochastically dominate} another random variable $Y$ if $\mathbb P\{X \geq t\} \geq \mathbb P\{Y \geq t\}$ for all $t$. 

\begin{lemma}\label{lem:stochastic-domination-convex-bodies}
    Let $d\in\N$ and let $K,K'\subset\R^d$ be centrally symmetric convex bodies with $0<\text{Vol}(K),\text{Vol}(K')<\infty$ and $K'\subset K$.
    For $m\in\N$, take $U^{(1)},\dots,U^{(m)}\iid\sim \cU(K)$ and
    ${U'}^{(1)},\dots,{U'}^{(m)}\iid\sim \cU(K')$, and set
    $\bar U=\frac1m\sum_{j=1}^m U^{(j)}$, $\bar U'=\frac1m\sum_{j=1}^m {U'}^{(j)}$. Then, for any subspace $V\subset\R^d$ and all $t>0$,
    \begin{equation*}
        \P \!\left(\,\|\pi_V \bar U\|_2 \le t\,\right)
    \;\le\;
    \P\!\left(\,\|\pi_V \bar U'\|_2 \le t\,\right),
    \end{equation*}
    i.e. $\|\pi_V \bar U\|_2$ stochastically dominates $\|\pi_V \bar U'\|_2$.
    \end{lemma}
\begin{proof}
    Let $\mu, \mu'$ denote the uniform distributions on $K$ and $K'$ respectively. As $K' \subset K$, we have that for any centrally symmetric and convex set $C \subset \R^d$, 
    \begin{equation*}
        \mu(C) = \frac{\operatorname{Vol}(K \cap C)}{\operatorname{Vol}(K)} \leq \frac{\operatorname{Vol}(K' \cap C)}{\operatorname{Vol}(K')} = \mu'(C),
    \end{equation*}
    meaning that $\mu'$ is more peaked than $\mu$ in the sense of \cite{kanter1977unimodality}. Corollary 3.2 of \cite{kanter1977unimodality} then states that $ \mu' \times \mu'$ is more peaked than $\mu \times \mu$. 
    The result follows by applying the above reasoning to the $m$-product of the uniform measure on the sets
    \begin{equation*}
        C_t = \left\{ (x_1,\dots,x_m) \in (\R^{d})^m : \| \pi_V m^{-1} \sum_{j=1}^m x_j \|_2 \leq t \right\}, \quad t > 0,
    \end{equation*}
    for which is straightforward to verify that each $C_t$ is centrally symmetric and convex. 
\end{proof}

The following lemma provides a concentration result for the marginal distribution of a random variable uniformly distributed on the unit ball in the multiscale-oscillation norm.

\begin{lemma}\label{lem:osc-ball-concentration}
    Let $K_{V_L}$ denote the unit ball in the multiscale-oscillation norm defined in \eqref{eq:osc-norm}. Let $U^{(1)},\dots,U^{(m)} \iid \sim \cU(K_{V_L})$. Let $\bar{U} = \frac{1}{m} \sum_{j=1}^m U^{(j)}$. Consider $l \in \{ l_0,\dots,L \}$ and $S \subset \{1,\dots,2^l\}$ of size $|S| := b_L \in \N$ and let $B_l = \{ (l,k) : k \in S \}$.
    
    Whenever $b_L/2^L < 1$, there exists a constant $c > 0$ that depends only on the wavelet basis such that
    \begin{equation*}
        \P\left( \|\bar{U}_{B_l}\|_2 \geq t \right) \leq 2 \exp\left( b_L \log(5) - c m \min\left( \frac{2^{2L} t^2}{2^{l}}, \frac{2^L t }{2^{l/2}} \right) \right).
    \end{equation*}
\end{lemma}
\begin{proof}
Given the compactly supported (detail) wavelets $\{\psi_{lk} : (l,k) \in V_L\}$, define for $l_1 \leq l_2$ the conflict graphs $ G_{l_1,l_2} = (V_{l_1,l_2}, E) $ with vertex set $ V_{l_1,l_2} = \cup_{l=l_1}^{l_2} \{ (l,k) : k = 1, \dots, 2^l \} $, with edge $v_{(l,k)}v_{(l',k')} \in E$ if and only if $\supp(\psi_{l_1 k_1}) \cap \supp(\psi_{l_2 k_2}) \neq \emptyset$. The detail wavelets up until resolution level $L$ constitute the conflict graph $ G_{l_0,L} = (V_{L}, E) $. A \emph{clique} in this graph is a set of vertices such that all pairs of vertices are connected by an edge. 

Consider the following claim. 

\textbf{Claim: } There exists a constant $C_0$ depending only on the support of the mother wavelet, for all $l$ there exists an induced subgraph $G'_{l_0,L} \subset G_{l_0,L}$ with maximal clique size at most $C_0$, that also contains the induced subgraph with vertices indexed by $B_l$ and contains at least $2^L$ vertices in total.

To prove the claim, note that the number of overlapping wavelets at any resolution level $l$ is bounded by a fixed constant $A_0$ depending only on the support of the mother wavelet. We separate two cases based on the resolution level $l$. 

\begin{itemize}
    \item Case 1: $l = L$; each vertex in $G_{L,L}$ has at most $A_0$ neighbors in the same level, so the maximal clique size is at most $C_0 = A_0$. So we can $C = C_0$ and $c=1$ to satisfy the claim. 
    \item Case 2: $l < L$; by the same argument as in Case 1, the maximal clique size of $G_{l,l}$ is at most $A_0$. The support of a wavelet at level $l$ contains at most a clique of size $2A_0$ at level $L$. Hence, the induced subgraph $G_{l,l} \cup G_{L,L}$ has maximal clique size at most $C_0 = 2A_0^2$. 
\end{itemize}
Taking the maximum of $A_0$ and $2A_0^2$ gives the claim.

Given the subgraph $G'_{l_0,L}$ and $u \in \R^{s_L}$, consider the function 
\begin{equation*}
    g(x) = \sum_{(l,k) \in G_{l_0,L}'} \operatorname{sign}(u_{lk}) \frac{\psi_{lk}(x)}{C_0 \|\psi\|_\infty 2^{l/2+1}}.
\end{equation*}
The function $g$ is a linear combination of the wavelets $\psi_{lk}$ with $\operatorname{osc}(g) \leq 1$, since there are at most $C_0$ non-zero wavelets $\psi_{lk}$ in the sum, with $\|\psi_{lk}\|_\infty \leq 2^{l/2}\|\psi\|_\infty$. Hence, any $u \in \{x : \|x\|_{V_L} \leq 1\}$ satisfies 
\begin{equation*}
    u \in \left\{ u \, : \, |u_{lk}| \leq 1, \sum_{(l,k) \in G_{l_0,L}'} \frac{|u_{lk}|}{C_0 \|\psi\|_\infty 2^{l/2+1}} \leq 1 \right\}.
\end{equation*}
Let $U'^{(1)},\dots,U'^{(m)}$ be independent uniform draws from the subset above. Let $\bar{U}' = \frac{1}{m} \sum_{j=1}^m U'^{(j)}$. By Lemma \ref{lem:stochastic-domination-convex-bodies}, the marginal distribution of $\|\bar{U}'_{B_l}\|_2$ stochastically dominates that of $\|\bar{U}_{B_l}\|_2$. Furthermore, since $B_l$ is a subset of the vertices of $G'_{l_0,L}$, the constraint $\sum_{(l,k) \in G_{l_0,L}'} \frac{|u_{lk}|}{C_0 \|\psi\|_\infty 2^{l/2+1}} \leq 1$ implies that the $N \geq (c2^L) \vee 1$ random variables
\begin{equation*}
    \left\{ \frac{U_{lk}'^{(j)}}{C_0 \|\psi\|_\infty 2^{l/2+1}} \, : \, (l,k) \in G_{l_0,L}' \right\} 
\end{equation*}
are symmetric, mean zero such that $U_{lk}'^{(j)}$ Dirichlet distributed random variables weight parameters $c_l := C_0 \|\psi\|_\infty 2^{l/2+1}$, independently for each $j=1,\dots,m$. Consider a $1/2$-net of the unit sphere in $\R^{b_L}$, denoted by $\mathcal{N}$. We have that $|\mathcal{N}| \leq 5^{b_L}$ and $\|\bar{U}'_{B_l}\|_2 \leq 2 \max_{v \in \mathcal{N}} \langle \bar{U}'_{B_l}, v \rangle$. Hence, by a union bound,
\begin{equation*}
    \P\left( \|\bar{U}'_{B_l}\|_2 \geq t \right) \leq 5^{b_L} \P\left( \langle \bar{U}'_{B_l}, v \rangle \geq t/2 \right).
\end{equation*}
writing $s = \frac{t}{2 C_0 \|\psi\|_\infty 2^{l/2+1}}$, the latter display equals 
\begin{equation*}
    5^{b_L} \P\left( \frac{1}{m} \sum_{j=1}^m \sum_{i=1}^{b_L} v_i \zeta^{(j)}_i \beta_i^{(j)} \geq s \right) 
\end{equation*}
where for $j=1,\dots,m$, $i=1,\dots,b$, $\zeta^{(j)}_i \sim \text{Rad}(1/2)$ independent and $(\beta^{(j)}_i, 1 - \beta^{(j)}_i)$ Dirichlet $(b,N-b)$ random variables, independent for $j=1,\dots,m$. The proof is finished by recalling that $N \geq c 2^L$ and applying Lemma \ref{lem:sub-exponential-zeta-beta-product}.
\end{proof}

\begin{remark}\label{rmk:approximate-sampling-oscillation-norm}
    The noise vector $V_{s_{L^*}} = (V_{lk})_{(l,k) \in V_{L^*}}$ can be efficiently sampled by exploiting the decomposition in Lemma~\ref{lem:K-norm-mechanism} and the stochastic domination technique in Lemma~\ref{lem:stochastic-domination-convex-bodies}. Specifically, an approximation to the noise can be generated by sampling the Gamma-distributed radial component and the a uniformly distributed direction component dominating on the unit ball of the multiscale-oscillation norm.
\end{remark}

Combining this lemma with tail bounds for the Gamma distribution and the concentration results of Lemma \ref{lem:osc-ball-concentration}, we can now provide the required tail bounds for the noise of the exponential mechanism corresponding to the multiscale oscillation-norm.

\begin{proof}[Proof of Lemma \ref{lem:osc-exponential-mechanism-tails}]
    
    By Lemma \ref{lem:K-norm-mechanism}, we can decompose each random vector $W^{(j)}$ as $W^{(j)} = D^{(j)} U^{(j)}$, where $D^{(j)} \sim \Gamma(s_L + 1, 1/\theta)$ and $U^{(j)} \sim \cU(K_{V_L})$, with $D^{(j)}$ and $U^{(j)}$ independent.

    Under this decomposition $\|\bar{W}_{B_l}\|_2 \leq \max_{j} D^{(j)} \|\bar{U}_{B_l}\|_2$, and
\begin{equation*}
\1 \left\{ \|\bar{W}_{B_l}\|_2 \geq t \right\} \leq \1 \left\{ \max_{j} D^{(j)} \geq s \right\} + \1\left\{ \|\bar{U}_{B_l}\|_2 \geq t / s \right\}.
\end{equation*}
Multiplying by $\|\bar{W}_{B_l}\|_2^2$ and taking expectations, Cauchy-Schwarz bounds the first term as
\begin{equation}\label{eq:cs-split-noise-bound-1}
\sqrt{\E \|\bar{W}_{B_l}\|_2^4  \P\left( \max_{j} D^{(j)} \geq s \right)} \lesssim \frac{2^{4L}}{\theta^{2} m^{}}  \sqrt{\P\left( \max_{j} D^{(j)} \geq s \right)}. 
\end{equation}
By Lemma \ref{lem:sub-exponentiality-parameter-gamma}, $D^{(j)}$ is concentrated around its mean $\mu := (s_L + 1) / \theta$ where $s_L := |V_L|$, which combined with a union bound gives
\begin{equation}\label{eq:cs-split-noise-bound-2}
\P(\max_{j} D^{(j)} \geq (1 + \delta) \mu ) \leq  m e^{- c \delta^2 (s_L + 1)}
\end{equation}
for $0 < \delta < 1$, where $c > 0$ is a universal constant. Taking $s = (1 + \delta) \mu$, we proceed to bound $\E \|\bar{W}_{B_l}\|_2^2 \mathbbm{1}\left\{\|\bar{U}_{B_l}\|_2 \geq t / s \right\}$. By independence of the $D^{(j)}$'s with the $U^{(j)}$'s and an application of Cauchy-Schwarz, we have that 
\begin{equation*}
\E \|\bar{W}_{B_l}\|_2^2 \mathbbm{1}\left\{\|\bar{U}_{B_l}\|_2 \geq t / s \right\} \leq \frac{(s_L + 1)(s_L+2)}{\theta^2} \E \|\bar{U}_{B_l}\|_2^2 \mathbbm{1}\left\{\|\bar{U}_{B_l}\|_2 \geq t / s \right\}.
\end{equation*}
By Lemma \ref{lem:indicator-layer-cake}, we can bound the latter expectation as
\begin{equation}\label{eq:integral-layer-cake-split-noise-bound}
\frac{t^2}{s^2}\P(\|\bar{U}_{B_l}\|_2 \geq t / s) + 2 \int_{t/s}^\infty u \P(\|\bar{U}_{B_l}\|_2 \geq u) \, du,
\end{equation}
The factor (and integrand) $\P(\|\bar{U}_{B_l}\|_2 \geq u)$ is subsequently controlled by the tail bound of Lemma \ref{lem:osc-ball-concentration}; 
\begin{equation*}
\P(\|\bar{U}_{B_l}\|_2 \geq u) \lesssim \exp\left( b_L \log(5) - c_0 m \min\left( \frac{2^{2L} u^2}{2^{l}}, \frac{2^L u }{2^{l/2}} \right) \right).
\end{equation*}
Using that $s_L \asymp 2^L \implies s \asymp 2^L / \theta$, this yields
\begin{equation*}
\frac{(s_L + 1)(s_L+2)t^2}{\theta^2 s^2} \P(\|\bar{U}_{B_l}\|_2 \geq t / s) \lesssim t^2 \exp\left( b_L \log(5) - c m \min\left( \frac{t^2 \theta^2}{2^l }, \frac{t \theta}{2^{l/2} } \right) \right).
\end{equation*}
For the integral term in \eqref{eq:integral-layer-cake-split-noise-bound}, Lemma \ref{lem:exponential-minimum} yields that $\int_{t/s}^\infty u \P(\|\bar{U}_{B_l}\|_2 \geq u) \, du$ is bounded by
\begin{equation*}
 e^{b_L \log(5)} \left( \frac{1}{c_0 m \min(2^{2L} / 2^l, s 2^L /(t 2^{l/2}))} + \frac{1}{m^2 c_0^2 2^{L}} \right) e^{ - c_0 m \min( 2^{2L - l} (t/s)^2, 2^{L-l/2} (t/s))}.
\end{equation*}
Again using that $s \asymp 2^L / \theta$ and $t\;\gtrsim\,\frac{2^{l/2}}{\sqrt{m} \theta}$, the latter expression is bounded by a constant multiple of
\begin{equation*}
\frac{2^{l}}{m 2^{2L}}\exp\left( b_L \log(5) - c m \min\left( \frac{t^2 \theta^2}{2^l }, \frac{t \theta}{2^{l/2} } \right) \right).
\end{equation*}
Putting the above bounds together, we find 
\begin{equation*}
\E \|\bar{W}_{B_l}\|_2^2 \mathbbm{1}\left\{\|\bar{U}_{B_l}\|_2 \geq t / s \right\}  \lesssim \left(t^2 + \frac{2^{l}}{m \theta^2} \right) \exp\left( b_L \log(5) - c m \min\left( \frac{t^2 \theta^2}{2^l }, \frac{t \theta}{2^{l/2} } \right) \right).
\end{equation*}
Combining this with \eqref{eq:cs-split-noise-bound-1} and \eqref{eq:cs-split-noise-bound-2} gives the first statement of the lemma. For the second statement, apply the first statement with $t = c_0 \frac{b_L 2^{l/2}}{\sqrt{m} \theta}$ in order to find 
\begin{equation*}
\E \|\bar{W}_{B_l}\|_2^2 \lesssim \frac{2^l b_l}{m \theta^2} + \frac{2^{4L}}{m \theta^2} e^{-c_2 2^L},
\end{equation*}
and the statement follows since $2^{4L} e^{-c_2 2^L} \lesssim 1$ for $L \in \N$.
\end{proof}

We finish the section with the proof of Lemma \ref{lem:block-thresholding-tailbound-privacy-noise-only} and Lemma \ref{lem:term-thresholding-tailbound-privacy-noise-only}, which are direct consequences of Lemma \ref{lem:osc-exponential-mechanism-tails}. 

\begin{proof}[Proof of Lemma \ref{lem:block-thresholding-tailbound-privacy-noise-only}]
    We aim to apply Lemma~\ref{lem:osc-exponential-mechanism-tails}. Thus, $\bar{V}_{lB_{lj}}$ corresponds to $\bar{W}_{B_l}$ with $b_L = b_l$ and $\theta = \varepsilon n$. The condition $b_l / 2^{L^*} < 1$ holds since $b_l \asymp \log N$ and $2^{L^*} \asymp N$.
    
    Taking $t = \kappa \frac{2^{l/2} b_l}{\sqrt{m} n \varepsilon} = \kappa \frac{2^{l/2} b_l}{\sqrt{m} \theta}$, the condition $t \ge c_0 \frac{2^{l/2}}{\sqrt{m} \theta}$ holds if $\kappa \ge c_0 / b_l$. Since $b_l \ge 1$, choosing $\kappa \ge c_0$ suffices.
    
    By Lemma~\ref{lem:osc-exponential-mechanism-tails},
    \[
    \mathbb{E} \|\bar{V}_{lB_{lj}}\|_2^2 \mathbbm{1}\{\|\bar{V}_{lB_{lj}}\|_2 \ge t\} \lesssim \left(t^2 + \frac{1}{m \theta^2}\right) e^{b_l \log 5 - c_1 m \min\left( \frac{t^2 \theta^2}{2^l}, \frac{t \theta}{2^{l/2}} \right)} + \frac{2^{4L^*}}{m \theta^2} e^{-c_2 2^{L^*}}.
    \]
    The last term is $o\left( \frac{1}{m n^2 \varepsilon^2} \right)$ since $2^{L^*} \asymp N$ and $2^{4L^*} e^{-c_2 2^{L^*}} \to 0$ as $N \to \infty$.
    
    Now, $t^2 = \kappa^2 \frac{2^l b_l^2}{m \theta^2} \le \kappa^2 \frac{N b_l^2}{m \theta^2}$ (since $2^l \le 2^{L^*} \le N$). Also, $\frac{1}{m \theta^2} \le t^2$ for sufficiently large $\kappa$ (as $t \to \infty$ with $\kappa$). The exponent is $b_l \log 5 - c_1 \min(\kappa^2 b_l^2, \kappa b_l \sqrt{m})$. To ensure this is $\le -b_l$, we need $c_1 \min(\kappa^2 b_l^2, \kappa b_l \sqrt{m}) \ge b_l (\log 5 + 1)$. Choosing $\kappa \ge \max\left( c_0, \frac{\log 5 + 1}{c_1} \right)$ ensures this for all $m \ge 1$ and sufficiently large $N$ (hence large $b_l$). With this $\kappa$, the first term is $\lesssim t^2 e^{-b_l} \le \kappa^2 \frac{N b_l^2}{m \theta^2} \cdot \frac{1}{N} = \kappa^2 \frac{b_l^2}{m \theta^2}$. Plugging in $\theta^2 = n^2 \varepsilon^2$, this is $\kappa^2 \frac{b_l^2}{m n^2 \varepsilon^2}$. Absorbing $\kappa^2$ into $C$ yields the bound.
    \end{proof}

    \begin{proof}[Proof of Lemma \ref{lem:term-thresholding-tailbound-privacy-noise-only}]

        This follows immediately from Lemma~\ref{lem:osc-exponential-mechanism-tails} applied with $b_L = 1$, noting that the term $\frac{2^{4L}}{m \theta^2} e^{-c_2 2^L}$ is $o\left( \frac{1}{m \theta^2} \right)$ and can be absorbed into the constant in $\lesssim$, and setting $\theta = n \varepsilon$. The bound holds for $t \ge c_0 \frac{2^{l/2}}{\sqrt{m} n \varepsilon}$, where $c_0$ is the constant from Lemma~\ref{lem:osc-exponential-mechanism-tails}.
    
    \end{proof}

\subsection{Proof of Theorems \ref{thm:upper-bound-global-risk} and \ref{thm:upper-bound-pointwise}}\label{sec:proof-upper-bound-thms}

\begin{proof}
By Plancharel's theorem, we have that
\begin{align}\label{eq:global-risk-decomposition1}
    \E_f \| \hat{f}^{\textsc{BTPW}} - f \|_2^2 &= \sum_{k} \E_f (\hat{f}_{0k} - f_{0k})^2 + \sum_{l=l_0}^{L^*} \sum_{k=1}^{2^{l}} \E_f (\hat{f}_{lk} - f_{lk})^2 + \sum_{l > L^*} \sum_{k} f_{lk}^2 
\end{align}
The first sum on the right-hand side consists of $O(1)$ terms, each term bounded by
\begin{equation*}
    \text{Var}_f( \hat{\phi}_{lk} ) + \text{Var}_f(V_{0k}) \asymp \frac{\|\phi\|_\infty}{N} + \frac{1}{mn^2 \varepsilon^2}, 
\end{equation*}
where the variance of $V_{0k}$ bound follows from Lemma \ref{lem:osc-exponential-mechanism-tails}. Using Lemma \ref{lemma:coefficient_decay}, the third term in \eqref{eq:global-risk-decomposition1} is bounded by
\begin{equation}\label{eq:Lstar-tailbound}
    \sum_{l > L^*} \sum_{k} f_{lk}^2 \lesssim \sum_{l > L^*} 2^{-2l \alpha} \lesssim 2^{-2 L^* \alpha } \lesssim \frac{1}{N^{\frac{2\alpha}{2\alpha + 1}}},
\end{equation}
where the last inequality uses $L^* \asymp \log_2(N)$, $\alpha > 0$. Next, we turn to the second term in \eqref{eq:global-risk-decomposition1}. We have
\begin{align*}
    \sum_{l=l_0}^{L^*} \sum_{k=1}^{2^{l}} \E_f (\hat{f}_{lk} - f_{lk})^2 = \sum_{l=l_0}^{L^*} \sum_{j \in \cJ_l} \E_f \| \eta_{\tau_l} (\hat{f}_{lB_j}^{\textsc{PW}}) - f_{lB_j} \|_2^2, 
\end{align*}
where $f_{lB_j}$ denotes the vector of wavelet coefficients of $f$ in the block $B_j$ at resolution level $l$, and $\cJ_l$ is the set of all blocks at resolution level $l$. 

By the oracle inequality of Lemma \ref{lem:private-block-thresholding-oracle}, we have that 
\begin{equation}\label{eq:oracle-block-thresholding-applied}
    \E_f \| \eta_{\tau_l} (\hat{f}_{lB_j}^{\textsc{PW}}) - f_{lB_j} \|_2^2 \leq  \min \{ \|f_{lB_j}\|_2^2, 4\tau_l^2 \} + 4 \E_f \|Z_{lB_j}\|_2^2 \mathbbm{1}\{ \|Z_{lB_j}\|_2 > \tau_l \} ,
\end{equation}
where $Z_{lB_j} = (Z_{lk})_{(l,k) \in B_j}$ and $Z_{lk} = \hat{\psi}_{lk} - f_{lk} + V_{lk}$. We analyze the two terms in the display above separately, starting with the second term. As $\hat{\psi}_{lk} - f_{lk}$ is a sum of $N = m n$ i.i.d. mean zero random variables and $V_{lk}$ is the privacy noise, a union bound yields that the second term is bounded from above by
\begin{equation*}
     \E_f \|(\hat{\psi} - f_{\cdot \cdot})_{lB_j}\|_2^2 \mathbbm{1}\{ \|(\hat{\psi} - f_{\cdot \cdot})_{lB_j}\|_2 > \sqrt{2 \log N} \} + \E_f \|V_{lB_j}\|_2^2 \mathbbm{1}\{ \|V_{lB_j}\|_2 > \tau_l' \},
\end{equation*}
where we write $f_{\cdot \cdot} = (f_{lk})_{l,k}$. The first term is controlled by Lemma \ref{lem:non-private-block-thresholding-tailbound} below, which provides the bound
\begin{equation*}
    \E_f \|(\hat{\psi} - f_{\cdot \cdot})_{lB_j}\|_2^2 \mathbbm{1}\{ \|(\hat{\psi} - f_{\cdot \cdot})_{lB_j}\|_2 > \sqrt{2 \log N} \} \lesssim \frac{b_l}{N} e^{- c \log N}.
\end{equation*}
Lemma \ref{lem:osc-exponential-mechanism-tails} provides the second bound;
\begin{equation*}
    \E_f \|V_{lB_j}\|_2^2 \mathbbm{1}\{ \|V_{lB_j}\|_2 > \tau_l' \} \lesssim \frac{2^l b_l}{m n^2 \varepsilon^2} \exp\left(- 2 \log N \right) + \frac{2^{4L^*}}{m n^2 \varepsilon^2} e^{-c 2^{L^*}}.
\end{equation*}
Combining the above bounds and using that $2^{L^*} \leq N$, we obtain that
\begin{equation*}
    \sum_{l=l_0}^{L^*} \sum_{j \in \cJ_l} \E_f \|Z_{lB_j}\|_2^2 \mathbbm{1}\{ \|Z_{lB_j}\|_2 > \tau_l \} \lesssim \frac{1}{N} + \frac{1}{m n^2 \varepsilon^2}.
\end{equation*}
Using Lemma \ref{lemma:coefficient_decay}, we have that 
\begin{align*}
    \sum_{l = l_0}^{L^*} \sum_{j \in \cJ_l} \min \{ \|f_{lB_j}\|_2^2, 4\tau_l^2 \} &\leq \sum_{l = l_0}^{L} \sum_{j \in \cJ_l} 4 \tau_l^2 + \sum_{l > L}^{L^*} \sum_{k=1}^{2^l} \|f_{lk}\|_2^2 \\ 
    &\lesssim \frac{2^L}{N} + \frac{2^{2L} b_L}{ m n^2 \varepsilon^2} + 2^{-2L(\alpha + 1/2)},
\end{align*}
where for the first term on the right-hand side we used that there are $O(2^l)/b_l$ blocks at resolution level $l$, with $l \mapsto b_l$ increasing, and the bound on the second term follows in similar fashion as \eqref{eq:Lstar-tailbound}. Solve \(\frac{2^L}{N} \vee \frac{2^{2L} \log N}{m n^2 \varepsilon^2} \asymp 2^{-2L \alpha}\), yielding \(L \asymp \frac{1}{2\alpha+1} \log_2 (N / \log N)\) or \(L \asymp \frac{1}{2\alpha+2} \log_2 (m n^2 \varepsilon^2 / \log^2 N)\), depending on whether the private threshold dominates the non-private threshold; $\frac{b_L}{N} \geq \frac{b_L^2 2^{l}}{m n^2 \varepsilon^2}$. Combining the above bounds, and plugging in the value of $L$ and $b_L \asymp \log N$, we obtain that
\begin{equation*}
    \E_f \| \hat{f}^{\textsc{BTPW}} - f \|_2^2 \lesssim N^{-\frac{2\alpha}{2\alpha + 1}} + \left( \frac{m n^2 \varepsilon^2}{\log N} \right)^{-\frac{2\alpha}{2\alpha + 2}},
\end{equation*}
which establishes the statement of Theorem \ref{thm:upper-bound-global-risk}.

\end{proof}

\begin{proof}[Proof of Theorem \ref{thm:upper-bound-pointwise}]
The proof is in spirit similar as that of Theorem \ref{thm:upper-bound-global-risk}. Fix $t_0 \in (0,1)$, $\alpha$, $p$ such that $\nu = \alpha - 1/p > 1/2$. Let $\hat{T} = \hat{f}^{KWT}(t_0)$. The pointwise error $\E_f(\hat{T} - f(t_0))^2$ is bounded above by
\begin{align}\label{eq:pointwise-risk-decomposition}
    4\E_f\left[ \sum_{k \in S_0(t_0)} \left(  \hat{f}_{0k}^{\textsc{PW}} - f_{0k} \right) \phi_{k}(t_0)  \right]^2 &+ 4\E_f\left[\sum_{l=l_0}^{L^*} \sum_{k \in S_l(t_0)} \left(  \eta_{\tau_l}(\hat{f}_{lk}^{\textsc{PW}}) - f_{lk} \right) \psi_{lk}(t_0)  \right]^2  \\ &+ 2 \left(\sum_{ l > L^* } \sum_{k \in S_l(t_0)}f_{lk} \psi_{lk}(t_0)\right)^2, \nonumber
\end{align}
where $S_l(t_0)$ denotes the $O(1)$ set of wavelet coefficients at resolution level $l$ that are non-zero at the point $t_0$. Since $|\phi_k| \leq c_\phi$ and $|S_0(t_0)| = O(1)$, first term is of the order 
\begin{equation*}
    \text{Var}_f(\hat{f}_{0k}^{\textsc{PW}}) = \frac{c_\phi^2}{m n} + \frac{c_\phi^2 L^*}{m n^2 \varepsilon^2},
\end{equation*}
where the last equality follows from Lemma \ref{lem:osc-exponential-mechanism-tails}. By Lemma \ref{lemma:coefficient_decay}, $|f_{lk}| \leq R 2^{-l(\nu + \frac{1}{2})}$, which combined with the bound $|\psi_{lk}(t_0)| \leq 2^{l/2}c_\psi$ gives
\begin{equation*}
    \left|\sum_{l >L^*} \sum_{k \in S_l(t_0)} f_{lk} \psi_{lk}(t_0) \right| \leq R c_\psi \sum_{l>L^*} 2^{-l\nu} \lesssim 2^{-L^*\nu}.
\end{equation*}
The second term in \eqref{eq:pointwise-risk-decomposition} is bounded by 
\begin{equation}\label{eq:second-term-pointwise-ub-refbacklater}
    2 \left[\sum_{l=l_0}^{L^*} \sum_{k \in S_l(t_0)}\left( \E_f \left( \eta_{\tau_l}(\hat{f}_{lk}^{\textsc{PW}}) - f_{lk} \right)^2  \right)^{1/2} |\psi_{lk}(t_0)|\right]^2.
\end{equation}
Following the same reasoning as in the proof of Theorem \ref{thm:upper-bound-global-risk}, we can bound the expectation $\E_f \left( \eta_{\tau_l}(\hat{f}_{lk}^{\textsc{PW}}) - f_{lk} \right)^2$ by 
\begin{equation}\label{eq:oracle-block-thresholding-applied}
    \E_f | \eta_{\tau_l} (\hat{f}_{lk}^{\textsc{PW}}) - f_{lk} |^2 \leq  \min \{ f_{lk}^2, 4\tau_l^2 \} + 4 \E_f Z_{lk}^2 \mathbbm{1}\{ Z_{lk}^2 > \tau_l^2 \},
\end{equation}
where $Z_{lk} = \hat{f}_{lk} - f_{lk} + \bar{V}_{lk}$. The tail term $4\E_f Z_{lk}^2 \mathbbm{1}\{|Z_{lk}| > \tau_l\}$ is negligible by the same arguments as in the proof for the global risk (Theorem \ref{thm:upper-bound-global-risk}): first bounding it from above by
\begin{equation*}
    \E_f |(\hat{\psi} - f_{\cdot \cdot})_{lk}|^2 \mathbbm{1}\{ |(\hat{\psi} - f_{\cdot \cdot})_{lk}| > \sqrt{2 \log n} \} + \E_f |\bar{V}_{lk}|^2 \mathbbm{1}\{ |\bar{V}_{lk}| > \tau_l' \},
\end{equation*}
we find that the first term is $O(N^{-1})$ by Lemma \ref{lem:non-private-block-thresholding-tailbound}. For the second term, we apply Lemma \ref{lem:osc-exponential-mechanism-tails}, where we note the difference in terms of the threshold $\tau_l'$ compared to Theorem \ref{thm:upper-bound-global-risk}, which essentially recognizes the following two cases:
\begin{enumerate}[label=(\roman*)]
    \item If $\log(N) / m \geq 1$, $\tau_l' \asymp  \sqrt{\frac{\log(N)}{m n^2 \varepsilon^2}}$,
    \item  If $\log(N) / m < 1$, $\tau_l' \asymp  \sqrt{\frac{\sqrt{\log(N)}}{m n^2 \varepsilon^2}}$.
\end{enumerate}
Each of these two cases yield the bound 
\begin{equation*}
    \E_f |\bar{V}_{lk}|^2 \mathbbm{1}\{ |\bar{V}_{lk}| > \tau_l' \} \lesssim e^{- \log (N)} + \frac{2^{4L^*}}{(m n^2 \varepsilon^2)^2} e^{-c 2^{L^*}} = O\left(\frac{1}{N}\right).
\end{equation*}
We obtain that
\begin{equation*}
    \E_f \left( \hat{f}_{lk} - f_{lk} \right)^2 \lesssim \min \left\{ \frac{\log N}{N} + \frac{(\log N)^{L_{m,N}/\log N} 2^{l}}{m n^2 \varepsilon^2}, \; \; 2^{-2 l(\nu + 1/2)} \right\},
\end{equation*}
where $L_{m,N}$ is as defined in the theorem statement. Applying this bound in \eqref{eq:second-term-pointwise-ub-refbacklater}, the sum over $l$ involves terms of the form $2^{l/2} \sqrt{\min(v_l, b_l)}$, which behave geometrically: increasing for small $l$ (dominated by variance $v_l$) and decreasing for large $l$ (dominated by bias $b_l$). Thus, the sum is bounded by a constant times its maximum, which is achieved at the balancing level $l$ where the variance term balances the bias term. Choosing 
\[ L = \left\lceil \frac{1}{2\nu + 1} \log_2\left(\frac{N}{\log N}\right) \right\rceil \wedge \left\lceil \frac{1}{2\nu + 2} \log_2\left(\frac{m n^2 \varepsilon^2}{L_{m,N}}\right) \right\rceil \]
and splitting the sum over $l \le L$ and $l > L$ yields the desired bound, in the same way as in the proof of Theorem \ref{thm:upper-bound-global-risk}.
\end{proof}

\section{Proofs relating to the lower bounds: Theorems \ref{thm:adaptation-lower-bound-pointwise} and \ref{thm:global-lower-bound}}\label{sec:lowerbounds:appendix}

\subsection{Global risk lower bound proofs (Theorem \ref{thm:global-lower-bound})}\label{sec:global-risk-lower-bound-proofs}

We provide proofs of the intermediate results required for the proof of Theorem \ref{thm:global-lower-bound}. Section \ref{sec:properties-of-induced-fisher-information-matrix} establishes properties of the induced Fisher information matrix, and Section \ref{sec:proofs-of-lemmas-global-lower-bound-key-lemma-and-data-processing} contains the proofs of Lemmas \ref{lem:global-lower-bound-key-lemma} and \ref{lem:data-processing}.

\subsubsection{Some properties of the transcript-induced Fisher information matrix}\label{sec:properties-of-induced-fisher-information-matrix}

The following lemma establishes some properties of the Fisher information within the sequential federated framework described in Section \ref{sec:lower-bound}. Results of this flavor can be found in earlier literature (e.g. \cite{zamir1998proof}), however, they typically require transcripts to have dominated conditional distributions, which are generally not available in the $(\varepsilon,\delta)$-FDP setting. 

Consider a $\mu$-dominated model $\{ P_u : u \in \cU \}$ on a measurable space $(\cX, \mathscr{X})$ for some open set $\cU \subset \R^d$. Let $X=(X_1,\dots,X_n) \overset{\text{i.i.d.}}{\sim} P_u$ be random variables. Assume that the score function $\nabla_u \log p_u(X)$ is defined $\mu$-a.s. and is square integrable in expected Euclidian norm for all $u \in \cU$. Let $T$ be a random variable generated by the Markov kernel $K$ between $(\cX^n, \mathscr{X}^n)$ and $(\cT, \mathscr{T})$. Let $\E_u$ denote expectation with respect to the joint distribution of $(X_i)_{i=1}^n$ and $T$. 

If the transcript $T$ has a dominated conditional distribution, meaning that the collection $\{ A \mapsto K(A|x) : x \in \cX^n \}$ is dominated by a measure $\nu$ on $(\cT, \mathscr{T})$, then the $T$-induced score function $t \mapsto s_n(t|u)$ can be defined in the `strong sense': $\nabla_u \log q_u(t)$, where $q_u$ is the Radon-Nikodym derivative of the marginal distribution of $T$ is given by
\begin{equation*}
dQ_u(t) := \int \left( \prod_{i=1}^n p(x_i|u) \right) dK(t | x) d\mu^n(x).
\end{equation*}

However, it is a rather restrictive assumption under $(\varepsilon,\delta)$-DP to assume that such a $\nu$ exists. Whilst typical mechanisms using additive noise (e.g. the Gaussian mechanism) satisfy this assumption, the $\delta > 0$ in the definition of $(\varepsilon,\delta)$-DP in principle does not rule out the possibility of non-dominated conditional distributions, and for many mechanisms, it is not easy to verify that such a $\nu$ exists.


Below, we show that in a `weak sense', the $T$-induced score function can still be well-defined, and that this is sufficient for the purposes of deriving the results in e.g. \cite{zamir1998proof} and \cite{gill1995applications}, which are central to our proof of Theorem \ref{thm:global-lower-bound}.

The following lemma shows that the \emph{$T$-induced score function}, the function $t \mapsto s_n(t|u)$ which satisfies
\begin{equation*}
\E_u H(T) s_n(T|u) = \nabla_u \E_u H(T) \quad \forall H \in L_\infty(Q_u),
\end{equation*}
is well-defined as an element of $L_2(Q_u)$ and satisfies the identity $s_n(T|u) = \E_u \left[ s_n(X|u) \mid T \right]$.

\begin{lemma}\label{lem:score-well-defined}
Assume that for all $x \in \mathcal{X}^n$, the map $u \mapsto p(x|u)$ is differentiable with score function $s_n(x|u) = \sum_{i=1}^n \nabla_u \log p(x_i|u)$ in the sense that 
\begin{equation*}
\nabla_u \E_u H(X) = \sum_{i=1}^n \E_u \left[ H(X) \nabla_u \log p(X_i|u) \right] \quad \forall H \in L_\infty(P_u^n).
\end{equation*}
If $\E_u [\|s_n(X|u)\|^2] < \infty$, then the $T$-induced score function $t \mapsto s_n(t|u)$ is well-defined as an element of $L_2(Q_u)$. Furthermore,
\begin{equation*}
s_n(T|u) = \E_u \left[ s_n(X|u) \mid T \right].
\end{equation*}
\end{lemma}
\begin{proof}
Since $s_n(X|u)$ is square integrable, the random variable $\E_u[s_n(X|u)|T]$ exists as an element of $L_2(Q_u)$. Write $g(T) = \E_u[s_n(X|u)|T]$. For any bounded measurable function $h(T)$, using the interchange of differentiation and integration assumed in the lemma, we have
\begin{align*}
\int h(t) g(t) dQ_u(t) &= \E_u [h(T) \E_u[s_n(X|u)|T]] = \E_u [h(T) s_n(X|u)] \\
&= \iint h(t) s_n(x|u) dK(t|x) p_u^n(x) d\mu^n(x) \\
&= \iint h(t) dK(t|x) \nabla_u p_u^n(x) d\mu^n(x) \\
&= \nabla_u \iint h(t) dK(t|x) p_u^n(x) d\mu^n(x) \\
&= \nabla_u \E_u [h(T)] = \nabla_u \int h(t) dQ_u(t),
\end{align*}
where we wrote $p_u^n(x) = \prod_{i=1}^n p(x_i|u)$. This identifies $g$ as the unique score function $t\mapsto s_n(t|u)$ in the $L_2$-sense.
\end{proof}

Define the \emph{$T$-induced Fisher information} as
\begin{equation*}
\cI(u) = \E_u \left[ s_n(T|u) s_n(T|u)^\top \right].
\end{equation*}

The $L_2$-characterization of the induced Fisher information matrix is sufficient for obtaining a Van Trees inequality for transcript-induced information, which is the content of the following lemma.

\begin{lemma}\label{lem:van-trees}
Let $\pi$ be a prior on $\cU \subset \R^d$ with density $\pi$ (with slight abuse of notation) with respect to Lebesgue measure, and assume that $\pi$ is absolutely continuous. Define the prior Fisher information as
\begin{equation*}
\cI(\pi) = \int_{\cU} \frac{\| \nabla \pi(u) \|^2}{\pi(u)} du,
\end{equation*}
assuming the quantity is finite.
Assume that $\pi(u) \to 0$ as $u \to \partial \cU$. Then for any estimator $\hat{u}(T)$ of $u$,
\begin{equation*}
\E_\pi \left[ \| \hat{u}(T) - u \|^2 \right] \geq \frac{d^2}{\E_\pi \text{Tr}(\cI(u)) + \cI(\pi)},
\end{equation*}
where $\cI(u)$ is the $T$-induced Fisher information matrix at $u$.
\end{lemma}

\begin{proof}
See e.g. \cite{gill1995applications}, Theorem 1. The result is a direct application of the multivariate van Trees inequality, noting that the score function $s_n(T|u)$ satisfies the moment condition $\E_u [s_n(T|u)] = 0$ and the `integration by parts' property derived in the proof of Lemma \ref{lem:score-well-defined}. The quantity $\E_\pi \text{Tr}(\cI(u))$ corresponds to the integrated Fisher information of the model.
\end{proof}

The next lemma decomposes the Fisher information for sequential observations, showing that the total information is the sum of the conditional information from each step. 

\begin{lemma}\label{lem:fisher-conditional-independence}
Let $T = (T_1, \dots, T_m)$ be a sequence of observations generated via a sequential mechanism, where for each $j=1,\dots,m$, $T_j$ is drawn from a distribution that depends only on the latent data $X_{}^{(j)}$ and the previous observations $T_{1}, \dots, T_{j-1}$, such that conditionally on $X$ and $T_{1}, \dots, T_{j-1}$, $T_j$ is independent of $u$, given by Markov kernel $K_j(\cdot | \cdot, \cdot)$. 

Then, $\E s_n(T|u) s_n(T|u)^\top$ equals
\begin{equation*}
  \sum_{j=1}^m \E_u \E_u \left[ s_n(X^{(j)}|u) \mid T_{1}, \dots, T_{j} \right] \E_u \left[ s_n(X^{(j)}|u) \mid T_{1}, \dots, T_{j} \right]^\top.
\end{equation*}
\end{lemma} 

\begin{proof}
In light of the conditional independence structure, the joint distribution of $T = (T_1, \dots, T_m)$ factorizes as
\begin{equation*}
dQ_u(t_1, \dots, t_m) = \bigotimes_{j=1}^m dQ_u(t_j | t_1, \dots, t_{j-1}).
\end{equation*}
We obtain that
\begin{equation*}
s_n(T|u) = \sum_{j=1}^m s_n(T_j|T_{1}, \dots, T_{j-1},u).
\end{equation*}
Since $T_j$ depends on $u$ only through $X^{(j)}$, and $X^{(j)}$ is independent of $X^{(-j)}$ given $u$, we find that $s_n(T_j|T_{1}, \dots, T_{j-1},u)$ equals
\begin{equation*}
 \E_u \left[ s_n(X^{(j)}|u) \mid T_{1}, \dots, T_{j} \right] - \E_u \left[ s_n(X^{(j)}|u) \mid T_{1}, \dots, T_{j-1} \right] =: \Delta_j.
\end{equation*}
The terms $\Delta_j$ form a martingale difference sequence with respect to the filtration $\cF_j = \sigma(T_{1}, \dots, T_{j})$. Indeed, $\E_u [\Delta_j | \cF_{j-1}] = 0$. Consequently, the cross terms in the expectation of the outer product vanish:
\begin{equation*}
\cI(u) = \E_u \left[ \left(\sum_{j=1}^m \Delta_j\right) \left(\sum_{l=1}^m \Delta_l\right)^\top \right] = \sum_{j=1}^m \E_u [\Delta_j \Delta_j^\top].
\end{equation*}
Crucially, since $X^{(j)}$ is independent of $T_{1}, \dots, T_{j-1}$, we have that
\begin{equation*}
\E_u [ s_n(X^{(j)}|u) | T_{1}, \dots, T_{j-1} ] = \E_u [ s_n(X^{(j)}|u) ] = 0,
\end{equation*}
which implies that $\Delta_j = \E_u [ s_n(X^{(j)}|u) \mid T_{1}, \dots, T_{j} ]$. The result follows.
\end{proof}

\subsubsection{Proofs of Lemmas \ref{lem:global-lower-bound-key-lemma} and \ref{lem:data-processing}}\label{sec:proofs-of-lemmas-global-lower-bound-key-lemma-and-data-processing}

\begin{proof}[Proof of Lemma \ref{lem:global-lower-bound-key-lemma}]

    Fix $\hat{f} \in \cF(\varepsilon,\delta)$, and let $T = (T^{(j)})_{j=1,\dots,m}$ denote the corresponding FDP transcripts. Let $\alpha_L$ is the value in $\tilde{\cA}$ closest to the solution of $2^{-L(\alpha+1)} = \tilde{\rho}_\alpha^{-1}$. Consider 
    \begin{equation*}
        F_{L}^U = 1 + \sum_{k=1}^{2^L} U_{Lk} \psi_{Lk},
    \end{equation*}
    where the vector $U_L = (U_{Lk})_{k=1}^{2^L}$ is supported on the hypercube 
    \begin{equation*}
        \cU_L = [-C_R 2^{-L(\alpha_L+1/2)}, C_R 2^{-L(\alpha_L+1/2)}]^{2^L}
    \end{equation*}
    under some distribution $\P^{U_L}$, with expectation operator denoted by$\E^{U_L}$.
    For each $L \in \cL$, $F_L^U$ is a probability density in $\cB_{pq}^{\alpha_L}(R)$. Consequently,
    \begin{align}\label{eq:global-lower-bound-key-lemma-1}
        \sup_{\alpha \in \tilde{\cA}}\sup_{f \in B_{pq}^\alpha(R)} \, \E_f  \, \tilde{\rho}_\alpha^{-2} \, \| \hat{f} - f \|_2^2  \geq \sup_{L \in \cL} \, \E^{U_L} \, \E_{F_L^U}  \, \tilde{\rho}_{\alpha_L}^{-2} \, \| \hat{f} - F_L^U \|_2^2.
       \end{align}
       By Plancharel's theorem, we have that
       \begin{equation*}
        \| \hat{f} - F_L^U \|_2^2 \geq \sum_{k=1}^{2^L} ( \hat{f}_{Lk} - U_{Lk} )^2,
       \end{equation*}
       where $\hat{f}_{Lk}$ denotes the wavelet coefficient of $\hat{f}$ at level $L$ and index $k$. Using the multivariate van Trees inequality (Lemma \ref{lem:van-trees} in Section \ref{sec:auxiliary_lemmas}), we obtain that  
       \begin{equation*}
           \tilde{\rho}_{\alpha_L}^{-2} \, \E^{U_L} \, \| (\hat{f}_{Lk})_{k=1}^{2^L} - (F^U_{Lk})_{k=1}^{2^L}\|_2^2 \geq \frac{2^{2L}}{ \tilde{\rho}^{2}_{\alpha_L} \, (\E^{U_L} [\text{Tr}(\cI_{L}(U_L))] + \cI(\P^{U_L}))},
       \end{equation*}
       where $\cI(\P^{U_L})$ is the Fisher information of the `prior' $\P^{U_L}$ (i.e. the marginal distribution of $U_L$), which is defined as 
       \begin{equation*}
           \cI(\P^{U_L}) = \int \frac{\left\| \nabla \pi_L(u) \right\|^2}{\pi_L(u)} \, du,
       \end{equation*}
       with $\pi_L(u)$ denoting the probability density of $U_L$, and $\cI_{L}(U_L)$ denotes the \emph{transcript induced Fisher information matrix} of the submodel obtained by considering $U_L \in \cU_L$. 

    The score of this submodel is given by
       \begin{equation*}
       S_L \equiv S_L(X_1,\dots,X_N) := \nabla_{(u)_L}  \sum_{j=1}^m \sum_{i=1}^n \log F_L^u(X_i^{(j)}) \mid_{u_L=U_L} = \sum_{j=1}^m S_L^{(j)},
       \end{equation*}
       where $S_L^{(j)} = \sum_{i=1}^n  \left(\frac{\psi_{Lk}(X_i^{(j)})}{\bar{F}_L^U(X_i^{(j)})} \right)_{k = 1}^{2^L}$. That is, by Lemma \ref{lem:score-well-defined}, we have that
       \begin{equation*}
           \cI_{L}(U_L) = \E_{{F}_L^U} \E_{{F}_L^U} \left[ S_L | T \right] \E_{{F}_L^U} \left[ S_L | T \right]^\top,
       \end{equation*}
       where we note that in the notation here ${F}_L^U$ is a product density over the $N$ data points $X_i^{(j)}$ for $j=1,\dots,m$ and $i=1,\dots,n$, and $\cI_{L}(U_L)$ coincides with $\cI_{L}(U)$ as defined in the lemma statement under the slight abuse of notation that $U_{(-L)} = 0$ is left out of the function notation.
       
       Noting that the resulting random variables $2^{L(\alpha_L+1/2)} U_{Lk}$ are nearly uniformly distributed on $[-C_R,C_R]$, a straightforward calculation finds that the induced prior Fisher information is of the order $2^{(2\alpha + 2) L}$, with constant that depends on $\eta$. Using $ 2^{\alpha L_\alpha } \asymp  \tilde{\rho}^{-1}_\alpha $ and $ 2^{- 2 L_\alpha } \tilde{\rho}^{2}_\alpha = \frac{\log (N)}{m n^2 \varepsilon^2}$, find that \eqref{eq:global-lower-bound-key-lemma-1} is lower bounded by a constant multiple of
       \begin{equation}\label{eq:global-lower-bound-key-lemma-1-bound-post-vantrees}
             \left({  \, \frac{\log (N)}{m n^2\varepsilon^2} \, \min_{L \in \cL} \, \E^{U_L} \text{Tr}(\cI_{L}(U_{L})) + 1}\right)^{-1}.
       \end{equation}
Clearly, $\E^{U_L} \text{Tr}(\cI_{L}(U_{L}))$ is bounded $\sup_{\P^U} \E^{U} \text{Tr}(\cI_{L}(U))$, where the supremum is over all distributions on $\cU := \cup_{L \in \cL} \cU_L$, with
\begin{equation*}
    \cI_{L}(U) = \E_{{F}_\cL^U} \E_{{F}_\cL^U} \left[ S_L | T \right] \E_{{F}_\cL^U} \left[ S_L | T \right]^\top
\end{equation*}
as defined in the lemma statement. Let $q_\cL$ be an element of the probability simplex over $\cL$, and note that the map $q \mapsto \sum_{L \in \cL} q_L \E^U \text{Tr}(\cI_{L}(U))$ is linear. Similarly, the map $\P^U \mapsto \E^U \text{Tr}(\cI_{L}(U))$ is linear. Due to the boundededness of the wavelets (and hence the score), $\text{Tr}(\cI_{L}(U))$ is bounded, making both of the aforementioned linear maps continuous.

The set of probability measures on $\cU$ is convex and weak-*-compact, and so is the probability simplex over the finite set $\cL$, so the corresponding min-max problem has a saddle point: there exists a probability measure $\P^U$ such that \eqref{eq:global-lower-bound-key-lemma-1-bound-post-vantrees} is bounded by 
\begin{equation*}
    \left({  \, \frac{\log (N)}{m n^2\varepsilon^2} \, \min_{L \in \cL} \, \E^{U} \text{Tr}(\cI_{L}(U)) + 1}\right)^{-1}.
\end{equation*}
Using the correspondence between $L \in \cL$ and $\alpha \in \tilde{\cA}$, the statement of the lemma follows.
    \end{proof}

            \begin{proof}[Proof of Lemma \ref{lem:data-processing}]
                Recall that, in the FDP setup of Definition \ref{def:federated_differential_privacy} and \ref{def:federated_differential_privacy_chained},
                \begin{equation*}
                    T^{(j)}| T^{(j-1)} \overset{d}{=} T^{(j)}| T^{(j-1)},\dots,T^{(1)}.
                \end{equation*}
                Due to the conditional independence of the transcripts, $\text{Tr}(\cI_{\cL}) = \sum_{j=1}^m \E[\text{Tr}(C_{T^{(j)}})]$, where $C_{T^{(j)}} = \E[S_{\cL}^{(j)} | T^{(j)},T^{(j-1)}] \E[S_{\cL}^{(j)} | T^{(j)},T^{(j-1)}]^T$ and $S_{\cL}^{(j)}$ is the stacked score for server $j$.
                
                For each server $j$, define 
                \begin{equation*}
                    G_{j,i} = \langle \E [ \bar{S}_{\cL}^{(j)} | T^{(j)}, T^{(j-1)}], S_{\cL}^{(j,i)}(X_i^{(j)}) \rangle, 
                \end{equation*}
                where $\bar{S}_{\cL}^{(j)} = \sum_{i=1}^n S_{\cL}^{(j,i)}(X_i^{(j)})$ is the score for server $j$, and $S_{\cL}^{(j,i)}(X_i^{(j)})$ is the single-observation score for the $i$-th sample on server $j$. Let $\breve{G}_{j,i} =\langle \bar{S}_{\cL}^{(j)} , S_{\cL}^{(j,i)}(\breve{X}_i^{(j)}) \rangle$ where $\breve{X}_i^{(j)}$ denotes an independent copy of $X_i^{(j)}$. 
                $|\psi_{lk}(x)| \leq 2^{l/2} \|\psi\|_\infty$. For 
                \begin{equation*}
                    u \in [-2^{-L(\alpha+1/2)} C_R, 2^{-L(\alpha+1/2)} C_R],
                \end{equation*}
                we have that 
                \begin{equation*}
                    \left| \frac{\psi_{lk}(x)}{F_{\cL}^u(x)} \right| \leq \frac{2^{l/2} \|\psi\|_\infty}{1 - 2^{- \alpha_{\min}/2 } \|\psi\|_\infty K R} \leq C_{\psi} 2^{l/2},
                \end{equation*}
                where $K$ in the second inequality depends on the support of the chosen wavelets.
        
                Consequently, we have that 
                \begin{equation}\label{eq:Gi-bounded}
                    |G_{j,i}| \leq C_{\psi}^2 n 2^{l} \quad \text{ and } \quad |\breve{G}_{j,i}| \leq C_{\psi}^2 n 2^{l}.
                \end{equation}
                Also,
                \begin{equation*}
                    \E_{F^U_{\cL}} \frac{\psi_{lk}(X_i^{(j)})}{F_{\cL}^U(X_i^{(j)})} = \int_0^1 \psi_{lk}(x)dx = 0.
                \end{equation*}
                Following Lemma 5.3 in \cite{Cai2024FL-NP-Regression}, we have that for any $M>0$,
                \begin{align*}
                    \text{Tr}(C_{T^{(j)}}) \leq & C n \varepsilon \sqrt{\text{Tr}(C_{T^{(j)}})} \sqrt{\lambda_{\max} (\E[S_{\cL}^{(j,1)} (S_{\cL}^{(j,1)})^T])} \\  &+ 2 M \delta + \int_M^\infty \P \left( (G_{j,i})^+ \geq t \right) dt + \int_M^\infty \P \left( (\breve{G}_{j,i})^- \geq t \right)dt,
                \end{align*}
                where $S_{\cL}^{(j,1)}$ is the single-observation score, $C > 0$ is universal, $(G_{j,i})^+ = \max(G_{j,i}, 0)$, and $(G_{j,i})^- = -\min(G_{j,i}, 0)$.
        
                Taking $M = C_{\psi}^2 n 2^{\max_{L \in \cL} L} \lesssim n N$ (since $\max L = O(\log N)$), we find that the latter two integrals are zero by \eqref{eq:Gi-bounded}. Under the assumption $\delta  \ll n^{} \varepsilon^2 / N $, the $M \delta$-term is $o(n^2 \varepsilon^2)$. Solving the quadratic inequality yields $\text{Tr}(C_{T^{(j)}}) \leq C n^2 \varepsilon^2$, where we use that $\lambda_{\max} (\E[S_{\cL}^{(j,1)} (S_{\cL}^{(j,1)})^T]) \lesssim 1$ due to orthogonality of the wavelet basis. Summing over $j=1,\dots,m$ gives the result.
            \end{proof}

\subsection{Pointwise risk lower bound proofs (Theorem \ref{thm:adaptation-lower-bound-pointwise})}\label{sec:pointwise-risk-lower-bound-proofs} 

We start with the proof of Theorem \ref{thm:adaptation-lower-bound-pointwise}, followed by the proof of Theorem \ref{thm:pointwise_cost_of_adaptation}. Auxiliary results are deferred to the sub-section at the end of this section.

\subsubsection{Proof of Theorem \ref{thm:adaptation-lower-bound-pointwise}}

Set
\begin{equation*}
B_N = \begin{cases}
\displaystyle\breve c\,\frac{\log A_N}{\sqrt{m}} & \text{if } m \le \log A_N, \\[6pt]
\sqrt{\breve c\,{\log A_N}} & \text{if } m > \log A_N,
\end{cases}
\quad\text{and}\quad
\Delta_N = \left(\frac{N}{\log A_N}\right)^{-\frac{\nu}{2\nu+1}}
\vee \left(\frac{m n^2\varepsilon^2}{B_N^2}\right)^{-\frac{\nu}{2\nu+2}},
\end{equation*}
with $\breve c>0$ chosen sufficiently small (specified below). Let $h$ be a compactly supported smooth bump function with $h(0)>0$, $\int h=0$, and $\|h\|_2>0$, and define
\begin{equation*}
g(x)= f_0(x) + a\,h\!\left(b\,(t_0-x)\right),
\qquad a:=c_1\Delta_N,\quad b:=c_2\,\Delta_N^{-1/\nu},
\end{equation*}
for fixed $c_1,c_2>0$ small so that $g\in\cB^\alpha_{p,q}(R)$; see Lemma~1 of \cite{cai2003pointwiseBesov}. Then
\begin{equation}\label{eq:Delta-and-TV}
\Delta := |g(t_0)-f_0(t_0)| \asymp \Delta_N,
\qquad
\|g-f_0\|_1 \asymp \frac{a}{b}.
\end{equation}
Since $f_0,g$ are densities, $p_{f_0,g}:=\|P_{f_0}-P_g\|_{\mathrm{TV}}=\tfrac12\|g-f_0\|_1 \asymp a/(2b)$.

Depending on which term in the definition of $\Delta_N$ is larger, we distinguish two regimes; one where the non-private rate dominates, and one where the private rate dominates.

\smallskip\noindent\textbf{Private regime:}
Assume the maximum in $\Delta_N$ is attained by $\Delta_N = (m n^2\varepsilon^2/B_N^2)^{-\nu/(2\nu+2)}$. With the above choice,
\begin{equation}\label{eq:TV-scaling-private}
\frac{a}{b}\asymp\Delta_N^{1+1/\nu}
=\left(\frac{B_N^2}{m n^2\varepsilon^2}\right)^{1/2}
\quad\Longrightarrow\quad
p_{f_0,g}\asymp\frac{{B_N}}{\sqrt{m}\, n\varepsilon}.
\end{equation}
Apply the private constrained-risk lemma (Lemma \ref{lem:private-constrained-risk-delta}) with the semi-metric
$\mathrm{d}(h_1,h_2)=|h_1(t_0)-h_2(t_0)|$, $f=f_0$, $g$ as above, and $\Delta$ from \eqref{eq:Delta-and-TV}. 
Let
\begin{equation*}
\gamma^2 := \frac{\E_{f_0}\!\left(\hat T - f_0(t_0)\right)^2}{\Delta^2}
\;\lesssim\; \frac{\log^{O(1)} A_N}{A_N}
\qquad\text{by }\eqref{eq:thm:adaptation-lower-bound-pointwise-assumption-clean}
\text{ and }\Delta^2\asymp \Delta_N^2.
\end{equation*}
Lemma \ref{lem:private-constrained-risk-delta} yields
\begin{align}
\E_g\!\left|\hat T - g(t_0)\right|^2
&\ge \frac{\Delta^2}{4}\Bigl( 1 - \sqrt{5}\,e^{m (\bar{\varepsilon}^2/2 \wedge \bar{\varepsilon}) + \log\gamma} - 4 m \bar{\delta} \Bigr),
\label{eq:key-bracket}
\end{align}
where $\bar{\varepsilon} = 6 \varepsilon n p_{f_0,g}$ and $\bar{\delta} = 4 e^{\bar{\varepsilon}} n \delta p_{f_0,g}$.

From \eqref{eq:TV-scaling-private}, we have
\begin{equation*}
\bar{\varepsilon} = 6\varepsilon n \, p_{f_0,g} \asymp \frac{B_N}{\sqrt{m}}.
\end{equation*}
Hence
\begin{equation*}
\frac{\bar{\varepsilon}^2}{2} \asymp \frac{B_N^2}{m}
\qquad\text{and}\qquad
m\,\bar{\varepsilon} \asymp B_N\sqrt{m}.
\end{equation*}
The quantity $\bar{\varepsilon}^2/2 \wedge \bar{\varepsilon}$ equals $\bar{\varepsilon}^2/2$ when $\bar{\varepsilon} \le 2$ (i.e., $B_N \lesssim \sqrt{m}$) and equals $\bar{\varepsilon}$ when $\bar{\varepsilon} > 2$ (i.e., $B_N \gtrsim \sqrt{m}$).

\smallskip
\textbf{Case 1: $m \le \log A_N$ and $B_N = \breve c\,\frac{\log A_N}{\sqrt{m}}$.}

\smallskip\noindent
Then
\begin{equation*}
\bar{\varepsilon} \asymp \breve c\,\frac{\log A_N}{m}.
\end{equation*}
Since $m \le \log A_N$, we have $\bar{\varepsilon} \gtrsim \breve c$. For $\breve c$ chosen large enough that $\bar{\varepsilon} \ge 2$, the minimum satisfies $\bar{\varepsilon}^2/2 \wedge \bar{\varepsilon} = \bar{\varepsilon}$, so
\begin{equation*}
m\,\bar{\varepsilon} \asymp \breve c\,\log A_N.
\end{equation*}
For the $\bar{\delta}$ term: $\bar{\delta} \asymp e^{\bar{\varepsilon}} n\delta \cdot \frac{B_N}{\sqrt{m}\,n\varepsilon} = e^{\bar{\varepsilon}} \delta \cdot \frac{B_N}{\sqrt{m}\,\varepsilon}$. Since $\bar{\varepsilon} \asymp \breve c\,\frac{\log A_N}{m} \le \breve c$ (as $m \ge 1$), we have $e^{\bar{\varepsilon}} = O(1)$. Thus
\begin{equation*}
m\,\bar{\delta} \asymp \frac{m\,\delta\, B_N}{\sqrt{m}\,\varepsilon} 
= \frac{\sqrt{m}\,\delta\,\breve c\,\log A_N}{\sqrt{m}\,\varepsilon}
= \frac{\breve c\,\delta\,\log A_N}{\varepsilon}
\;\ll\; \frac{1}{A_N}
\;\le\; \gamma^2,
\end{equation*}
by condition \eqref{eq:delta-condition}.

For the exponential term: since $\gamma^2 \lesssim \frac{\log^{O(1)} A_N}{A_N}$, we have
\begin{equation*}
\log\gamma \;\le\; -\tfrac{1}{2}\log A_N + O(\log\log A_N).
\end{equation*}
Therefore
\begin{equation*}
m\,\bar{\varepsilon} + \log\gamma 
\;\le\; \breve c'\log A_N - \tfrac{1}{2}\log A_N + O(\log\log A_N)
= -(\tfrac{1}{2} - \breve c')\log A_N + O(\log\log A_N).
\end{equation*}
Choosing $\breve c$ (and hence $\breve c'$) small enough that $\breve c' < 1/4$, we obtain
\begin{equation*}
e^{m\,\bar{\varepsilon} + \log\gamma} \;\lesssim\; A_N^{-c} \;\ll\; 1
\end{equation*}
for some constant $c > 0$. Thus the bracket in \eqref{eq:key-bracket} is $1 - o(1)$, and
\begin{equation*}
\E_g\!\left|\hat T - g(t_0)\right|^2 
\;\gtrsim\; \Delta_N^2 
\asymp \left(\frac{m n^2\varepsilon^2}{B_N^2}\right)^{-\frac{2\nu}{2\nu+2}}
= \left(\frac{m^2 n^2\varepsilon^2}{\log^2 A_N}\right)^{-\frac{2\nu}{2\nu+2}}.
\end{equation*}
Since $L_{m,N} = \frac{\log^2 A_N}{m}$ when $m \le \log A_N$, this matches the claimed rate.

\smallskip
\textbf{Case 2: $m > \log A_N$ and $B_N = \sqrt{\breve c\,\log A_N}$.}

\smallskip\noindent
Then
\begin{equation*}
\bar{\varepsilon} \asymp \sqrt{\frac{\breve c\,\log A_N}{m}}.
\end{equation*}
Since $m > \log A_N$, we have $\bar{\varepsilon} \lesssim \sqrt{\breve c}$. For $\breve c$ small enough that $\bar{\varepsilon} \le 2$, the minimum satisfies $\bar{\varepsilon}^2/2 \wedge \bar{\varepsilon} = \bar{\varepsilon}^2/2$, so
\begin{equation*}
m \cdot \frac{\bar{\varepsilon}^2}{2} \asymp m \cdot \frac{\breve c\,\log A_N}{m} = \breve c\,\log A_N.
\end{equation*}
For the $\bar{\delta}$ term: since $\bar{\varepsilon} = O(1)$, we have $e^{\bar{\varepsilon}} = O(1)$, and
\begin{equation*}
m\,\bar{\delta} \asymp \frac{m\,\delta\,B_N}{\sqrt{m}\,\varepsilon}
= \frac{\sqrt{m}\,\delta\,\sqrt{\breve c\,\log A_N}}{\varepsilon}
\;\ll\; \frac{\sqrt{m\log A_N}}{A_N}
\;\le\; \gamma^2,
\end{equation*}
by condition \eqref{eq:delta-condition} and the fact that $\gamma^2 \gtrsim A_N^{-1}\log^{O(1)} A_N$.

For the exponential term:
\begin{equation*}
m \cdot \frac{\bar{\varepsilon}^2}{2} + \log\gamma 
\;\le\; \breve c'\log A_N - \tfrac{1}{2}\log A_N + O(\log\log A_N)
= -(\tfrac{1}{2} - \breve c')\log A_N + O(\log\log A_N).
\end{equation*}
Choosing $\breve c$ small enough, we obtain
\begin{equation*}
e^{m\cdot\bar{\varepsilon}^2/2 + \log\gamma} \;\lesssim\; A_N^{-c} \;\ll\; 1
\end{equation*}
for some $c > 0$. Thus the bracket in \eqref{eq:key-bracket} is $1 - o(1)$, and
\begin{equation*}
\E_g\!\left|\hat T - g(t_0)\right|^2 
\;\gtrsim\; \Delta_N^2 
\asymp \left(\frac{m n^2\varepsilon^2}{B_N^2}\right)^{-\frac{2\nu}{2\nu+2}}
= \left(\frac{m n^2\varepsilon^2}{\log A_N}\right)^{-\frac{2\nu}{2\nu+2}}.
\end{equation*}
Since $L_{m,N} = \log A_N$ when $m > \log A_N$, this matches the claimed rate.

\smallskip\noindent\textbf{The non-private regime:}
If instead $\Delta_N \asymp (N/\log A_N)^{-\nu/(2\nu+1)}$, we appeal to the standard (non-private) constrained-risk argument (e.g.\ Theorem~1 of \cite{cai2003pointwiseBesov}). For completeness, Lemma \ref{lem:classical-adaptation-lower-bound} with the same perturbation $g$ shows
\begin{equation*}
\underset{N\to\infty}{\lim\sup}\ e^{-C\frac{B_N}{N}}\,
\E_{f_0}\!\left(\frac{g(X_1)}{f_0(X_1)}\right)^2 \;<\;\infty,
\end{equation*}
from which the classical constrained-risk inequality yields
\begin{equation*}
\E_g\!\left|\hat T - g(t_0)\right|^2 \;\gtrsim\; \Delta_N^2
= \left(\frac{N}{\log A_N}\right)^{-\frac{2\nu}{2\nu+1}}.
\end{equation*}

\subsubsection{Proof of Corollary \ref{cor:pointwise-adaptation-lb}}\label{sec:proof-pointwise-cost-of-adaptation}

\begin{proof}[Proof of Corollary \ref{cor:pointwise-adaptation-lb}]
    Take $f_0$ to be the uniform density on $[0,1]$, which lies in $\cB^{\alpha_1}_{p_1,q_1}(R')$ for some $R'<R$ and in the interior of the $\cB^{\alpha_2}_{p_2,q_2}(R)$ ball, with $f_0(t_0) = 1$. Since $\hat{T}$ is rate optimal for the smoother Besov class $\cB^{\alpha_1}_{p_1,q_1}$, it achieves the minimax rate
    \begin{equation*}
    \E_{f_0} (\hat{T} - f_0(t_0))^2 \lesssim \left(\frac{N}{\log N}\right)^{-\frac{2\nu_1}{2\nu_1 + 1}} \vee \left(\frac{m n^2 \varepsilon^2}{\log^2 N}\right)^{-\frac{2\nu_1}{2\nu_1 + 2}},
    \end{equation*}
    where $\nu_1 := \alpha_1 - 1/p_1 > \nu_2 := \alpha_2 - 1/p_2 > 1/2$. Since $f_0$ lies in the interior of the less smooth class $\cB^{\alpha_2}_{p_2,q_2}(R)$, the above risk is strictly smaller than the minimax rate over this class,
    \begin{equation*}
    \left(\frac{N}{\log N}\right)^{-\frac{2\nu_2}{2\nu_2 + 1}} \vee \left(\frac{m n^2 \varepsilon^2}{\log^2 N}\right)^{-\frac{2\nu_2}{2\nu_2 + 2}}.
    \end{equation*}
    Thus, condition \eqref{eq:thm:adaptation-lower-bound-pointwise-assumption-clean} holds with $\nu = \nu_2$ and
    \begin{equation*}
    A_N^{-1} = \frac{\left(\frac{N}{\log N}\right)^{-\frac{2\nu_1}{2\nu_1 + 1}} \vee \left(\frac{m n^2 \varepsilon^2}{\log^2 N}\right)^{-\frac{2\nu_1}{2\nu_1 + 2}}}{\left(\frac{N}{\log N}\right)^{-\frac{2\nu_2}{2\nu_2 + 1}} \vee \left(\frac{m n^2 \varepsilon^2}{\log^2 N}\right)^{-\frac{2\nu_2}{2\nu_2 + 2}}}.
    \end{equation*}
    Since $\nu_1 > \nu_2$, it follows that $A_N \to \infty$ as $N \to \infty$. Moreover, under the assumption $\varepsilon \gtrsim (\sqrt{m} n)^{-\omega}$ for some $\omega \in [0,1)$, we have $A_N \gtrsim N^{\gamma}$ for some $\gamma > 0$ (with the exact $\gamma$ depending on whether the non-private or private rate dominates). The conclusion of the corollary then follows immediately from Theorem \ref{thm:adaptation-lower-bound-pointwise}.
    \end{proof}

\subsubsection{Auxiliary lemmas to Theorem \ref{thm:adaptation-lower-bound-pointwise}}

The following lemma extends Lemma \ref{lem:private-constrained-risk-delta}, its `$\delta=0$' version presented in the main article for general $\delta \geq 0$.

\begin{lemma}\label{lem:private-constrained-risk-delta}
Consider a model $\{P_f \, : \, f \in \Theta\}$ on $(\cX,\mathscr{X})$ indexed by a semi-metric space $(\Theta,\mathrm{d})$, and $f,g \in \Theta$ such that $\mathrm{d}(f, g) \geq \Delta$ for some $\Delta > 0$. Consider servers $j=1,\ldots,m$ each with i.i.d. samples $X_1^{(j)},\ldots,X_n^{(j)}$ with distribution $P_h$ for $h \in \Theta$.

If an $(\varepsilon,\delta)$-FDP estimation protocol $\hat{T}$ on the basis of $(X^{(j)}_i)_{i=1,\dots,n}^{j=1,\dots,m}$ satisfies 
\begin{align*}
     \E_f \, \mathrm{d} \left(\hat{T}, f \right)^2 &\leq \gamma^2 \Delta^2 \quad \text{ for some } \gamma > \sqrt{m \bar{\delta}},
\end{align*}
then
\begin{align*}
     \E_g \, \mathrm{d} \left(\hat{T}, g \right)^2 \geq \frac{\Delta^2}{4} \left[1 - \sqrt{5} \exp \left( m (\bar{\varepsilon}^2/2 \wedge \bar{\varepsilon}) + \log \gamma \right) - 4 m \bar{\delta} \right],
\end{align*}
where $\bar{\varepsilon} = 6 n \varepsilon \|P_f - P_g \|_{\text{TV}}$ and $\bar{\delta} = 4 e^{\bar{\varepsilon}} n \delta \|P_f - P_g \|_{\text{TV}}$.
\end{lemma}

\begin{proof}[Proof of Lemma \ref{lem:private-constrained-risk-delta}]
 Let $E$ denote the event $\{\mathrm{d}(\hat{T}, g) \geq c\Delta\}$. We have
\begin{equation*}
    \E_g \, \mathrm{d} \left(\hat{T}, g \right)^2 \geq (c\Delta)^2 \, \P_g\left( E \right).
\end{equation*}
We prove the lemma for general $c \in (0,1)$, plugging in $c = 1/2$ obtains the statement of the lemma. By combining the triangle inequality with the assumptions of the lemma and Markov's inequality, find that
\begin{equation}\label{eq:bound-Ec-Pf}
    \P_f\left( E^c \right) \leq \frac{\gamma^2}{(1-c)^2}.
\end{equation}
Let $T^{(j)}$ denote the transcript at server $j=1,\dots,m$. By Lemma \ref{lem:coupling-pointwise-lb}, there exists transcripts $\tilde{T}^{(j)}$ such that $\tilde{T}^{(j)}|[\tilde{T}^{(-j)},X_{n:1}^{(m:1)}]$ is in distribution equal to $\tilde{T}^{(j)}|[\tilde{T}^{(j-1:1)},X_{n:1}^{(j)}]$, $\| \P_h^{T^{(j)}|T^{(j-1:1)}} - \P_h^{\tilde{T}^{(j)}|\tilde{T}^{(j-1:1)}} \|_{\text{TV}} \leq \bar{\delta}$ for $h \in \{f,g\}$ and 
\begin{equation*} 
    \log \left( \frac{d\P_g^{\tilde{T}^{(j)}|\tilde{T}^{(j-1:1)}}}{d\P_f^{\tilde{T}^{(j)}|\tilde{T}^{(j-1:1)}}} \right) \in [ - \bar{\varepsilon} \, , \, \bar{\varepsilon} ].
\end{equation*}
Write $\tilde{T}$ for the estimator $\hat{T}$ with transcripts $T^{(j)}$ replaced by $\tilde{T}^{(j)}$. Let $\tilde{E}$ denote the event $\{\mathrm{d}(\tilde{T}, g) \geq c\Delta\}$. Following the conditional independence structure of the transcripts, we have that
\begin{equation*}
    \E_f \left[ \left( \frac{\P_g^{\tilde{T}}}{\P_f^{\tilde{T}}} \right)^2 \right] = \prod_{j=1}^m \E_f \left[ \left( \frac{\P_g^{\tilde{T}^{(j)}|\tilde{T}^{(j-1:1)}}}{\P_f^{\tilde{T}^{(j)}|\tilde{T}^{(j-1:1)}}} \right)^2 \right].
\end{equation*}
By the bound on the likelihood ratio, the latter expression is bounded by $\exp(2m\bar{\varepsilon})$.
By Lemma \ref{lem:chi-square-two-sided}, the right-hand side is bounded by $e^{m \bar{\varepsilon}^2}$.
By combining the coupling characterization of total variation with a standard data processing inequality and tensorization inequality (see Lemmas \ref{lem:coupling_and_TV}, \ref{lem:TV_distance_data_processing} and \ref{lem:TV_distance_product_measures}), we obtain
\begin{equation}\label{eq:switch-E-to-tilde-E}
    \P_f\left( E^c \right) \leq \P_f\left( \tilde{E}^c \right) + m \bar{\delta} \quad \text{ and } \quad \P_g\left( E^c \right) \leq \P_g\left( \tilde{E}^c \right) + m \bar{\delta}. 
\end{equation}
Combining the above with the Cauchy-Schwarz inequality, we find that 
\begin{align*}
    \P_g\left( E \right) = 1 - \P_g\left( E^c \right) \geq 1 - e^{m( \bar{\varepsilon}^2/2 \wedge \bar{\varepsilon})} \sqrt{\P_f( \tilde{E}^c ) } - m \bar{\delta}. 
\end{align*}
Using \eqref{eq:bound-Ec-Pf} and the fact that $m \bar{\delta} < \gamma^2$, the result follows from the inequality ${\P_f( \tilde{E}^c )} \leq 5 \gamma^2 $.
\end{proof}

The following lemmas are standard, we provide references for their proofs.

\begin{lemma}\label{lem:TV_distance_product_measures}
    Let $P = \bigotimes_{j=1}^m P_j$ and $Q = \bigotimes_{j=1}^m Q_j$ for probability measures $P_j,Q_j$ defined on a common measurable space $(\cX,\mathscr{X})$, with probability densities $p_j,q_j$ for $j=1,\dots m$. It holds that
    \begin{equation*}
    \| P - Q \|_{\mathrm{TV}} \leq \underset{j=1}{\overset{m}{\sum}} \|P_j - Q_j\|_{\mathrm{TV}}. 
    \end{equation*}
    \end{lemma}
    \begin{proof}
        See e.g. \cite{tsybakov2009introduction}, Chapter 2. 
    \end{proof}

    \begin{lemma}\label{lem:TV_distance_data_processing}
        Let $(\cX,\mathscr{X})$ and $(\cY,\mathscr{Y})$ be two measurable spaces and let $K:\mathscr{Y} \times \cX \to [0,1]$ be a Markov kernel. For any probability measures $P,Q$ defined on $\mathscr{X}$ it holds that
        \begin{equation*}
        \| P K - QK \|_{\mathrm{TV}} \leq \| P - Q \|_{\mathrm{TV}}.
        \end{equation*}
    \end{lemma}	
    \begin{proof}
        See e.g. \cite{lecam2000asymptotics}.
    \end{proof}

    \begin{lemma}\label{lem:coupling_and_TV}
        For any two probability measures $P$ and $Q$ on a measurable space $(\cX, \mathscr{X})$ with $\cX$ a Polish space and $\mathscr{X}$ its Borel sigma-algebra. There exists a coupling $\P^{X,\tilde{X}}$ such that
        \begin{equation*}
        \| P - Q \|_{\mathrm{TV}} = 2 \P^{X,\tilde{X}} \left( X \neq \tilde{X} \right).
        \end{equation*}
        \end{lemma}
        \begin{proof}
            See e.g.\ Section 8.3 in~\cite{thorisson2000coupling}.
        \end{proof}	

\begin{lemma}\label{lem:coupling-pointwise-lb}
    Let $T_f,T_g$ be $(\varepsilon,\delta)$-DP transcripts based on $n$ i.i.d. samples from a distributions $P_f$ and $P_g$, respectively, defined on the same sample space. Write $\bar{\varepsilon}:= 6 n \|P_f - P_g \|_{\text{TV}}$ and $\bar{\delta} := 4 e^{\bar{\varepsilon}} n \delta\|P_f - P_g \|_{\text{TV}}$. 
    
    Then, there exists a random variables $\tilde{T}_f,\tilde{T}_g$ such that $\| \P^{T_h} - \P^{\tilde{T}_h} \|_{\text{TV}} \leq \delta'$ for $h \in \{f,g\}$ and 
    \begin{equation*}
        D_{KL}\left(\P^{\tilde{T}_f}, \P^{\tilde{T}_g}\right) \leq \min \left\{ \bar{\varepsilon} \, , \, (\bar{\varepsilon})^2 \right\}.
    \end{equation*}
\end{lemma}
\begin{proof}
    By Lemma 6.1 in \cite{karwa2017finite}, we have that 
    \begin{equation*}
        \P_g\left( T \in A  \right) \leq e^{\bar{\varepsilon}} \P_f\left( T \in A \right) + \bar{\delta}.
    \end{equation*}
    The proof now follows by Lemma I.5 of \cite{cai2023private}.
\end{proof}

The following lemma can be seen as Property 1 of \cite{cuff2016privacyMI}, but derived for the Chi-square divergence instead of KL-divergence.

\begin{lemma}\label{lem:chi-square-two-sided}
    Let $P\ll Q$ with likelihood ratio $L=\frac{dP}{dQ}$. 
    Assume a two-sided bound
    \[
      e^{-\varepsilon} \;\le\; L \;\le\; e^{\varepsilon} 
      \qquad (Q\text{-a.s.}),\ \ \varepsilon\ge 0.
    \]
    Then
    \[
      1+\chi^2(P\|Q)=\E_Q[L^2]\;\le\;e^{\varepsilon}+e^{-\varepsilon}-1
      =2\cosh(\varepsilon)-1
      \;\le\;e^{\varepsilon^2}.
    \]
    \end{lemma}
    
    \begin{proof}
    Let $Y=\log L\in[-\varepsilon,\varepsilon]$. We want to maximize 
    $\E[e^{2Y}]$ subject to the constraint $\E[e^{Y}]=\E_Q[L]=1$. 
    By Hoeffding's reduction principle \cite{Hoeffding1963} for convex functionals, the maximum is achieved by a two--point distribution supported on $\{\pm\varepsilon\}$: say $Y=\varepsilon$ with probability $a$ and $Y=-\varepsilon$ with probability $1-a$. 
    The constraint
    \[
      a e^{\varepsilon} + (1-a)e^{-\varepsilon} = 1
    \]
    determines $a=(1-e^{-\varepsilon})/(e^{\varepsilon}-e^{-\varepsilon})$. 
    Then
    \[
      \E[e^{2Y}] = a e^{2\varepsilon}+(1-a)e^{-2\varepsilon}
      = e^{\varepsilon}+e^{-\varepsilon}-1.
    \]
    Finally,
    \[
      e^{\varepsilon}+e^{-\varepsilon}-1
      = 2\cosh(\varepsilon)-1
      \;\le\; e^{\varepsilon^2},
    \]
    since $\cosh x \le e^{x^2/2}$ for all $x$, and 
    $e^{\varepsilon^2}-(2e^{\varepsilon^2/2}-1)=(e^{\varepsilon^2/2}-1)^2\ge 0$. 
    \end{proof}

The following lemma was proven in \cite{cai2003pointwiseBesov}; we provide a proof here for completeness.
    
\begin{lemma}\label{lem:classical-adaptation-lower-bound}
Let $\Delta_N = \left(\frac{N}{B_N}\right)^{-\frac{\nu}{2\nu + 1}}$ for a sequence $1 \leq B_N \ll N$ and consider a compactly supported function $h$ such that $h(0) > 0$, $\|h\|_2^2 > 0$, $\int h \, dx = 0$ and define $g(x) = f_0(x) + a h(b(t_0 - x))$, for $a = c \Delta_N$. $b = \Delta_N^{- \frac{1}{\nu}}$ for $c > 0$. 

It holds that 
\begin{equation*}
    \underset{N \to \infty}{\lim \sup} \; e^{-C {B_N}/N} \E_{f_0} \left( \frac{g}{f_0}(X_1)  \right)^2 < \infty,
\end{equation*}
for some constant $C > 0$ depends on $c$, $\|h\|_2^2$ and $f_0(t_0)$.
\end{lemma}

\begin{proof}
Using that $\int f_0 =1$ and $\int h = 0$,
    \begin{align*}
         \E_{f_0} \left[\left(\frac{g( X_1)}{f_0( X_1)}\right)^2 \right]
        &= \E_{f_0} \left[\left(1+ \frac{a h^2(b(t_0 - X_1))}{f_0( X_1)}\right)^2 \right]= \left[1+ \E_{f_0} \frac{a^2 h^2(b(t_0 - X_1))}{f^2_0( X_1)} \right].
    \end{align*}
    Next, note that since $f_0(t_0)>0$ and $f_0$ is continuous, there exists a constant $c_0 > 0$ and $M > 0$ such that $f_0(t) \geq c_0$ for $t \in (t_0 - M, t_0 + M)$. As $\Delta_N \to 0$ implies that $b \to \infty$, we find that for $N$ large enough, the support of $h^2(b(t_0 - x))$ is contained in $(t_0 - M, t_0 + M)$. We have that
    \begin{align*}
        \E\frac{a^2 h^2(b(t_0 - X_1))}{f^2_0(X_1)} &= \int_{t_0 - b^{-1} M}^{t_0 + b^{-1} M} \frac{h^2(t)}{f_0(t)} dt \leq \frac{2 a^2 b^{-1} M \|h\|_2^2}{c_0} \leq \frac{2 B_N M \|h\|_2^2}{c_0 N},
    \end{align*}
     from which the result follows.
\end{proof}

\subsubsection{Construction of a super-efficient proportion estimator}\label{sec:appendix:DP-hodge}

We construct a super-efficient private proportion estimator displaying the super-efficiency phenomenon of Example \ref{example} as follows. The transcripts  
    \begin{equation*}
        T^{(j)} = \frac{1}{n} \sum_{i=1}^n Y^{(j)}_i + \frac{2}{n \varepsilon} W^{(j)} \quad \text{ with } W^{(j)} \iid \operatorname{Lap}(1) \text{ for } j = 1,\dots,m
    \end{equation*}
    satisfy $(\varepsilon,0)$-FDP (see e.g. \cite{dwork2006calibrating}). On the basis of these transcripts, one could compute a private version of the Hodge estimator 
    \begin{equation*}
        \hat{T} = \begin{cases}
            1/2 \quad \text{ if } | m^{-1} \sum_{j=1}^m T^{(j)} - 1/2 | \leq \frac{C \log (N)}{\sqrt{m n \varepsilon^2}} \\
            m^{-1} \sum_{j=1}^m T^{(j)} \quad \text{ otherwise.}
        \end{cases}
    \end{equation*}
    It is easy to see that $\hat{T}$ attains the CDP and LDP minimax rates for any fixed $p$ (see e.g. \cite{duchi2013local}):
    \begin{align*}
        \E_p | \hat{T} - p |^2 &\lesssim 1/n + (n^2 \varepsilon^2)^{-1} \quad \text{for } m = 1, \\ 
        \E_p | \hat{T} - p |^2 &\lesssim (m \varepsilon^2)^{-1} \quad \quad \quad \; \; \text{for } n=1.
    \end{align*}
    In particular, for $C > 0$ large enough, the estimator is super-efficient at $p = 1/2$: 
    \begin{equation*}
        \E_{1/2} | \hat{T} - 1/2 |^2 \lesssim N^{- c C}
    \end{equation*}
    for some constant $c > 0$.

    \begin{lemma}\label{lem:coupling-pointwise-lb}
        Let $T_f,T_g$ be $(\varepsilon,\delta)$-DP transcripts based on $n$ i.i.d. samples from a distributions $P_f$ and $P_g$, respectively, defined on the same sample space. Write $\bar{\varepsilon}:= 6 n \|P_f - P_g \|_{\text{TV}}$ and $\bar{\delta} := 4 e^{\bar{\varepsilon}} n \delta\|P_f - P_g \|_{\text{TV}}$. 
        
        Then, there exists a random variables $\tilde{T}_f,\tilde{T}_g$ such that $\| \P^{T_f} - \P^{\tilde{T}_f} \|_{\text{TV}} \leq \delta'$ and 
        \begin{equation*}
            \log \left( \frac{d\P^{\tilde{T}_f}}{d\P^{\tilde{T}_g}}\right) \in [- \bar{\varepsilon} \, , \, \bar{\varepsilon} ].
        \end{equation*}
        \end{lemma}
        \begin{proof}
        By Lemma 6.1 in \cite{karwa2017finite}, we have that 
        \begin{equation*}
            \P_g\left( T \in A  \right) \leq e^{\bar{\varepsilon}} \P_f\left( T \in A \right) + \bar{\delta}.
        \end{equation*}
        The proof now follows by Lemma I.5 of \cite{cai2023private}.
        \end{proof}

\section{Additional proofs and technical lemmas}\label{sec:technical_lemmas}

In this section, we provide the proofs of the lemmas stated in the main text.

\subsection{Besov spaces and wavelets}\label{sec:besov_space_lemmas}

In a slight abuse of notation, we shall denote the father wavelet by $\psi_{l_0  k} = \phi_{l_0 +1, k}$ and represent any function $f\in L_2[0,1]$ in the form
\begin{align}\label{eq:wavelet_transform}
f=\sum_{l=l_0'}^{\infty}\sum_{k=0}^{2^{l}-1} f_{lk} \psi_{lk},
\end{align}
for some $l_0' \leq l_0$. Next, we shall characterize Besov spaces for $0 < \alpha < A$ through the wavelet decomposition. Loosely speaking, Besov space $\cB^\alpha_{p,q}$ containts functions having $\alpha$ bounded derivatives in $L_p$-space, with $q$ giving a finer control of the degree of smoothness. We refer the reader to \cite{triebel1992theory} for a detailed description. Wavelet bases allow characterization of the Besov spaces, where $\alpha$, $p$ and $q$ are parameters that capture the decay rate of wavelet basis coefficients. 

Let us define the norms 
\begin{equation}\label{eq:besov_norm-equivalence}
\| f\|_{\cB^{\alpha}_{p,q}}^{wav} \asymp  \begin{cases}
    \left( \underset{l=l_0}{\overset{\infty}{\sum}} \left(2^{l(\alpha+1/2-1/p)} \left\| (f_{lk})_{k=0}^{2^l-1} \right\|_p\right)^{q}\right)^{1/q} &\text{ for } 1 \leq q < \infty, \\
    \; \underset{l\geq l_0}{\sup} \; 2^{l(\alpha+1/2-1/p)} \left\| (f_{lk})_{k=0}^{2^l-1} \right\|_p   &\text{ for } q = \infty,
\end{cases}
\end{equation}
for $\alpha\in(0,A)$, $1\leq q \leq \infty$, $1 \leq p \leq \infty$. The above definition of the Besov space and norm is equivalent to the one given in \eqref{eq:Besov_space_adaptation_lower_bound} (see e.g. Chapter 4 in \cite{gine_mathematical_2016}).

There are two crucial properties of wavelets that we shall highlight here, which are repeatedly used in the proof of the main theorems. 

\begin{lemma}\label{lemma:coefficient_decay}
Let $ f \in \cB^\alpha_{p,q}(R) $ with $ p \geq 2 $, $ 1\leq q\leq \infty $. Then, for every level $ l\ge0 $,
\begin{enumerate}
\item[(i)] for every index $ k=0,1,\dots,2^l-1 $, we have
    \begin{equation*}
    |f_{lk}| \leq C R\, 2^{-l(\alpha+1/2-1/p)},
    \end{equation*}
    for some constant $C>0$.
\item[(ii)] Moreover,
    \begin{equation*}
    \sum_{k=0}^{2^l-1} |f_{lk}|^2 \leq C R^2 \, 2^{-2 l \alpha},
    \end{equation*}
    where $C>0$ is a universal constant.
\end{enumerate}
\end{lemma}

\begin{proof}
Part (i) follows from the wavelet characterization of the Besov norm \eqref{eq:besov_norm-equivalence}: we have
    \begin{equation*}
    \left( \sum_{l=l_0}^\infty \left[2^{l(\alpha+1/2-1/p)} \left\| (f_{lk})_{k=0}^{2^l-1} \right\|_{\ell_p} \right]^q \right)^{1/q} \leq R.
    \end{equation*}
    which means that for any fixed level $ l $,
    \begin{equation*}
    2^{l(\alpha+1/2-1/p)} \left\| (f_{lk})_{k=0}^{2^l-1} \right\|_{\ell_p} \leq R.
    \end{equation*}
    The bound on the maximum follows immediately since
    \begin{equation*}
    \max_{0\leq k<2^l} |f_{lk}| \leq \left\| (f_{lk})_{k=0}^{2^l-1} \right\|_{\ell_p}.
    \end{equation*}

For part (ii), let $d=2^l$ denote the dimension at level $l$ and $v = (f_{lk})_{k=0}^{d-1} \in \mathbb{R}^d$. By H\"older's inequality (or the norm monotonicity for $p \geq 2$), we have
    \begin{equation*}
    \|v\|_2 \leq d^{1/2 - 1/p} \|v\|_p.
    \end{equation*}
    Squaring both sides yields
    \begin{equation*}
    \|v\|_2^2 \leq d^{1 - 2/p} \|v\|_p^2 \leq d^{1 - 2/p} \cdot R^2 \cdot 2^{-2l(\alpha + 1/2 - 1/p)},
    \end{equation*}
    where the second inequality uses the bound from part (i). The result follows by substituting $d = 2^l$ and the norm-equivalence mentioned above.
\end{proof}

\subsection{Auxiliary and technical lemmas}\label{sec:auxiliary_lemmas}

The following results are either technical lemmas or known, and are included for the sake of completeness.

\begin{lemma}\label{lem:sub-exponential-zeta-beta-product}
Consider for $j=1,\dots,m$, $i=1,\dots,b$ independent $\zeta^{(j)}_i \sim \text{Rad}(1/2)$ and $(\beta^{(j)}_i, 1 - \beta^{(j)}_i)$ Dirichlet $(b,N-b)$ random variables, writing $\zeta^{(j)} = (\zeta^{(j)}_1,\dots,\zeta^{(j)}_b)$ and $\beta^{(j)} = (\beta^{(j)}_1,\dots,\beta^{(j)}_b)$, $N \geq 5$. Consider component-wise product $X_j := \zeta^{(j)} \circ \beta^{(j)}$. 

Then, 
\begin{equation*}
    \P\left( \left| \frac{1}{m} \sum_{j=1}^m X_j \right| \geq \tau \right) \leq 2 \exp\left( - c m \min\left( \frac{N^2 \tau^2}{4}, \frac{\tau N}{2} \right) \right).
\end{equation*}
\end{lemma}
\begin{proof}
Let $v \in \R^b$ be of unit length and fix any $t \in \R$. We have
\begin{align*}
    \E e^{t \langle X_j, v \rangle} &= \E e^{t \sum_{i=1}^b v_i \zeta^{(j)}_i \beta^{(j)}_i} = \E \prod_{i=1}^b \E\left[ e^{t v_i \zeta^{(j)}_i \beta^{(j)}_i} | \beta^{(j)}_i \right] = \E \prod_{i=1}^b \cosh(t v_i \beta^{(j)}_i) \\
    &\leq \E \prod_{i=1}^b e^{(t v_i \beta^{(j)}_i)^2 / 2} = \E e^{t^2 \sum_{i=1}^b v_i^2 (\beta^{(j)}_i)^2 / 2}.
\end{align*}
By convexity of $x\mapsto e^{t x}$,
\begin{equation*}
e^{t\sum_{j=1}^b v_j^2 \beta_j^2}\le \sum_{j=1}^b v_j^2 e^{t\beta_j^2}.
\end{equation*}
Taking expectations and using exchangeability of the Dirichlet coordinates,
\begin{equation*}
\E e^{t\sum_{j=1}^b v_j^2 \beta_j^2} \le \sum_{j=1}^b v_j^2\,\E e^{t\beta_j^2}
= \E e^{t\beta_1^2},
\quad\text{where }\;\beta_1\sim\mathrm{Beta}(1,N-1).
\end{equation*}

Now bound $\E e^{t\beta_1^2}$. For $0\le t\le 2/N$ (hence $t\le1$ when $N\ge2$) and $u\in[0,1]$ we have $e^u\le 1+u+u^2$. With $u=t\beta_1^2$,
\begin{equation*}
\E e^{t\beta_1^2}\le 1 + t\,\E[\beta_1^2] + t^2\,\E[\beta_1^4].
\end{equation*}
For $\beta_1\sim\mathrm{Beta}(1,N-1)$ the moments are
\begin{equation*}
\E[\beta_1^2]=\frac{2}{N(N+1)},\qquad
\E[\beta_1^4]=\frac{24}{N(N+1)(N+2)(N+3)}.
\end{equation*}
Using $1+x\le e^{x}$,
\begin{equation*}
\E e^{t\beta_1^2}\le
\exp\!\left(\frac{2t}{N(N+1)}+\frac{24t^2}{N(N+1)(N+2)(N+3)}\right).
\end{equation*}
Since $t\le 2/N$,
\begin{equation*}
\frac{24t^2}{N(N+1)(N+2)(N+3)}
\le
\frac{48}{N^2(N+1)(N+2)(N+3)}\,t
\le \frac{1}{N^2}\,t,
\end{equation*}
and also $\frac{2}{N(N+1)}\le \frac{2}{N^2}$. Therefore
\begin{equation*}
\E e^{t\beta_1^2}\le \exp\!\left(\frac{3t}{N^2}\right).
\end{equation*}
Combining with the first step yields that $X_j$ is mean-zero and $(2/N)$-sub-exponential. The result now follows by Bernstein's inequality for sums of independent sub-exponential random variables (see e.g. Theorem 2.8.1 in \cite{vershynin2018hdp}).
\end{proof}

The next lemma is a standard tail-bound for Gamma centered distributions.

\begin{lemma}\label{lem:sub-exponentiality-parameter-gamma}
Let 
\begin{equation*}
D\sim \Gamma(\alpha,\theta), \quad X:=D-\E[D].
\end{equation*}
For $\tau \le 2 \alpha \theta$, it holds that
\begin{equation*}
\P \left( X \geq \tau \right) \leq \exp\left( - \frac{\tau^2}{8 \alpha \theta^2} \right).
\end{equation*}
\end{lemma}
\begin{proof}
We have $\E[D]=\alpha\theta$ and the MGF of $D$ for $0 \leq t<1/\theta$ is given by $\E\exp(tD)=(1-\theta t)^{-\alpha}$. Thus, the MGF of $X$ satisfies
\begin{equation*}
\ln\E\exp(tX)=-t\alpha\theta-\alpha\ln(1-\theta t).
\end{equation*}
Using the Taylor expansion $-\ln(1-\theta t)=\theta t+\frac{(\theta t)^2}{2}+\frac{(\theta t)^3}{3}+\cdots$, for $|\theta t|\le 1/2$ it holds that
\begin{equation*}
-\ln(1-\theta t)\le \theta t+ 2(\theta t)^2.
\end{equation*}
It follows that for $|t|\le 1/(2\theta)$,
\begin{equation*}
\E\exp(tX)\le \exp\Bigl(2\alpha\,\theta^2t^2\Bigr).
\end{equation*}
By a Chernoff bound, we have for $0 \leq s \leq 1 / (2t \theta)$ 
\begin{align*}
\P \left( X \geq \tau \right) \leq e^{-s\tau} \E e^{s \tau X} \leq e^{-s\tau + 2 s^2 \tau^2 \alpha \theta^2}.
\end{align*}
Setting $s = \tau / (4 \alpha \theta^2)$ yields the desired result and satisfies the condition $s\tau \leq 2 \alpha \theta$.
\end{proof}

\begin{lemma}\label{lem:indicator-layer-cake}
For any nonnegative random variable $Z$, we have that
\begin{equation*}
\E\Bigl[ Z^2\,\mathbf{1}\{Z\ge \tau\} \Bigr] 
=\tau^2\,\P\{Z\ge \tau\} + 2\int_{\tau}^\infty s\,\P\{Z\ge s\}\,ds.
\end{equation*}
\end{lemma}
\begin{proof}
Define
\begin{equation*}
X = Z^2\,\mathbf{1}\{Z\ge \tau\}.
\end{equation*}
Since $X\ge 0$, by the layer-cake representation we have
\begin{equation*}
\E[X] = \int_0^\infty \P\{X\ge t\}\,dt.
\end{equation*}
Note that the support of $X$ is contained in $\{0\}\cup [\tau^2,\infty)$. Therefore, we split the integral at $t=\tau^2$:
\begin{equation*}
\E[X] = \int_0^{\tau^2} \P\{X\ge t\}\,dt + \int_{\tau^2}^\infty \P\{X\ge t\}\,dt.
\end{equation*}

For $0\le t<\tau^2$:  
On this range, if $Z\ge \tau$ then $Z^2\ge \tau^2 > t$; hence,
\begin{equation*}
\{X\ge t\} = \{Z^2\,\mathbf{1}\{Z\ge \tau\} \ge t\} = \{Z\ge \tau\}.
\end{equation*}
Thus,
\begin{equation*}
\int_0^{\tau^2}\P\{X\ge t\}\,dt 
=\tau^2\,\P\{Z\ge \tau\}.
\end{equation*}

For $t\ge \tau^2$:  
On the event $\{Z\ge \tau\}$ we have $X = Z^2$; hence,
\begin{equation*}
\{X\ge t\} = \{Z^2 \ge t\}.
\end{equation*}
Consequently,
\begin{equation*}
\int_{\tau^2}^\infty \P\{X\ge t\}\,dt 
=\int_{\tau^2}^\infty \P\{Z^2\ge t\}\,dt.
\end{equation*}

Therefore, 
\begin{equation*}
\E\Bigl[ Z^2\,\mathbf{1}\{Z\ge \tau\} \Bigr] 
=\tau^2\,\P\{Z\ge \tau\} + \int_{\tau^2}^\infty \P\{Z^2\ge t\}\,dt.
\end{equation*}

Performing the change of variable $t=s^2$ (with $s\ge \tau$), we have
\begin{equation*}
\int_{\tau^2}^\infty \P\{Z^2\ge t\}\,dt 
=\int_{\tau}^\infty \P\{Z^2\ge s^2\}\,2s\,ds.
\end{equation*}
Since $Z$ is nonnegative, $\{Z^2\ge s^2\}=\{Z\ge s\}$. Therefore,
\begin{equation*}
\int_{\tau^2}^\infty \P\{Z^2\ge t\}\,dt = 2\int_{\tau}^\infty s\,\P\{Z\ge s\}\,ds.
\end{equation*}
This completes the proof.
\end{proof}

The following technical lemma is in conjunction with the previous lemma used to derive our tail bounds.

\begin{lemma}\label{lem:exponential-minimum}
Consider $a, b, c, d > 0$. It holds that
\begin{equation*}
    \int_{d}^\infty u e^{- c \min(a u^2, b u)} \, du \leq \left( \frac{1}{c \min(a, b/d)} + \frac{1}{c^2 b^2} \right) e^{ - c \min(a d^2, b d)}.
\end{equation*}
\end{lemma}

\begin{proof}
Let $u_0 = b/a > 0$ and $m = \min(a d^2, b d)$. Note that $m = a d^2$ if $d \leq u_0$ and $m = b d$ if $d > u_0$. Also, $\min(a, b/d) = a$ if $d \leq u_0$ and $\min(a, b/d) = b/d$ if $d > u_0$.

\textbf{Case 1: $d > u_0$.} Here, $\min(a u^2, b u) = b u$ for all $u \geq d > u_0$, so
\begin{equation*}
\int_d^\infty u \, e^{- c b u} \, du = \left[ - \frac{u}{c b} e^{- c b u} - \frac{1}{c^2 b^2} e^{- c b u} \right]_d^\infty = \left( \frac{d}{c b} + \frac{1}{c^2 b^2} \right) e^{- c b d}.
\end{equation*}
Since $m = b d$ and $\min(a, b/d) = b/d$, the right-hand side is
\begin{equation*}
\left( \frac{1}{c (b/d)} + \frac{1}{c^2 b^2} \right) e^{- c m} = \left( \frac{d}{c b} + \frac{1}{c^2 b^2} \right) e^{- c b d},
\end{equation*}
which matches exactly.

\textbf{Case 2: $d \leq u_0$.} Split the integral at $u_0$:
\begin{equation*}
\int_d^\infty u \, e^{- c \min(a u^2, b u)} \, du = \int_d^{u_0} u \, e^{- c a u^2} \, du + \int_{u_0}^\infty u \, e^{- c b u} \, du.
\end{equation*}
For the first integral,
\begin{equation*}
\int_d^{u_0} u \, e^{- c a u^2} \, du = \left[ -\frac{1}{2 c a} e^{- c a u^2} \right]_d^{u_0} = \frac{1}{2 c a} \left( e^{- c a d^2} - e^{- c a u_0^2} \right).
\end{equation*}
For the second integral, since $u_0 = b/a$ and $c b u_0 = c a u_0^2$,
\begin{equation*}
\int_{u_0}^\infty u \, e^{- c b u} \, du = \left[ -\frac{u}{c b} e^{- c b u} - \frac{1}{c^2 b^2} e^{- c b u} \right]_{u_0}^\infty = \left( \frac{u_0}{c b} + \frac{1}{c^2 b^2} \right) e^{- c b u_0} = \left( \frac{1}{c a} + \frac{1}{c^2 b^2} \right) e^{- c a u_0^2}.
\end{equation*}
Thus, the total integral is
\begin{equation*}
\frac{1}{2 c a} \left( e^{- c a d^2} - e^{- c a u_0^2} \right) + \left( \frac{1}{c a} + \frac{1}{c^2 b^2} \right) e^{- c a u_0^2}.
\end{equation*}
Since $e^{- c a u_0^2} \leq e^{- c a d^2}$ (with equality only if $d = u_0$), we bound
\begin{equation*}
\frac{1}{2 c a} e^{- c a d^2} - \frac{1}{2 c a} e^{- c a u_0^2} + \left( \frac{1}{c a} + \frac{1}{c^2 b^2} \right) e^{- c a u_0^2} \leq \frac{1}{2 c a} e^{- c a d^2} + \left( \frac{1}{c a} + \frac{1}{c^2 b^2} \right) e^{- c a d^2}.
\end{equation*}
This simplifies to
\begin{equation*}
\left( \frac{3}{2 c a} + \frac{1}{c^2 b^2} \right) e^{- c a d^2} \leq \left( \frac{1}{c a} + \frac{1}{c^2 b^2} \right) e^{- c a d^2},
\end{equation*}
since $\frac{3}{2 c a} \leq \frac{1}{c a}$. Here, $m = a d^2$ and $\min(a, b/d) = a$, so the bound is
\begin{equation*}
\left( \frac{1}{c \min(a, b/d)} + \frac{1}{c^2 b^2} \right) e^{- c m},
\end{equation*}
as required. In both cases, the integral is bounded by the stated expression, completing the proof.
\end{proof}

    The following lemma provides the tail bound for the non-privacy related noise.

    \begin{lemma}\label{lem:non-private-block-thresholding-tailbound}
        Consider $l \in \{ l_0,\dots,L^* \}$ and $S \subset \{1,\dots,2^l\}$ of size $|S| := b_l \leq \lceil \log N \rceil$, and let $B_l = \{ (l,k) : k \in S \}$.
    
        For all $c_0 \geq 2$, there exists a constant $c_1 > 0$ such that
        \begin{equation*}
            \E_f \|(\hat{\psi} - f)_{B_l}\|_2^2 \,
            \mathbbm{1}\!\left\{ \|(\hat{\psi} - f)_{B_l}\|_2 > \sqrt{\tfrac{c_0 \log N}{N}} \right\}
            \;\lesssim\; \frac{b_l}{N} \, e^{- c_1 \log N}.
        \end{equation*}
    \end{lemma}
    
    \begin{remark}\label{rmk:non-private-block-thresholding-tailbound}
        For the constant $c_1 = 1$, it suffices to take $c_0 = 5$. 
    \end{remark}
    
    \begin{proof}
        Define $ E = (\bar{\psi} - f)_{B_l} = (\bar{\psi}_{lk} - f_{lk})_{k \in S} $ and let $ W = \|E\|_2^2 $. Then
        \begin{equation*}
        \E_f \!\left[ W \, \mathbbm{1}\!\left\{ W > \tfrac{c_0 \log N}{N} \right\} \right].
        \end{equation*}
        By Lemma~\ref{lem:indicator-layer-cake},
        \begin{equation}\label{eq:layer-cake-block}
            \E_f [W \mathbbm{1}\{ W > \tau \}] = \tau \P(W > \tau) + \int_\tau^\infty s \, \P(W > s) \, ds,
        \end{equation}
        where $ \tau = c_0 \tfrac{\log N}{N} $.
    
        Each coordinate satisfies $ U_k := \bar{\psi}_{lk} - f_{lk} $, with 
        \begin{equation*}
        \E_f U_k^2 = N^{-1}\Var_f(\psi_{lk}(X_i)) \;\lesssim\; \frac{2^l}{N}.
        \end{equation*}
        Thus
        \begin{equation*}
        \E_f W \;\lesssim\; \frac{b_l 2^l}{N}.
        \end{equation*}
    
        The random variable $W$ is a sum of $b_l$ independent, sub-exponential terms (variance proxy $\sigma^2 \lesssim b_l 2^l / N$, envelope $M \asymp 2^l/N$). Hence Bernstein’s inequality yields
        \begin{equation*}
        \P(W > s) \;\leq\; \exp\!\left( - \frac{N s^2}{2 b_l 2^l + (2/3) s \cdot c_\psi^2 2^l} \right).
        \end{equation*}
    
        Applying this bound in~\eqref{eq:layer-cake-block}, with $\tau = c_0 \tfrac{\log N}{N}$ and using that $b_l \leq \log N$, $2^l \leq N$, and $L^* \asymp \log_2 N$, one obtains
        \begin{equation*}
        \tau \P(W > \tau) \;\lesssim\; \frac{b_l}{N} e^{-c_1 \log N}, 
            \qquad 
        \int_\tau^\infty s \P(W > s) \, ds \;\lesssim\; \frac{1}{N} e^{-c_1 \log N}.
        \end{equation*}
    
        Combining the two contributions establishes the claim.
    \end{proof}


    \putbib[references]
\end{bibunit}

\end{document}